\documentclass{amsart}

\usepackage{etoolbox}

\makeatletter
\let\old@tocline\@tocline
\let\section@tocline\@tocline
\newcommand{\subsection@dotsep}{4.5}
\newcommand{\subsubsection@dotsep}{4.5}
\patchcmd{\@tocline}
{\hfil}
{\nobreak
	\leaders\hbox{$\m@th
		\mkern \subsection@dotsep mu\hbox{.}\mkern \subsection@dotsep mu$}\hfill
	\nobreak}{}{}
\let\subsection@tocline\@tocline
\let\@tocline\old@tocline

\patchcmd{\@tocline}
{\hfil}
{\nobreak
	\leaders\hbox{$\m@th
		\mkern \subsubsection@dotsep mu\hbox{.}\mkern \subsubsection@dotsep mu$}\hfill
	\nobreak}{}{}
\let\subsubsection@tocline\@tocline
\let\@tocline\old@tocline

\let\old@l@subsection\l@subsection
\let\old@l@subsubsection\l@subsubsection

\def\@tocwriteb#1#2#3{%
	\begingroup
	\@xp\def\csname #2@tocline\endcsname##1##2##3##4##5##6{%
		\ifnum##1>\c@tocdepth
		\else \sbox\z@{##5\let\indentlabel\@tochangmeasure##6}\fi}%
	\csname l@#2\endcsname{#1{\csname#2name\endcsname}{\@secnumber}{}}%
	\endgroup
	\addcontentsline{toc}{#2}%
	{\protect#1{\csname#2name\endcsname}{\@secnumber}{#3}}}%

\newlength{\@tocsectionindent}
\newlength{\@tocsubsectionindent}
\newlength{\@tocsubsubsectionindent}
\newlength{\@tocsectionnumwidth}
\newlength{\@tocsubsectionnumwidth}
\newlength{\@tocsubsubsectionnumwidth}
\newcommand{\settocsectionnumwidth}[1]{\setlength{\@tocsectionnumwidth}{#1}}
\newcommand{\settocsubsectionnumwidth}[1]{\setlength{\@tocsubsectionnumwidth}{#1}}
\newcommand{\settocsubsubsectionnumwidth}[1]{\setlength{\@tocsubsubsectionnumwidth}{#1}}
\newcommand{\settocsectionindent}[1]{\setlength{\@tocsectionindent}{#1}}
\newcommand{\settocsubsectionindent}[1]{\setlength{\@tocsubsectionindent}{#1}}
\newcommand{\settocsubsubsectionindent}[1]{\setlength{\@tocsubsubsectionindent}{#1}}

\renewcommand{\l@section}{\section@tocline{1}{\@tocsectionvskip}{\@tocsectionindent}{}{\@tocsectionformat}}%
\renewcommand{\l@subsection}{\subsection@tocline{2}{\@tocsubsectionvskip}{\@tocsubsectionindent}{}{\@tocsubsectionformat}}%
\renewcommand{\l@subsubsection}{\subsubsection@tocline{3}{\@tocsubsubsectionvskip}{\@tocsubsubsectionindent}{}{\@tocsubsubsectionformat}}%
\newcommand{\@tocsectionformat}{}
\newcommand{\@tocsubsectionformat}{}
\newcommand{\@tocsubsubsectionformat}{}
\expandafter\def\csname toc@1format\endcsname{\@tocsectionformat}
\expandafter\def\csname toc@2format\endcsname{\@tocsubsectionformat}
\expandafter\def\csname toc@3format\endcsname{\@tocsubsubsectionformat}
\newcommand{\settocsectionformat}[1]{\renewcommand{\@tocsectionformat}{#1}}
\newcommand{\settocsubsectionformat}[1]{\renewcommand{\@tocsubsectionformat}{#1}}
\newcommand{\settocsubsubsectionformat}[1]{\renewcommand{\@tocsubsubsectionformat}{#1}}
\newlength{\@tocsectionvskip}
\newcommand{\settocsectionvskip}[1]{\setlength{\@tocsectionvskip}{#1}}
\newlength{\@tocsubsectionvskip}
\newcommand{\settocsubsectionvskip}[1]{\setlength{\@tocsubsectionvskip}{#1}}
\newlength{\@tocsubsubsectionvskip}
\newcommand{\settocsubsubsectionvskip}[1]{\setlength{\@tocsubsubsectionvskip}{#1}}

\patchcmd{\tocsection}{\indentlabel}{\makebox[\@tocsectionnumwidth][l]}{}{}
\patchcmd{\tocsubsection}{\indentlabel}{\makebox[\@tocsubsectionnumwidth][l]}{}{}
\patchcmd{\tocsubsubsection}{\indentlabel}{\makebox[\@tocsubsubsectionnumwidth][l]}{}{}

\newcommand{\@sectypepnumformat}{}
\renewcommand{\contentsline}[1]{%
	\expandafter\let\expandafter\@sectypepnumformat\csname @toc#1pnumformat\endcsname%
	\csname l@#1\endcsname}
\newcommand{\@tocsectionpnumformat}{}
\newcommand{\@tocsubsectionpnumformat}{}
\newcommand{\@tocsubsubsectionpnumformat}{}
\newcommand{\setsectionpnumformat}[1]{\renewcommand{\@tocsectionpnumformat}{#1}}
\newcommand{\setsubsectionpnumformat}[1]{\renewcommand{\@tocsubsectionpnumformat}{#1}}
\newcommand{\setsubsubsectionpnumformat}[1]{\renewcommand{\@tocsubsubsectionpnumformat}{#1}}
\renewcommand{\@tocpagenum}[1]{%
	\hfill {\mdseries\@sectypepnumformat #1}}

\let\oldappendix\appendix
\renewcommand{\appendix}{%
	\leavevmode\oldappendix%
	\addtocontents{toc}{%
		\protect\settowidth{\protect\@tocsectionnumwidth}{\protect\@tocsectionformat\sectionname\space}%
		\protect\addtolength{\protect\@tocsectionnumwidth}{2em}}%
}
\makeatother

\makeatletter
\settocsectionnumwidth{2em}
\settocsubsectionnumwidth{2.5em}
\settocsubsubsectionnumwidth{3em}
\settocsectionindent{1pc}%
\settocsubsectionindent{\dimexpr\@tocsectionindent+\@tocsectionnumwidth}%
\settocsubsubsectionindent{\dimexpr\@tocsubsectionindent+\@tocsubsectionnumwidth}%
\makeatother

\settocsectionvskip{10pt}
\settocsubsectionvskip{0pt}
\settocsubsubsectionvskip{0pt}

\settocsectionformat{\bfseries}
\settocsubsectionformat{\mdseries}
\settocsubsubsectionformat{\mdseries}
\setsectionpnumformat{\bfseries}
\setsubsectionpnumformat{\mdseries}
\setsubsubsectionpnumformat{\mdseries}

\let\oldtableofcontents\tableofcontents
\renewcommand{\tableofcontents}{%
	\vspace*{-\linespacing}
	\oldtableofcontents}

\usepackage{soul}
\usepackage{amssymb,latexsym,amsmath,amsthm,graphicx,bm}
\usepackage{sidecap,multicol}
\usepackage[all,cmtip]{xy}
\usepackage[pdfencoding=auto,hidelinks]{hyperref}
\usepackage[capitalize,nameinlink]{cleveref}

\usepackage[shortlabels]{enumitem}

\usepackage{tikz}
\usetikzlibrary{arrows.meta,decorations.pathmorphing,decorations.markings,backgrounds,positioning,fit,shapes}

\newcommand{\mathsc}[1]{{\normalfont\textsc{#1}}}

\newcommand{\ZZ}{\mathbb{Z}}
\newcommand{\QQ}{\mathbb{Q}}
\newcommand{\RR}{\mathbb{R}}
\newcommand{\PP}{\mathbb{P}}

\newcommand{\CC}{\mathbb{C}}

\newcommand{\kk}{\Bbbk}

\newcommand{\cA}{\mathcal{A}}
\newcommand{\cB}{\mathcal{B}}
\newcommand{\cC}{\mathcal{C}}
\newcommand{\cD}{\mathcal{D}}

\newcommand{\cF}{\mathcal{F}}
\newcommand{\cG}{\mathcal{G}}

\newcommand{\cL}{\mathcal{L}}
\newcommand{\cM}{\mathcal{M}}

\newcommand{\cO}{\mathcal{O}}

\newcommand{\cQ}{\mathcal{Q}}

\newcommand{\cS}{\mathcal{S}}
\newcommand{\cT}{\mathcal{T}}

\newcommand{\cV}{\mathcal{V}}
\newcommand{\cW}{\mathcal{W}}

\newcommand{\mh}{}
\mathchardef\mh="2D

\DeclareMathOperator{\coker}{coker}

\DeclareMathOperator{\Ext}{Ext}
\DeclareMathOperator{\Hom}{Hom}
\DeclareMathOperator{\RHom}{RHom}

\DeclareMathOperator{\End}{End}
\DeclareMathOperator{\Cone}{Cone}
\DeclareMathOperator{\mmod}{mod}

\DeclareMathOperator{\Ob}{Ob}
\DeclareMathOperator{\Mod}{Mod}
\DeclareMathOperator{\Perf}{Perf}

\DeclareMathOperator{\Gr}{Gr}

\newcommand{\Vect}{\mathrm{Vect}}

\newcommand{\del}{\partial}
\newcommand{\circnec}{\underset{\mathrm{nec}}{\circ}}
\newcommand{\nec}{\mathrm{nec}}
\newcommand{\cyc}{\mathrm{cyc}}
\newcommand{\cl}{\mathrm{cl}}
\newcommand{\ev}{\mathrm{ev}}

\newcommand{\id}{\mathrm{id}}
\newcommand{\Sgn}{\mathrm{Sgn}}
\newcommand{\Bord}{\mathrm{Bord}}
\DeclareMathOperator{\sheafHom}{\mathcal{H}\kern -1.2pt \mathit{om}}

\newtheorem{theorem}{Theorem}
\newtheorem{proposition}[theorem]{Proposition}
\newtheorem{corollary}[theorem]{Corollary}
\newtheorem{lemma}[theorem]{Lemma}

\newtheorem*{theorem*}{Theorem}
\theoremstyle{definition}
\newtheorem{definition}{Definition}

\theoremstyle{remark}
\newtheorem*{remark}{Remark}
\newtheorem*{example}{Example}

\usepackage[backend=biber, bibencoding=utf8, style=alphabetic]{biblatex}
\title[Pre-CY algebras and TQFT]{Pre-Calabi-Yau algebras and topological quantum field theories}
\author{Maxim Kontsevich, Alex Takeda and Yiannis Vlassopoulos}

\begin{document}
\tikzset{>={Stealth[scale=1.2]}}
\tikzset{->-/.style={decoration={
			markings,
			mark=at position #1 with {\arrow{>}}},postaction={decorate}}}
\tikzset{-w-/.style={decoration={
			markings,
			mark=at position #1 with {\arrow{Stealth[fill=white,scale=1.4]}}},postaction={decorate}}}
\tikzset{->-/.default=0.65}
\tikzset{-w-/.default=0.65}
\tikzstyle{bullet}=[circle,fill=black,inner sep=0.5mm]
\tikzstyle{circ}=[circle,draw=black,fill=white,inner sep=0.5mm]
\tikzstyle{vertex}=[circle,draw=black,thick,inner sep=0.5mm]
\tikzset{darrow/.style={double distance = 4pt,>={Implies},->},
	darrowthin/.style={double equal sign distance,>={Implies},->},
	tarrow/.style={-,preaction={draw,darrow}},
	qarrow/.style={preaction={draw,darrow,shorten >=0pt},shorten >=1pt,-,double,double
		distance=0.2pt}}

\begin{abstract}
	We introduce a notion generalizing Calabi-Yau structures on A-infinity algebras and categories, which we call pre-Calabi-Yau structures. This notion does not need either one of the finiteness conditions (smoothness or compactness) which are required for Calabi-Yau structures to exist. In terms of noncommutative geometry, a pre-CY structure is as a polyvector field satisfying an integrability condition with respect to a noncommutative analogue of the Schouten-Nijenhuis bracket. We show that a pre-CY structure defines an action of a certain PROP of chains on decorated Riemann surfaces. In the language of the cobordism perspective on TQFTs, this gives a partially defined extended 2-dimensional TQFT, whose 2-dimensional cobordisms are generated only by handles of index one. We present some examples of pre-CY structures appearing naturally in geometric and topological contexts. 
\end{abstract}

\maketitle
\tableofcontents

\section{Introduction}
The complex of Hochschild chains $C_*(A)$ of a compact Calabi-Yau algebra $A$ has the structure of an algebra over a differential graded (dg) \textsc{prop}, which encodes the topology of certain moduli spaces of topological surfaces. This structure has been studied from many perspectives \cite{kontsevich2009notes, costello2007topological,costello2005gromovwitten,wahl2016hochschild}, and is given by a collection of operations $C_*(A)^{\otimes m} \to C_*(A)^{\otimes n}$ for $m \ge 1, n \ge 0$; the spaces of such operations are parametrized by chains on the moduli space $\cM_{g,\vec{m},\vec{n}}$ of Riemann surfaces, decorated with $m \ge 1$ incoming and $n \ge 0$ outgoing marked points and choices of a real tangent direction on each marked point.

Over the past decade, there have been various developments in the study of Calabi-Yau structures, so we recapitulate the general outline of this theory, in the context of $A_\infty$-algebras. An $A_\infty$-algebra $(A,\mu=\sum_{i\geq 1} \mu^i)$ is called \emph{compact} if its cohomology $H^*(A,\mu^1)$ is finite dimensional. A Calabi-Yau structure of dimension $d$ on a compact $A_\infty$-algebra $A$ is a class $\omega \in \Hom_\kk(CC_*(A),\kk)[-d]$, that is, in the dual of the cyclic complex of $A$, whose projection to
\[ \Hom(C_*(A),\kk[-d]) \cong \RHom_{A\mh\mmod\mh A}(A[d],A^\vee) \]
defines a quasi-isomorphism of $A$-bimodules between a shift of the diagonal bimodule $A[d]$ and the linear dual $A^\vee$. This has been called a \emph{compact, proper} or \emph{right} Calabi-Yau structure in the existing literature.

The action of the \textsc{prop} of chains on moduli spaces of decorated surfaces has been interpreted by Lurie in the context of the cobordism hypothesis; using a certain unfolding construction \cite{lurie2009classification}, one shows that a compact CY structure produces a `non-compact' fully extended 2d \textsc{tqft}, whose 2-dimensional cobordisms are required to have at least one input. In the case of cyclic $A_\infty$-algebras, that is, finite-dimensional $A_\infty$-algebras with a cyclically compatible pairing, one can concretely write down formulas for these operations.

When $A$ is homologically smooth \cite{kontsevich2009notes}, there is a dual notion of Calabi-Yau structure, which has appeared in the literature by the name of \emph{smooth} or \emph{left Calabi-Yau structure}. It is given by a class $\omega \in CC^-_*(A)[-d]$ in the negative cyclic homology of $A$, whose projection to
\[ C_*(A)[-d] \cong \RHom_{A\mh\mmod\mh A}(A^!,A[-d]) \]
gives a quasi-isomorphism between the \emph{left} or \emph{bimodule dual} $A^! = \RHom_{A\mh\mmod\mh A}(A,A^e)$ and the shifted diagonal bimodule $A[d]$. The idea that one should consider such quasi-isomorphisms appeared initially in \cite{ginzburg2006calabiyau}, where it was not required that this quasi-isomorphism come from a class negative cyclic homology; this weaker has been called a `weak CY structure' or `Ginzburg CY structure' in the literature.

Each one of these types of structures, compact and smooth CY, requires a certain finiteness condition in order to exist: a compact CY structure on $A$ can only exist if the cohomology of $(A,\mu^1)$ is finite-dimensional, and a smooth CY structure, only if $A$ is homologically smooth, that is, if the diagonal bimodule $A$ has a finite-length resolution. In many applications, it is necessary to go beyond the finite case, for instance in string topology, certain types of Fukaya categories and categories of coherent sheaves appearing in homological mirror symmetry.

The purpose of this paper is to describe as explicitly as possible a type of structure that generalizes compact and smooth CY structures. We call this a \emph{pre-Calabi-Yau structure}; it can exist on an $A_\infty$-category $\cA$ that may be neither compact nor smooth, and defines a certain type of partial 2d \textsc{tqft} structure. Before explaining what are these structures, let us mention our main result following from their definition of this type of structure.
\begin{theorem}\label{thm:MainPROP}
	The Hochschild chain complex $C_*(\cA)$ of a pre-CY category $\cA$ of dimension $d$ has the action of the \textsc{prop} $Q^d$, containing operations
	\[ C_*(\cA)^{\otimes m} \otimes Q^d \to C_*(\cA)^{\otimes n} \]
	for $m \ge 1, n \ge 1$, that is, operations \emph{with at least one input and at least one output}. The \textsc{prop}s $Q^d$ calculate the cohomologies of the moduli spaces $\cM_{g,\vec{m},\vec{n}}$, with coefficients in powers $\cL^d$ of a certain rank one local system.
	
	Moreover, there is an open-closed extension of this \textsc{prop} (as a colored \textsc{prop}) acting on both $C_*(\cA)$ and the morphism spaces $\cA(X,Y)$ of the $A_\infty$-category $\cA$.
\end{theorem}

The proof of \cref{thm:MainPROP} relies on a new cell complex describing the spaces $\cM_{g,\vec{m},\vec{n}}$. This cell complex is obtained by using a modified form of uniformization of surfaces by Strebel differentials; we use quadratic differentials with higher order poles. We rely on the description of moduli spaces of meromorphic differentials given by \cite{gupta2016quadratic,gupta2019meromorphic}, itself a generalization of the Hubbard-Masur theorem \cite{hubbard1979quadratic}. The use of higher-order poles allows us to easily describe an open-closed extension of this \textsc{prop}.

The classical theory of Strebel differential is related to the topology of moduli spaces of Riemann surfaces as explained in \cite{kontsevich1992intersection}. A Riemann surface with a Strebel differential determines the data of a metric ribbon graph; this gives a cell decomposition of the corresponding moduli space into cells labeled by topological types of ribbon graphs. In our case, we will obtain instead \emph{acyclic marked ribbon quivers}, i.e. ribbon graphs endowed with directed edges in an acyclic manner, and with some markings. Each such ribbon quiver gives a cell in the open-closed moduli space; this cell decomposition is similar to, but finer than, the cell decomposition given by `black-and-white graphs' appearing in e.g. \cite{wahl2016hochschild,ega2015comparing}. The combinatorial data of the quiver gives the action of this colored open-closed \textsc{prop} on the morphism spaces of a pre-CY category $A$ and its Hochschild complex $C_*(A)$; this action reduces to the black-and-white graph PROP action discussed in \emph{op.cit.} when $A$ is a cyclic $A_\infty$-algebra.

We can motivate the definition of pre-Calabi-Yau structures from the general principle that a noncommutative version of some type of geometric structure $S$ should be a structure on an algebra $A$ that induces a structure of type $S$ on the representation spaces $\mathrm{Rep}(A,n)$ of $n$-dimensional representations of $A$, for all $n$. In the framework of this principle, we propose that \emph{a pre-Calabi-Yau structure is a  noncommutative shifted Poisson structure}. This has been explored elsewhere in the literature: in \cite{yeung2018precalabiyau}, it is proven that for a (nonpositively-graded, dg) algebra $A$, pre-CY structures on $A$ induce shifted Poisson structures on the representation spaces $\mathrm{Rep}(A,n)$.

\subsection{Definition and some examples}
Let us now describe this structure, first for an $A_\infty$-algebra. For any graded vector space $A$, we define its space of `order $k$ higher Hochschild cochains'
\[ C^*_{(k)}(A) = \prod_{n_1,\dots,n_k \geq 0} \Hom(\bigotimes_{i=1}^k A[1]^{\otimes{n_i}}, A^{\otimes k}) \]

\[\begin{tikzpicture}
\node [vertex] (phi) at (0,0) {$\phi$};
\node (a11) at (.7,1) {};
\node (a12) at (1.2,0) {};
\node (a21) at (0,-1.3) {};
\node (a31) at (-1.4,-0.2) {};
\node (a32) at (-1.3,0.6) {};
\node (a33) at (-0.7,1) {};
\node (b1) at (0,1.5) {};
\node (b2) at (1.5,-0.9) {};
\node (b3) at (-1.5,-0.9) {};

\draw [->-] (phi) to (b1);
\draw [->-] (phi) to (b2);
\draw [->-] (phi) to (b3);
\draw [->, shorten <=-2mm, shorten >=2mm] (a11) to (phi);
\draw [->, shorten <=-2mm, shorten >=2mm] (a12) to (phi);
\draw [->, shorten <=-1mm, shorten >=2mm] (a21) to (phi);
\draw [->, shorten <=-2mm, shorten >=2mm] (a31) to (phi);
\draw [->, shorten <=-1mm, shorten >=2mm] (a32) to (phi);
\draw [->, shorten <=-2mm, shorten >=2mm] (a33) to (phi);
\end{tikzpicture}\]
Each order $k$ higher cochain $\phi$ can be visualized as a vertex with $k$ outgoing arrows, and $n_i$ incoming arrows on each angle, as in the figure above. For any choice of integer $d$, we define the space of `order $k$ dimension $d$ cyclic Hochschild cochains'
\[ C_{(k, d)}^*(A) := (C_{(k)}^*(A))^{(\ZZ_k,d)}[(d-2)(k-1)] \]
where $(\ZZ_k,d)$ denotes invariants under a certain action of the cyclic group rotating the vertex, with signs depending on the parity of $d$. We endow the space
\[ C^*_{[d]}(A) := \prod_k C^*_{(k,d)}(A) \]
with a binary operation we call the \emph{necklace product} $\circnec$. Its associated \emph{necklace bracket} $[-,-]_\nec$ gives $C^*_{[d]}(A)[1]$ the structure of a dg Lie algebra.
\begin{definition}\label{def:main}
	A pre-Calabi-Yau structure of dimension $d$ on the graded vector space $A$ is a solution $m = \sum_k m_{(k)}$ to the Maurer-Cartan equation $m \circnec m = 0$, of degree one in the dg Lie algebra $C^*_{[d]}(A)[1]$.
\end{definition}

Even though our purpose for defining pre-CY structures is for dealing with $A_\infty$-algebras/categories $A$ that may be neither smooth nor proper, it is useful for purposes of intuition to describe them in the simpler case where $A$ is finite-dimensional: a pre-Calabi Yau structure of dimension $d$ on a finite-dimensional graded vector space $A$ is equivalent to a cyclic $A_\infty$ structure of dimension $d-1$ on $A \oplus A^\vee[1-d]$ (cyclic with respect to the standard pairing, in the sense of \cite{kontsevich2009notes}), such that $A$ is an $A_\infty$-subalgebra \cref{prop:finiteDimensional}. That is, a pre-CY structure is an extension of an $A_\infty$-structure on $A$, by operations with inputs and outputs in both $A$ and its linear dual $A^\vee$, satisfying an appropriate cyclic symmetry.

We describe some examples of pre-CY structures appearing in algebra, geometry and topology. Directly from the definition above, an easy example of pre-CY category is a cyclic $A_\infty$-category itself (\cref{prop:cyclictopre}). Using a noncommutative perspective on $A_\infty$- and pre-CY structures, in \cref{thm:Minimal} we describe how to produce pre-CY structures on minimal models of $A_\infty$-categories starting from a \emph{noncommutative Lagrangian} inside of a noncommutative symplectic space, that is, a certain type of subspace of a proper Calabi-Yau algebra.

Other sources of examples of pre-CY structures are constructions of `relative orientations' in topology and algebraic geometry. In \cref{sec:finDimManifolds}, we describe how to produce a pre-CY structure on the de Rham cohomology of a manifold $M$ oriented relative to its boundary $\del M$. An algebraic analogue of this class of examples is given by algebraic varieties endowed with sections of their anti-canonical bundles. In \cref{thm:varietyWithSection} we show that, given such a variety $X$ and a generator of its derived category $E$, from the data of a section $s$ of $\omega_X^{-1}$ one can construct a pre-CY structure on the algebra $A_X = \End(E)$. This includes, as a special case, Calabi-Yau varieties endowed with the the inverse of the trivialization of their canonical bundles.

\subsection{Algebraic and geometric aspects}

The general definition of a pre-CY structure can also be phrased operadically: recall that an $A_\infty$-structure is the data of a module structure over chains on the topological operad of Stasheff associahedra; each cell on this chain complex is labeled by a \emph{corolla}, a disc with many inputs and one output. The definition of pre-CY structure is an extension of that: a pre-CY structure is a collection of maps giving an algebra structure over a \emph{dioperad} $MC^d_{n_1,\dots,n_k}$ of chains on spaces of \emph{multicorollas}, that is, discs with many inputs and many outputs; we discuss this perspective later in \cref{sec:multicorolla}. 

A discussion of the homotopy theory of this dioperad is the object of the work of Leray and Vallette \cite{leray2022pre}, where it is shown that one can obtain the dioperad we describe from a Koszul duality-type construction on the dioperad parametrizing algebras with double Poisson brackets. Moreover, the notion of morphisms of pre-CY algebras one obtains from this general procedure agrees with the notion we define in \cref{sec:categoryPreCYalgebras}. This perspective partially explains the relation between pre-CY structure and other versions of noncommutative Poisson structures which have appeared in the literature, such as the notions of double Poisson algebras and double Poisson quasi-algebras of van den Bergh \cite{van2008double}. These algebras turn out to be particular cases of pre-CY algebras; this is proven in \cite{iyudu2021pre} for double Poisson algebras and in \cite{fernandez2021cyclic} for double Poisson quasi-algebras.

As for the relationship to derived geometry, a general statement relating Calabi-Yau structures and derived symplectic geometry has been proven in \cite{brav2021relative}: from a smooth Calabi-Yau structure of dimension $d$ on a smooth, finite-type dg category $\cC$, one gets a shifted symplectic structure on the associated `moduli space of objects' $\cM_\cC$, which is a derived stack defined by To\"en and Vaqui\'e \cite{toen2007moduli}. By fixing the dimensions, one recovers the symplectic structures on representation spaces, which sit inside of $\cM_\cC$. More broadly, the authors of \cite{brav2021relative} actually prove a relative version of this statement, giving derived Lagrangians inside of the spaces $\cM_\cC$. 

We expect that an analogous statement holds in general for pre-Calabi-Yau structures and shifted Poisson structures as defined in \cite{calaque2017shifted}, generalizing the result about representation spaces from \cite{yeung2018precalabiyau}; to our knowledge a proof of this statement in full generality, directly using the stacks $\cM_\cC$, has not appeared in the literature. Nevertheless, some version of this statement follows from the recent work of Pridham \cite{pridham2020shifted}, using a more abstract formalism of noncommutative derived stacks; there, the author mentions that a pre-CY structure on a finite-dimensional pre-CY algebra $A$ gives rise to a shifted Poisson structure on its moduli stack, using the relation between pre-CY structures and relative CY structures.

Without passing to spaces of representations, the definition of pre-CY structures can also be given in the language of noncommutative geometry developed in \cite{kontsevich2009notes}. In that setting, an $A_\infty$-structure on $A$ is equivalent to a solution of a certain Maurer-Cartan problem in the dg Lie algebra given by the (shifted) Hochschild cochains $C^*(A)[1]$, with Lie bracket given by the Gerstenhaber bracket $[-,-]_G$; this is the data of a \emph{homological vector field} on the noncommutative pointed scheme associated to $A$.  In this interpretation, the space $C^*_{[d]}(A)$ is the space of shifted polyvector fields on the noncommutative space associated to $A$, and the necklace bracket is the analog of the Schouten-Nijenhuis bracket. The $k=1$ component of the Maurer-Cartan equation shows that a pre-CY algebra is also an $A_\infty$-algebra; a pre-CY structure extending a given $A_\infty$ structure is a system of integrable polyvector fields extending the corresponding vector field.

Pre-CY structures have a well-behaved deformation theory, extending the deformation theory of $A_\infty$-structures. In \cref{sec:deformation}, we describe how to calculate the relevant obstructions for smooth or compact $A_\infty$ categories, using the diagonal bimodule and its duals. \\

\subsection{Topological aspects}
Let us briefly explain how this work fits in the cobordism perspective on \textsc{tqft}s. As we already mentioned, that a (compact) CY structure leads to an action of a \textsc{prop} of noncompact surfaces has been interpreted by Lurie to be the characterization of a certain type of fully extended oriented 2d \textsc{tqft}. Namely, there is a notion of Calabi-Yau object in any symmetric monoidal $(\infty,2)$-category $\cA$, such that the data of such an object in $\cA$ is equivalent to the data of a `non-compact' fully extended oriented 2d \textsc{tqft}  $Z: \Bord^{nc}_2 \to \cA$. Such a \textsc{tqft} only assigns values to cobordisms given by surfaces with $\ge 1$ incoming boundary components, which form the $(\infty,2)$-category of `noncompact' 2d cobordisms $\Bord^{nc}_2 $. The precise relation between the spaces $\cM_{g,\vec{m},\vec{n}}$ and the $\Bord^{nc}_2$ is given by a certain `unfolding construction' \cite{lurie2009classification}. 

We expect that \cref{thm:MainPROP} can be given a similar interpretation, using instead a variation $\Bord^{mid}_2$ of $\Bord^{nc}_2 $, made up of `middle-index oriented bordisms', that is, 2-bordisms that can be generated by handles of index one only; every such bordism has at least one input and at least one output. Then, if $\cC$ is a `good' symmetric monoidal $(\infty,2)$-category, in the sense of \cite{lurie2009classification}, linear over $\QQ$ (for example, an $(\infty,2)$-category of algebras and bimodules over a field $\kk$ of characteristic zero), a pre-CY object in $\cC$ (in the example, a pre-CY algebra over $\kk$) should determine a middle-index oriented 2d \textsc{tqft} valued in $\cC$: a symmetric monoidal functor $Z: \Bord^{mid}_2 \to \cC$. Proving such a correspondence with all its $\infty$-categorical technicalities is beyond the scope of this paper, but we propose this perspective as a heuristic explanation of why, for example, compact and smooth CY structures should give rise to pre-CY structures, essentially by restriction of the \textsc{tqft} structure to middle-index cobordisms. We postpone these discussions about the relations between smooth CY structures and pre-CY structures to \cite{KTV2}.

\subsection{Relation to existing literature}
Both this paper and \cite{KTV2} are updated and expanded versions of unfinished preprints that were written by the first and third named authors in 2013, and which have since then circulated among the community in their unfinished form, being cited by other articles and lectures such as \cite{iyudu2020precalabiyau,iyudu2021pre,yeung2018precalabiyau,fernandez2021cyclic,seidel2021fukaya}; since the appearance of the first complete version of this paper in 2021, our formalism has also found applications in high-energy physics \cite{sharapov2023chiral,sharapov2023more}.

Due to this state of affairs, some results and definitions here have also been discussed elsewhere in the literature, often citing that unfinished version of this paper. Throughout the text, we have made an attempt to be thorough in referring to those articles, and to be clear about what results in this expanded version are new. Let us give a non-extensive overview of related work. In the finite-dimensional case, pre-CY algebras have been defined and studied before by another name: some time between 2013 and now we learned from T. Tradler and M. Zeinalian that they had made an equivalent definition in \cite{tradler2007infinity} where this structure was called a \emph{$V_\infty$-algebra}. Moreover, P. Seidel also informed us that in \cite{seidel2012fukaya} he gave the same definition (Definition 3.5), calling them \emph{boundary algebras}. In the infinite dimensional case, some relations between pre-CY structures and symplectic/Poisson geometry appear in \cite{yeung2018precalabiyau,iyudu2020precalabiyau}.

Tradler-Zeinalian also describe the proof of an analogous result to our \cref{thm:MainPROP}, describing a \textsc{prop} they call $\cD\cG_\infty$ (for directed graphs); this turns out to be the closed part of our \textsc{prop}, and in the finite-dimensional case their proof is equivalent to ours. On the other hand, the uniformization of open-closed moduli space by meromorphic Strebel differentials is new, and gives an answer the question posed in Remark 5.6 of \cite{tradler2007infinity}.

Another place where a similar \textsc{prop} appears is in the work of Wahl-Westerland \cite{wahl2016hochschild} on \emph{black-and-white graphs}, which are similar objects to ours but without directions along the edges. The moduli spaces we describe here are homotopic to the ones proved by \cite{ega2015comparing} to give classifying spaces for open-closed cobordisms. Our quiver structure gives a finer stratification; the induced action on $A$ and $C_*(A)$ agrees with their description in the case where $A$ is a cyclic $A_\infty$-algebra, but the finer stratification allows us to generalize away from the finite-dimensional case.

This model of black-and-white graphs, acting on a cyclic $A_\infty$-algebra, has appeared in the recent papers \cite{caldararu2020categorical1,caldararu2020categorical2,caldararu2024effective}, where the authors refer to its use in the context of categorical Gromov-Witten invariants. We believe that the formalism of this paper can be used to extend their formalism to calculate invariants of categories that are not cyclic $A_\infty$. Though we make no further steps in that direction, we emphasize that one of the points of developing this theory is to give some understanding of this PROP action, together with tools to calculate it, in cases where $A$ is \emph{not compact}; in those cases it will be impossible to find a cyclic $A_\infty$-model of it. 

\subsection{Current and future research}
In the follow-up paper \cite{KTV2}, we discuss how to produce pre-CY structures on some categories that fail to be compact, but have instead smooth Calabi-Yau structures. There, we explain how the relation between pre-CY structures and smooth CY structures should be seen as a noncommutative version of the Legendre transform perspective on Poisson and symplectic structures.

Following this connection, one can produce pre-CY structures on algebras that are smooth but not proper, such as the algebra of chains on the based loop space of an orientable manifold. One fruitful line of research is the application of the graphical calculus governing pre-CY structures to the study of string topology operations; this is being currently developed by a collaboration involving the second-named author in \cite{rivera2023algebraic,RivTak}.

Another current line of research regards the properadic aspects of pre-CY structures, made possible by the use of properadic tools such as in \cite{leray2022pre}. Together with C. Emprin, the second-named author of this paper has been developing the theory of formality for pre-CY structures, using a properadic version of Kaledin classes \cite{emprin2024kaledin}. Combined with the PROP developed in this paper, one can use these formality results for pre-CY algebras to address formality of the corresponding genus zero structure, which is a framed $E_2$-algebra structure.\\ 

\noindent \textit{Acknowledgments:} AT would like to thank Dori Bejleri, Kai Cieliebak, Sheel Ganatra, Subhojoy Gupta, Aaron Mazel-Gee, Paul Seidel, Vivek Shende, Bruno Vallette and Michael Wolf for helpful conversations, and IHES and Institut Mittag-Leffler for excellent working conditions. YV would like to thanks IHES for continued support and excellent working conditions. He would also like to thank Ludmil Katzarkov for support at the University of Vienna and the University of Miami and Marie Claude Vergne for her expert help with making several of the pictures. This work was supported by the Simons Collaboration on Homological Mirror Symmetry.

\section{Background material}\label{sec:Background}
Throughout this paper, we work over a fixed field $\kk$ of characteristic zero. We will denote by $\Vect$ the symmetric monoidal categories of $\ZZ$-graded vector spaces, with monoidal structure given by the tensor product $\otimes$. We will use cohomological grading, and denote by $[1]$ the shift functor acting on objects of $\Vect$ as $(V[1])^n = V^{n+1}$. For any homogeneous element $a \in V$, we write $\deg(a)$ for its degree and denote by $\bar a = \deg(a)-1$ its degree in $V[1]$. In fact, all our statements hold equally for the $\ZZ/2$-graded case, and we will simply write `graded vector space'.

\subsection{A-infinity algebras and categories}\label{subsec:Ainfinity}
Let us denote by $\cA$ the data of a set of objects $\Ob(\cA)$ and for any two objects $X,Y \in \Ob(\cA)$, a graded vector space $\cA(X,Y)$.

The space of Hochschild cochains on $\cA$ of length $n$ is the graded vector space
\[ C^*(\cA)^n := \prod_{X_0,\dots, X_n \in \Ob(\cA)} \Hom(\cA(X_0,X_1)[1] \otimes \dots \otimes \cA(X_{n-1},X_n)[1], \cA(X_0,X_n)) \]

\begin{definition}\label{def:HochschildCochains}
The space of \emph{Hochschild cochains} of $\cA$ is the graded vector space given by
\[ C^*(\cA) = \prod_{n \ge 0} C^*(\cA)^n \]
\end{definition}
Note that in our notation the superscript $ ^*$ will always indicate total cohomological degree.

\begin{example}
If $\cA$ has a single object $X$ with a graded vector space $\cA(X,X) = A$ concentrated in degree zero, $C^*(\cA)$ is given by $\Hom(A^{\otimes n},A)$ in degree $+n$.
\end{example}

We endow $C^*(\cA)$ with the Gerstenhaber product $\circ$, a non-associative operation defined by
\[ f \circ g(a_1, \dots, a_n) = \sum_{i,j} (-1)^\# f(a_1, \dots , a_i, g(a_{i+1},\dots), a_j, \dots, a_n) \]
where $\# = \bar{g}\sum_{k=1}^i \overline{a_k}$, with $\bar g = \deg(g)-1$. This gives a product of degree $-1$, that is, a morphism of graded vector spaces
\[ C^*(\cA) \otimes C^*(\cA) \to C^*(\cA)[-1]. \]

\begin{definition}\label{def:GerstenhaberBracket}
The \emph{Gerstenhaber bracket} $[-,-]$ is the binary operation on $C^*(\cA)$ defined by
\[ [f,g] = f \circ g - (-1)^{\bar{f}\bar{g}} g \circ f\]
\end{definition}
The Gerstenhaber bracket endows the shifted space of Hochschild cochains $C^*(\cA)[1]$ with the structure of a graded Lie algebra.

\begin{definition}\label{def:AinftyStructure}
An $A_\infty$ structure on $\cA$ is an element $\mu \in C^2(\cA)$ satisfying $\mu \circ \mu = 0$, with vanishing length zero component $\mu^0 = 0$.
\end{definition}
We will often refer to $\cA$ as an $A_\infty$-category while leaving the $A_\infty$ structure $\mu$ implicit. For any $n \ge 1$, we denote by $\mu^n \in C^2(\cA)^n$ the length $n$ component of such a structure; this is the data of maps
\[ \mu^n: \cA(X_0,X_1)[1] \otimes \dots \otimes \cA(X_{n-1},X_n)[1] \to \cA(X_0,X_n)  \]
for all $(n+1)$-tuple of objects $X_i$. The first component $\mu^1$ squares to zero, and endows $C^*(\cA)$ with the structure of a cochain complex.

We recall a noncommutative geometry perspective of $A_\infty$-algebras \cite{kontsevich2009notes}. In this interpretation, the Gerstenhaber bracket corresponds to the bracket of vector fields on a noncommutative space determined by the graded vector space $A$, and an $A_\infty$-structure on $A$ is an integrable vector field of degree one (or \emph{homological vector field}) on that space. 

The notions above, together with definitions of morphisms, quasi-isomorphisms etc. have appeared in many places in the literature, often with different sign conventions. We refer the reader to \cite[Sec.3.2]{sheridan2020formulae} for the notions of $A_\infty$-functors and unitality conditions, with signs that are consistent with the ones we use.

\subsubsection{Modules and bimodules}
Throughout this paper we will make use of modules and bimodules over $A_\infty$-categories. Again, there are detailed expositions of this theory, for example \cite{ganatra2013symplectic,sheridan2020formulae}. We will use the same sign conventions as \cite{seidel2008subalgebras}; though this reference deals with $A_\infty$-algebras only, it is straightforward to generalize the formulas to $A_\infty$-categories. For example, for the definition of bimodules, we have:

\begin{definition}
A $(\cA,\cB)$-bimodule $M$ over a pair of $A_\infty$-categories $(\cA,\mu_\cA)$ and $(\cB,\mu_\cB)$ is a graded vector space $M(X,X')$ for every pair of object $X \in \cB, X' \in \cA$, along with maps of degree one
\begin{align*}
	\hspace{-1.5cm} \mu^{r,1,s}_M&: \cB(X_1, X_2)[1] \otimes \dots \otimes \cB(X_{r-1},X_r)[1] \otimes M(X_r,X'_1) \otimes \cA(X'_1,X'_2)[1] \otimes \dots \otimes \cA(X'_{s-1}, X'_s)[1] \\
	& \to M(X_1,X'_s)
\end{align*}
for $r,s \ge 0$ and tuples of objects $X_i,X'_j$, satisfying equation 2.5 of \cite{seidel2008subalgebras}.
\end{definition}

From now on, we will work only with homologically unital $A_\infty$-categories and homologically unital modules over them. Given any pair $(\cA,\cB)$ of $A_\infty$-categories, $(\cA,\cB)$-bimodules form a dg category, with morphisms given by collections of maps satisfying Equation 2.8 of \cite{seidel2008subalgebras}.
Moreover, given any three $A_\infty$-categories $(\cA,\cB,\cC)$, there is an operation of tensor product over $\cB$ given by a dg functor
\[ -\otimes_\cB -: \cA\mh\Mod\mh\cB \times \cB\mh\Mod\mh\cC \to \cA\mh\Mod\mh\cC\]
(Equation 2.15 of \cite{seidel2008subalgebras}) and a two-sided tensor functor 
\[ -\otimes_{\cA\mh\cB} -: \cA\mh\Mod\mh\cB \times \cB\mh\Mod\mh\cA \to \Vect \]
which is a special case of the cyclic tensor product \cite[Sec.5]{seidel2008subalgebras}.

We denote by $\cA_\Delta \in \cA\mh\Mod\mh\cA$, the \emph{diagonal bimodule} of $\cA$, whose spaces $\cA_\Delta(X,Y)$ are equal to the morphism spaces of the $A_\infty$-category $\cA$ itself; this object carries a natural quasi-isomorphism $\cA_\Delta \otimes_\cA \cM \xrightarrow{\sim} \cM$ for any left $\cA$-module $\cM$ (and similarly for right modules and bimodules) \cite[Prop.2.2]{ganatra2013symplectic}.

\subsubsection{Hochschild co/homology}
Let $(\cA,\mu)$ be an $A_\infty$ category; let us recall the definitions of Hochschild homology and cohomology of $\cA$.
\begin{definition}\label{def:HochschildCohomology}
The \emph{Hochschild cochain complex} is the graded vector space $C^*(\cA)$ endowed with the differential $d := [\mu, -]$ of degree +1; its cohomology $HH^*(\cA) := H^*(C^*(\cA), d)$ is called the \emph{Hochschild cohomology} of the $A_\infty$ category $\cA$.
\end{definition}
Note that by definition, $[\mu,\mu] = 0$ so $d^2=0$ and we have a class $[\mu] \in HH^2(\cA)$.

We define the space of Hochschild chains $C_*(\cA)$ as the graded vector space
\[ C_n(\cA) := \prod_{X_0,\dots, X_n \in \Ob(\cA)} \cA(X_0,X_1) \otimes \dots \otimes \cA(X_{n-1},X_n) \otimes \cA(X_n,X_0)[n] \]
\begin{definition}\label{def:HochschildHomology}
The \emph{Hochschild chain complex} is the graded vector space $C_*(\cA)$ endowed with a differential $b$ of degree +1 defined on generators as follows
\begin{align*}
    b(a_0 \otimes a_1 \otimes \dots \otimes a_n)   &= \sum_{i,j} (-1)^{\#_1} \mu^{i+j}(a_{j+1},\dots,a_0,\dots,a_i) \otimes a_{i+1} \otimes \dots \otimes a_j \\
         &+ \sum_{i,j} (-1)^{\#_2} a_0 \otimes \dots \otimes \mu^{j-i}(a_{i+1},\dots,a_j) \otimes \dots \otimes a_n,
\end{align*}
where $\#_1 = (\bar a_1 + \dots + \bar a_j)(\bar a_{j+1} + \dots + \bar a_n)$ and $\#_2 = \overline{a_0} + \dots + \overline{a_i}$. This complex calculates the \emph{Hochschild homology} $HH_*(\cA) := H^*(C_*(\cA), b)$ of the $A_\infty$ category $\cA$.
\end{definition}

\begin{remark}
Note that even though we denote Hochschild homology with a subscript, we are still using the cohomological grading convention for it; the differential $b$ has degree +1.
\end{remark}

\subsection{Graphical calculus for A-infinity categories and bimodules}\label{sec:graphical}
Anyone who tries to do calculations with $A_\infty$-structures will eventually encounter the annoying appearance of signs everywhere. We now present a unified graphical calculus describing $A_\infty$-categories and -bimodules, which allows one to write the relevant formulas with a minimum of explicit signs. The resulting sign conventions are in conformity with the signs appearing in e.g. \cite{seidel2008subalgebras,ganatra2013symplectic}. For simplicity we work with an $A_\infty$-algebra $A$, but the procedure is easily extended to categories.

\subsubsection{Representation for Hochschild cochains}\label{sec:RepHochChains}
Let $\phi \in C^*(A)$ denote a Hochschild cochain of $A$. We interpret $\phi = \sum_n \phi^n$ as a collection of maps
\[ \phi^n: A[1]^{\otimes n} \to A[1] \]
(note the shift in the target) of degree $\bar \phi = \deg(\phi) - 1$ and denote it graphically by the diagram
\[\begin{tikzpicture}[baseline={([yshift=-.5ex]current bounding box.center)}]
	\node [vertex] (a) at (0,0) {$\phi$};
	\draw [->-] (a) to (0,-1);
\end{tikzpicture}\]

One can use this graphical representation to calculate the action of the $E_2$-operad on Hochschild cochains described in \cite{kontsevich2000deformations}; given any rooted tree, which we visualize as sitting inside a circle with the root at the bottom, we can insert Hochschild cochains into the vertices and assign a Hochschild cochain to the whole diagram. We will later extend this construction to more general directed trees (\cref{sec:graphicalhigher}) and then to `ribbon quivers' (\cref{sec:PROPs}), and for that we need a coherent way of calculating signs. This can be deduced from the paper cited above, but for concreteness let us work through an example. Consider the following diagram
\[ \Gamma = \begin{tikzpicture}[baseline={([yshift=-.5ex]current bounding box.center)}]
	\draw (0,0.7) circle (1.7);
	\node [vertex] (a) at (-0.6,1.4) {$\phi_1$};
	\node [vertex] (b) at (0.6,0.8) {$\phi_2$};
	\node [vertex] (c) at (0,0) {$\phi_3$};
	\draw [->-] (a) to (c);
	\draw [->-] (b) to (c);
	\draw [->-] (c) to (0,-1);
\end{tikzpicture}\]

An \emph{ordering} on $\Gamma$ is a \emph{linear extension} of the partial ordering induced by the directions along the arrows. In the diagram above, we have a partial ordering
\[ \phi_1 > \phi_3, \quad \phi_2 > \phi_3 \]
and we can choose for instance the ordering
\[ (\phi_3 \phi_2 \phi_1), \text{\ meaning\ } \phi_3 < \phi_2 < \phi_1 \]
which we suggested in the figure above by drawing $\phi_1$ above $\phi_2$.

Each diagram is embedded in a disc, with a marked point on the boundary corresponding to the outgoing arrow at the bottom. To the ordered diagram above we assign a Hochschild cochain $\left(\Gamma, (\phi_3 \phi_2 \phi_1)\right)$ whose value on $a_1 \otimes a_2 \otimes \dots \otimes a_n \in A[1]^{\otimes n}$ is calculated as follows.

\begin{enumerate}
	\item We write the elements $a_i$ as incoming arrows around the circle, \emph{clockwise} starting from the bottom. We then choose some way of connecting these arrows to the vertices while respecting the cyclic order \emph{without crossings}. For example, given an element $a_1 \otimes a_2 \otimes a_3 \otimes a_4 \otimes a_5 \in A[1]^{\otimes 5}$, one such diagram is
	\[\begin{tikzpicture}[baseline={([yshift=-.5ex]current bounding box.center)}]
		\draw (0,0.7) circle (1.7);
		\node [vertex] (phi1) at (-0.6,1.4) {$\phi_1$};
		\node [vertex] (phi2) at (0.6,0.8) {$\phi_2$};
		\node [vertex] (phi3) at (0,0) {$\phi_3$};

		\draw [->-] (phi1) to (phi3);
		\draw [->-] (phi2) to (phi3);
		\draw [->-] (phi3) to (0,-1);

		\draw [->-] (-1.55,0) to (phi3);
		\draw [->-] (-1.68,1) to (phi1);
		\draw [->-] (0,2.4) to (phi1);
		\draw [->-] (1.68,1) to (phi2);
		\draw [->-] (1.55,0) to (phi3);

		\node at (-1.75,0) {$a_1$};
		\node at (-1.88,1) {$a_2$};
		\node at (0,2.6) {$a_3$};
		\node at (1.88,1) {$a_4$};
		\node at (1.75,0) {$a_5$};
	\end{tikzpicture}\]

	\item We write the ordering and the element in $A[1]^{\otimes n}$ next to each other:
	\[ (\phi_3 \phi_2 \phi_1)(a_1 a_2 a_3 a_4 a_5) \]

	\item We permute the $a_i$ such that $\phi_1$ precedes its inputs immediately, recording the Koszul sign:
	\[ (\phi_3 \phi_2 \phi_1)(a_1 a_2 a_3 a_4 a_5) = (-1)^{\bar a_1 (\bar a_2 + \bar a_3)} (\phi_3 \phi_2 \phi_1)(a_2 a_3 a_1 a_4 a_5) \]

	\item We evaluate $\phi_1(a_2,a_3)$ and write the result in the place occupied by $\phi_1$ and its inputs:
	\[ (-1)^{\bar a_1 (\bar a_2 + \bar a_3)} (\phi_3 \phi_2) (\phi_1(a_2, a_3) a_1 a_4 a_5) \]

	\item Repeat steps (3) and (4) until we are left with an element of $A[1]$, which we then interpret as an element of $A$.	In the case above, this result is
	\[ (-1)^{\bar \phi_1 \bar a_1 + \bar \phi_2(\bar a_1 + \bar a_2 + \bar a_3) + \bar \phi_1 \bar \phi_2} \phi_3(a_1, \phi_1(a_2, a_3),\phi_2(a_4), a_5) \]
	In the sign, $\bar \phi$ as usual denotes the degree of $\phi$ as a map from copies of $A[1]$ to copies of $A[1]$; that is, $\bar \phi = \deg(\phi) -1$.
\end{enumerate}

In the example above, we could have chosen the ordering $(\phi_3 \phi_1 \phi_2)$ instead, switching $\phi_1$ and $\phi_2$. The exponent in the sign then would have been
\[ \bar \phi_1 \bar a_1 + \bar \phi_2(\bar a_1 + \bar a_2 + \bar a_3) \]
which differs by $\bar \phi_1 \bar \phi_2$ from the one above. One can check that this is a general fact: maps determined by any two orderings differ by a minus sign with exponent given by $\sum_{(ij)} \bar \phi_i \bar \phi_j$, where we sum over transpositions $(ij)$ in the permutation between the orderings.

The Gerstenhaber product and bracket are part of the $E_2$-algebra structure on $C^*(\cA)$, and are given by:
\[ \phi \circ \psi = \begin{tikzpicture}[baseline={([yshift=-.5ex]current bounding box.center)}]
	\draw (0,0) circle (1);
	\node [vertex] (psi) at (0,0.55) {$\psi$};
	\node [vertex] (phi) at (0,-0.4) {$\phi$};

	\draw [->-=0.8] (psi) to (phi);
	\draw [->-=0.8] (phi) to (0,-1);
\end{tikzpicture}, \qquad
[\phi, \psi] = \begin{tikzpicture}[baseline={([yshift=-.5ex]current bounding box.center)}]
	\draw (0,0) circle (1);
	\node [vertex] (psi) at (0,0.55) {$\psi$};
	\node [vertex] (phi) at (0,-0.4) {$\phi$};

	\draw [->-=0.8] (psi) to (phi);
	\draw [->-=0.8] (phi) to (0,-1);
\end{tikzpicture} - (-1)^{\bar\psi\bar\phi}\begin{tikzpicture}[baseline={([yshift=-.5ex]current bounding box.center)}]
	\draw (0,0) circle (1);
	\node [vertex] (phi) at (0,0.55) {$\phi$};
	\node [vertex] (psi) at (0,-0.4) {$\psi$};

	\draw [->-=0.8] (phi) to (psi);
	\draw [->-=0.8] (psi) to (0,-1);
\end{tikzpicture}\]
and the $A_\infty$-structure equations are just $\begin{tikzpicture}[baseline={([yshift=-.5ex]current bounding box.center)}]
	\draw (0,0) circle (1);
	\node [vertex] (phi) at (0,0.55) {$\mu$};
	\node [vertex] (psi) at (0,-0.4) {$\mu$};

	\draw [->-=0.8] (phi) to (psi);
	\draw [->-=0.8] (psi) to (0,-1);
\end{tikzpicture} = 0$,
with $\deg(\mu) = \bar \mu + 1 = 2$ and $\mu^0 = 0$. One can check that the analogous algorithm to what we described above reproduces the correct signs.

\subsubsection{Hochschild chains}
It is also possible to include Hochschild chains in this type of graphical calculus, and describe operations such as the action of Hochschild cochains on Hochschild chains; this perspective appears in \cite{kontsevich2009notes}.

Above, we placed Hochschild cochains inside of disks on the plane. For Hochschild chains, one should imagine instead that it \emph{travels along a cylinder with a distinguished point on each boundary}. We squash the cylinder into an annulus for ease of representation, with the input end around the outside. For example, the identity map on Hochschild chains is
\[ \begin{tikzpicture}[baseline={([yshift=-.5ex]current bounding box.center)}]
	\draw (0,0) circle (1);
	\draw (0,0) circle (0.4);
	\draw [->-] (0,1) to (0,0.4);
	\node at (0,0.98) {$\bullet$};
	\node at (0,0.38) {$\bullet$};
\end{tikzpicture}\]
We then send a chain $a_0 \otimes a_1 \otimes \dots \otimes a_p$ from the outside, with $a_0$ along the arrow leaving the marked point $\bullet$ and the $a_i$ around in \emph{clockwise order}, i.e.
\[ \begin{tikzpicture}[baseline={([yshift=-.5ex]current bounding box.center)}]
	\draw (0,0) circle (1);
	\draw (0,0) circle (0.4);
	\node at (0,0.98) {$\bullet$};
	\node at (0,0.38) {$\bullet$};

	\draw [->-] (0,1) to (0,0.4);
	\draw [->-] (0.84,0.55) to (0.33, 0.22);
	\draw [->-] (0.84,-0.55) to (0.33, -0.22);
	\draw [->-] (-0.84,0.55) to (-0.33, 0.22);
	\draw [->-] (-0.84,-0.55) to (-0.33, -0.22);
	\node at (0,1.2) {$a_0$};
	\node at (1.1,0.6) {$a_1$};
	\node at (1.1,-0.6) {$a_2$};

	\node at (-1.3,-0.6) {$a_{p-1}$};
	\node at (-1.1,0.6) {$a_p$};
	\node at (0,-0.7) {$\dots$};
\end{tikzpicture}\]
which obviously gives the identity map.

Given a directed tree embedded in this cylinder, we can analogously define the action of diagrams of Hochschild cochains, as we did before in the disc. First we pick an ordering of the vertices as above and evaluate
	\[ (\phi_N \dots \phi_1) (a_0 a_1 \dots a_p) \]
with the same signs rules as above, and at the end, we permute the resulting outputs to put them in clockwise order starting from the marking $\bullet$.

The natural operations on Hochschild chains can then be described very simply. For example, the Hochschild differential corresponding to some $A_\infty$-structure $\mu \in C^2(A)$ is
\[ b = \begin{tikzpicture}[baseline={([yshift=-.5ex]current bounding box.center)}]
	\draw (0,0) circle (1.3);
	\draw (0,-0.7) circle (0.4);
	\node at (0,1.28) {$\bullet$};
	\node at (0,-0.27) {$\bullet$};

	\node [vertex] (mu) at (0,0.5) {$\mu$};
	\draw [->-] (0,1.3) to (phi);
	\draw [->-] (mu) to (0,-0.3);
\end{tikzpicture} + \begin{tikzpicture}[baseline={([yshift=-.5ex]current bounding box.center)}]
	\draw (0,0) circle (1.3);
	\draw (0,0.4) circle (0.4);
	\node at (0,1.28) {$\bullet$};
	\node at (0,0.78) {$\bullet$};
	\node [vertex] (mu) at (0,-0.8) {$\mu$};
	\draw [->-] (0,1.3) to (0,0.8);
	\draw [->-] (mu) to (0,0);
\end{tikzpicture}\]
and the cap product giving an action of $C^*(A)$ and $C_*(A)$, is
\[ \begin{tikzpicture}[baseline={([yshift=-.5ex]current bounding box.center)}]
	\draw (0,0) circle (1.3);
	\draw (0,-0.7) circle (0.4);
	\node at (0,1.28) {$\bullet$};
	\node at (0,-0.27) {$\bullet$};

	\node [vertex] (mu) at (0,0.5) {$\mu$};
	\node [vertex] (phi) at (-0.8,0.5) {$\phi$};
	\draw [->-=0.8] (phi) to (mu);
	\draw [->-] (0,1.3) to (mu);
	\draw [->-] (mu) to (0,-0.3);
\end{tikzpicture}
\]

In the case where $A$ has a strict unit $e_A$, we can also describe the Connes differential
\[ B = \begin{tikzpicture}[baseline={([yshift=-.5ex]current bounding box.center)}]
	\draw (0,0) circle (1.3);
	\draw (0,0.4) circle (0.4);
	\node at (0,1.28) {$\bullet$};
	\node at (0,-0.03) {$\bullet$};
	\node [vertex] (mu) at (0,-0.8) {$1$};
	\draw [->-] (0,1.3) to (0,0.8);
	\draw [->-] (mu) to (0,0);
\end{tikzpicture}\]
where $1 \in C^0(A)$ denotes the constant Hochschild cochain that evaluates $\kk \to e_A$ and $T_+(A[1]) \to 0$.

\subsubsection{Morphisms and bimodules}
Later in this paper, we will need to describe some operations analogous to the above, but using bimodules (that is, Hochschild chains/cochains of $\cA$ valued in some bimodule $\cM$). So let us describe an extension of the graphical calculus above that also includes bimodules. Though in this paper we will only be using $(\cA,\cA)$-bimodules (that is, where the $A_\infty$-algebras acting on the left and the right are the same), it is natural to allow them to be different. \footnote{This will also be useful in a future paper, where we plan to address relative Calabi-Yau structures in the sense of \cite{brav2019relative,brav2021relative} and therefore must deal with more than one $A_\infty$-category.}

Again for simplicity we describe the case of $A_\infty$-algebras. Let $(A,\mu_A)$ and $(B,\mu_B)$ be two $A_\infty$-algebras. We interpret an $(A,B)$-bimodule $M$ as a `boundary condition' (bold line) between planar regions labeled by $B$ (white) and $A$ (shaded), together with a `boundary point operator' $\mu_M$ which sits on the bold line, as follows:
\[\begin{tikzpicture}[baseline={([yshift=-.5ex]current bounding box.center)}]
	\draw (0,0) circle (1);
	\filldraw[fill=black!20!white] (0,-1) arc (-90:90:1) -- cycle;
	\node [vertex] (muM) at (0,0) {$\mu_M$};

	\draw [->-=0.8, very thick] (0,1) to (muM);
	\draw [->-=0.8, very thick] (muM) to (0,-1);
	\node at (-1.2,0) {$B$};
	\node at (1.2,0) {$A$};
\end{tikzpicture}
\]

Note that the bold arrow representing an $(A,B)$-bimodule $M$ has $A$ to \emph{its left} (as seen by someone walking along the arrow) and $B$ to \emph{its right}. On some element $(a_1,\dots,a_s,b_1,\dots,b_r) \in T(A[1]) \otimes T(B[1])$, we evaluate the vertex by inserting $b_i$ elements on the left and $a_i$ elements on the right, clockwise, and $m \in M$ on the top. We require that $\deg(\mu_M) = 1$, seen as a map
\[ B[1]^{\otimes r} \otimes M \otimes A[1]^{\otimes s} \to M \]

In order to properly use the sign convention, we need to specify how to do steps (2) and (3), i.e. how to assign degrees for the arrows and for the regions in the presence of a bimodule line. The prescription is essentially the same, except that we count elements $m \in M$ according to their degree $\deg(m)$ in $M$.

With this convention in mind, we can express the $A_\infty$ structure equation for bimodules as
\[\begin{tikzpicture}[baseline={([yshift=-.5ex]current bounding box.center)}]
	\draw (0,0) circle (1.5);
	\filldraw[fill=black!20!white] (0,-1.5) arc (-90:90:1.5) -- cycle;
	\node [vertex] (muM) at (0,-0.6) {$\mu_M$};
	\node [vertex] (muB) at (-0.8,0.3) {$\mu_B$};

	\draw [->-=0.8, very thick] (0,1.5) to (muM);
	\draw [->-=0.8, very thick] (muM) to (0,-1.5);
	\draw [->-=0.8] (muB) to (muM);
\end{tikzpicture} +
\begin{tikzpicture}[baseline={([yshift=-.5ex]current bounding box.center)}]
	\draw (0,0) circle (1.5);
	\filldraw[fill=black!20!white] (0,-1.5) arc (-90:90:1.5) -- cycle;
	\node [vertex] (muM) at (0,-0.6) {$\mu_M$};
	\node [vertex] (mutop) at (0,0.8) {$\mu_M$};

	\draw [->-=0.8, very thick] (0,1.5) to (mutop);
	\draw [->-=0.8, very thick] (mutop) to (muM);
	\draw [->-=0.8, very thick] (muM) to (0,-1.5);
\end{tikzpicture} +
\begin{tikzpicture}[baseline={([yshift=-.5ex]current bounding box.center)}]
	\draw (0,0) circle (1.5);
	\filldraw[fill=black!20!white] (0,-1.5) arc (-90:90:1.5) -- cycle;
	\node [vertex] (muM) at (0,-0.6) {$\mu_M$};
	\node [vertex] (muA) at (0.8,0.3) {$\mu_A$};

	\draw [->-=0.8, very thick] (0,1.5) to (muM);
	\draw [->-=0.8, very thick] (muM) to (0,-1.5);
	\draw [->-=0.8] (muA) to (muM);
\end{tikzpicture} = 0
\]

The dg category $A\mh\Mod\mh B$ of $(A,B)$-bimodules has objects given by boundary conditions with a $\mu$ vertex as above, that is, a pair of a graded vector space $M$ together with a set of maps $\mu_M: B[1]^{\otimes r} \otimes M \otimes A[1]^{\otimes s} \to M$. Morphisms between bimodules are given by `boundary condition changing operators': a morphism $F:M \to N$ is represented by vertex between a $M$-line and an $N$-line, and the differential on the space of morphisms is given by $F \mapsto [\mu,F]$, where $[\mu,F]$ is the vertex given by the sum of diagrams
\[
\hspace{-1.2cm}
\begin{tikzpicture}[baseline={([yshift=-.5ex]current bounding box.center)}]
\draw (0,0) circle (1.5);
\filldraw[fill=black!20!white] (0,-1.5) arc (-90:90:1.5) -- cycle;
\node [vertex] (muM) at (0,-0.6) {$\mu_N$};
\node [vertex] (mutop) at (0,0.8) {$F$};

\draw [->-=0.8, very thick] (0,1.5) to (mutop);
\draw [->-=0.8, very thick] (mutop) to (muM);
\draw [->-=0.8, very thick] (muM) to (0,-1.5);
\end{tikzpicture}
\ -(-1)^{\bar F} \left(
\begin{tikzpicture}[baseline={([yshift=-.5ex]current bounding box.center)}]
\draw (0,0) circle (1.5);
\filldraw[fill=black!20!white] (0,-1.5) arc (-90:90:1.5) -- cycle;
\node [vertex] (muM) at (0,-0.6) {$F$};
\node [vertex] (muB) at (-0.8,0.3) {$\mu_B$};

\draw [->-=0.8, very thick] (0,1.5) to (muM);
\draw [->-=0.8, very thick] (muM) to (0,-1.5);
\draw [->-=0.8] (muB) to (muM);
\end{tikzpicture} +
\begin{tikzpicture}[baseline={([yshift=-.5ex]current bounding box.center)}]
\draw (0,0) circle (1.5);
\filldraw[fill=black!20!white] (0,-1.5) arc (-90:90:1.5) -- cycle;
\node [vertex] (muM) at (0,-0.6) {$F$};
\node [vertex] (mutop) at (0,0.8) {$\mu_M$};

\draw [->-=0.8, very thick] (0,1.5) to (mutop);
\draw [->-=0.8, very thick] (mutop) to (muM);
\draw [->-=0.8, very thick] (muM) to (0,-1.5);
\end{tikzpicture} +
\begin{tikzpicture}[baseline={([yshift=-.5ex]current bounding box.center)}]
\draw (0,0) circle (1.5);
\filldraw[fill=black!20!white] (0,-1.5) arc (-90:90:1.5) -- cycle;
\node [vertex] (muM) at (0,-0.6) {$F$};
\node [vertex] (muA) at (0.8,0.3) {$\mu_A$};

\draw [->-=0.8, very thick] (0,1.5) to (muM);
\draw [->-=0.8, very thick] (muM) to (0,-1.5);
\draw [->-=0.8] (muA) to (muM);
\end{tikzpicture}
\right)
\]

Left- and right-modules over some $A_\infty$-category can be analogously described by setting one of the sides to be the rank one $A_\infty$-algebra given by the ground field $\kk$ in degree zero.

\subsubsection{Tensor products of bimodules}\label{sec:TensorProductsBimodules}
We can also use this graphical calculus to describe the tensor product of bimodules. Let $(A,\mu_A), (B,\mu_B), (C,\mu_C)$ be three $A_\infty$-algebras, and $M \in A\mh\Mod\mh B, N \in B\mh\Mod\mh C$ two bimodules. We have the following graphical `definition' of $M \otimes_B N$:
\[\begin{tikzpicture}[baseline={([yshift=-.5ex]current bounding box.center)}]
	\node (mn) at (0,1.4) {$M\otimes_B N$};
	\draw [->-,very thick] (0,1.2) to (0,0);
	\end{tikzpicture} =
	\begin{tikzpicture}[baseline={([yshift=-.5ex]current bounding box.center)}]
		\node (m) at (0,1.4) {$N$};
		\node (n) at (0.8,1.4) {$M$};
		\filldraw[fill=black!20!white] (0,0) -- (0.8,0) -- (0.8,1.2) -- (0,1.2) -- cycle;
		\draw [->-,very thick] (0,1.2) to (0,0);
		\draw [->-,very thick] (0.8,1.2) to (0.8,0);
		\node at (-0.4,0.7) {$C$};
		\node at (0.4,0.7) {$B$};
		\node at (1.2,0.7) {$A$};
		\end{tikzpicture}
\]
where the shaded area is where $B[1]$-lines travel. That is, as a graded vector space it is $M \otimes T(B[1]) \otimes N$, with structure map $\mu = \mu_{M \otimes_B N}$ given by the expression:
\[
\begin{tikzpicture}[baseline={([yshift=-.5ex]current bounding box.center)}]
\node (mn) at (0,2.5) {$M\otimes_B N$};
\node [vertex] (mumn) at (0,1.25) {$\mu$};
\draw [->-,very thick] (0,2.3) to (mumn);
\draw [->-, very thick] (mumn) to (0,0);
\end{tikzpicture} =
\begin{tikzpicture}[baseline={([yshift=-.5ex]current bounding box.center)}]
\node (m) at (0,2.5) {$N$};
\node (n) at (1.2,2.5) {$M$};
\filldraw[fill=black!20!white] (0,0) -- (1.2,0) -- (1.2,2.3) -- (0,2.3) -- cycle;

\node [vertex] (mun) at (0,1.25) {$\mu_N$};
\draw [->-,very thick] (0,2.3) to (mun);
\draw [->-, very thick] (mun) to (0,0);
\draw [->-,very thick] (1.2,2.3) to (1.2,0);
\end{tikzpicture} +
\begin{tikzpicture}[baseline={([yshift=-.5ex]current bounding box.center)}]
\node (m) at (0,2.5) {$N$};
\node (n) at (1.2,2.5) {$M$};
\filldraw[fill=black!20!white] (0,0) -- (1.2,0) -- (1.2,2.3) -- (0,2.3) -- cycle;

\node [vertex] (mub) at (0.6,1.25) {$\mu_B$};
\draw [->-,very thick] (0,2.3) to (0,0);
\draw [->-] (mub) to (0.6,0);
\draw [->-,very thick] (1.2,2.3) to (1.2,0);
\end{tikzpicture} +
\begin{tikzpicture}[baseline={([yshift=-.5ex]current bounding box.center)}]
\node (m) at (0,2.5) {$N$};
\node (n) at (1.2,2.5) {$M$};
\filldraw[fill=black!20!white] (0,0) -- (1.2,0) -- (1.2,2.3) -- (0,2.3) -- cycle;

\node [vertex] (mum) at (1.2,1.25) {$\mu_M$};
\draw [->-,very thick] (0,2.3) to (0,0);
\draw [->-, very thick] (mum) to (1.2,0);
\draw [->-,very thick] (1.2,2.3) to (mum);
\end{tikzpicture}
\]

Given two $A_\infty$-algebras $A,B$, we can tensor two bimodules $M \in A\mh\Mod\mh B$ and $N \in B \mh \Mod \mh A$ simultaneously over $A$ and $B$, and get a differential graded vector space $M \otimes_{A-B} N$, the \emph{two-sided tensor product}. As a graded vector space this is given by $M \otimes T(B[1]) \otimes N \otimes T(A[1])$. We can express this as traveling along a cylinder (or annulus) with two marked lines along which elements of $M$ and $N$ travel. The differential $d_{M \otimes_{A \otimes B} N}$ is then given by the diagram:
\[
\hspace{-1.4cm}
\begin{tikzpicture}[baseline={([yshift=-.5ex]current bounding box.center)}]
\draw (0,0) circle (1.8);
\draw (0,0) circle (0.5);
\filldraw[fill=black!20!white] (0,-1.8) arc (-90:90:1.8) -- (0,0.5) arc (90:-90:0.5) -- cycle;

\node at (0,2) {$N$};
\node at (0,-2) {$M$};
\node [vertex] (mu) at (0,-1.2) {$\mu_M$};

\draw [->-=1, very thick] (0,-1.8) to (mu);
\draw [->-=1, very thick] (mu) to (0,-0.5);
\draw [->-=0.5, very thick] (0,1.8) to (0,0.5);
\end{tikzpicture} +
\begin{tikzpicture}[baseline={([yshift=-.5ex]current bounding box.center)}]
\draw (0,0) circle (1.8);
\draw (0,0) circle (0.5);
\filldraw[fill=black!20!white] (0,-1.8) arc (-90:90:1.8) -- (0,0.5) arc (90:-90:0.5) -- cycle;

\node at (0,2) {$N$};
\node at (0,-2) {$M$};
\node [vertex] (mu) at (-1.2,0) {$\mu_A$};

\draw [->-=1, very thick] (0,1.8) to (0,0.5);
\draw [->-] (mu) to (-0.5,0);
\draw [->-=0.5, very thick] (0,-1.8) to (0,-0.5);
\end{tikzpicture} +
\begin{tikzpicture}[baseline={([yshift=-.5ex]current bounding box.center)}]
\draw (0,0) circle (1.8);
\draw (0,0) circle (0.5);
\filldraw[fill=black!20!white] (0,-1.8) arc (-90:90:1.8) -- (0,0.5) arc (90:-90:0.5) -- cycle;

\node at (0,2) {$N$};
\node at (0,-2) {$M$};
\node [vertex] (mu) at (0,1.2) {$\mu_N$};

\draw [->-=1, very thick] (0,1.8) to (mu);
\draw [->-=1, very thick] (mu) to (0,0.5);
\draw [->-=0.5, very thick] (0,-1.8) to (0,-0.5);
\end{tikzpicture} +
\begin{tikzpicture}[baseline={([yshift=-.5ex]current bounding box.center)}]
\draw (0,0) circle (1.8);
\draw (0,0) circle (0.5);
\filldraw[fill=black!20!white] (0,-1.8) arc (-90:90:1.8) -- (0,0.5) arc (90:-90:0.5) -- cycle;

\node at (0,2) {$N$};
\node at (0,-2) {$M$};
\node [vertex] (mu) at (1.2,0) {$\mu_B$};

\draw [->-=1, very thick] (0,1.8) to (0,0.5);
\draw [->-] (mu) to (0.5,0);
\draw [->-=0.5, very thick] (0,-1.8) to (0,-0.5);
\end{tikzpicture}
\]
seen as an operation from $M \otimes_{A \otimes B} N$ (outside the annulus) to itself (inside the annulus) of degree $+1$.

\subsection{Calabi-Yau structures}

\subsubsection{The diagonal bimodule and its duals}\label{sec:diagonalBimodule}
Recall that given any $A_\infty$-algebra $A$ there is a canonical object of $A \mh \Mod \mh A$, its diagonal bimodule $A_\Delta$. As a graded vector space, it is equal to $A$, and its structure maps are produced from the structure maps of $A$ with sign changes to account for shifts.

Given graded vector spaces $V_0, V_1, \dots V_n$, consider the graded vector space $\Hom(V_1\otimes \dots \otimes V_n, V_0)$. Shifting one of the factors by $[-1]$, say $V_i$, gives an isomorphism
\[ \Hom(V_1 \otimes \dots \otimes V_i [-1] \otimes \dots V_n, V_0) \cong \Hom(V_1 \otimes \dots \otimes V_n,V_0) \]
sending $\phi$ to $\phi'$ given by
\[ \phi'(v_1, \dots v_n) = (-1)^\# \phi(v_1,\dots, v_n) \]
where $\# = \sum_{j=1}^{i-1} \deg(v_j)$; recall also that in the category of dg vector spaces, the differential on the shift $d_{V[1]}$ is given by $-d_V$.

The structure map of the bimodule $A_\Delta$, that is, a morphism $\mu_{A_\Delta}: T(A[1]) \otimes A \otimes T(A[1]) \to A$ is given by the structure map $\mu_A: T(A[1]) \to A[1] $ of the $A_\infty$-algebra, but with the appropriate shift coming from the considerations above:
\[ \mu_{A_\Delta}(a_1,\dots,a_n, a, a'_1,\dots, a'_m) = (-1)^\#  \mu_{A}(a_1,\dots,a_n, a, a'_1,\dots, a'_m), \]
with $\# = \bar a_1 + \dots \bar a_n + 1$, and on the right-hand side $a$ is seen as an element of $A[1]$.

We now describe two other canonical objects in the category $A\mh\Mod\mh A$, the \emph{linear dual bimodule} $A^\vee_\Delta$ and the \emph{inverse dualizing bimodule} $A^!_\Delta$.

The linear dual bimodule $A^\vee_\Delta$, as a graded vector space, is given by the  dual $\Hom_\kk(A,\kk)$, and it has structure map $\mu_{A^\vee_\Delta}: T(A[1]) \otimes A^\vee \otimes T(A[1]) \to A^\vee$ given by
\[ \mu_{A^\vee_\Delta}(a_1,\dots, a_m,a^\vee,a'_1,\dots,a'_n)(a) =
(-1)^\# a^\vee \left( \mu_A(a'_1,\dots,a'_n, a, a_1,\dots,a_m) \right) \]
where $\# = (\bar a_1 + \dots \bar a_m + \deg(a^\vee))\cdot (\bar a'_1 + \dots + \bar a'_n + \deg(a)) + \bar a'_1 + \dots + \bar a'_n$.
The proof that these maps satisfy the structure conditions can be obtained by adapting the proof of \cite[Lem.3.9]{tradler2007infinity} to our sign conventions.

Using the two-sided tensor product of bimodules, we can extend the canonical evaluation $A \otimes A^\vee \to \kk$ to an evaluation morphism
\[ \ev_A: A_\Delta^\vee \otimes_{A\mh A} A_\Delta \to \kk \]
which satisfies the following graphical equation in $\Hom_\kk(A_\Delta^\vee \otimes_{A\mh A} A_\Delta, \kk)$:
\[
\begin{tikzpicture}[baseline={([yshift=-.5ex]current bounding box.center)}]
\node (toplabel) at (0,3.6) {$A^\vee_\Delta$};
\node (botlabel) at (0,0) {$A_\Delta$};
\node [vertex] (top) at (0,2.4) {$\mu_{A^\vee_\Delta}$};
\node [vertex] (bot) at (0,1.2) {$\ev_A$};
\draw [->-=0.9,very thick] (toplabel) to (top);
\draw [->-=0.9, very thick] (top) to (bot);
\draw [->-=0.9, very thick] (botlabel) to (bot);
\end{tikzpicture} \quad + \quad
\begin{tikzpicture}[baseline={([yshift=-.5ex]current bounding box.center)}]
\node (toplabel) at (0,3.6) {$A^\vee_\Delta$};
\node (botlabel) at (0,0) {$A_\Delta$};
\node [vertex] (top) at (0,2.4) {$\ev_A$};
\node [vertex] (bot) at (0,1.2) {$\mu_{A_\Delta}$};
\draw [->-=0.9,very thick] (toplabel) to (top);
\draw [->-=0.9, very thick] (bot) to (top);
\draw [->-=0.9, very thick] (botlabel) to (bot);
\end{tikzpicture} = 0
\]

We use the linear dual bimodule when $A$ is compact, that is, when $H^*(A,\mu_A)$ is finite-rank as a vector space. Let $M$ be any compact object of $A\mh\Mod\mh A$, and consider some morphism $F \in \Hom_{A\mh A}(M, A^\vee_\Delta)$. Let $\tilde F$ be the map of graded vector spaces $M \otimes_{A\mh A} A_\Delta \to \kk$ given by
\[ \tilde F(a_1,\dots,a_k,m,a'_1,\dots,a'_l, a) = \ev(F(a_1,\dots,a_k,m,a'_1,\dots,a'_l), a) \]
Recall that $M \otimes_{A\mh A} A_\Delta$ is a model for the Hochschild chain complex $C_*(A,M)$ of $A$ with coefficients in the bimodule $M$. We can express the relevant tensor-hom adjunction using bimodules as follows.
\begin{lemma} \label{lemma:IsoCompact}
	When $A$ and $M$ are compact, the map $\Hom_{A\mh A}(M,\cA_\Delta^\vee) \to \Hom_\kk(M \otimes_{A\mh A} A_\Delta, \kk)$ given by $F \mapsto \tilde F$ is a quasi-isomorphism.
\end{lemma}
\begin{proof}
	Follows straightforwardly from tensor-hom adjunction for $A_\infty$-bimodules \cite[Prop.2.10]{ganatra2013symplectic} together with the fact that the double dual of $\cA$ is quasi-equivalent to $\cA$ when $\cA$ is compact.
\end{proof}

Graphically, the lemma above says that the following local replacement
\[
\begin{tikzpicture}[baseline={([yshift=-.5ex]current bounding box.center)}]
\node (toplabel) at (0,3) {$M$};
\node (botlabel) at (0,0) {$A_\Delta^\vee$};
\node [vertex] (top) at (0,1.5) {$F$};
\draw [->-=0.6,very thick] (toplabel) to (top);
\draw [->-=0.6, very thick] (top) to (botlabel);
\end{tikzpicture} \qquad \rightsquigarrow \qquad
\begin{tikzpicture}[baseline={([yshift=-.5ex]current bounding box.center)}]
\node (toplabel) at (0,3) {$M$};
\node (botlabel) at (0,0) {$A_\Delta$};
\node [vertex] (top) at (0,1.5) {$\tilde F$};
\draw [->-=0.6,very thick] (toplabel) to (top);
\draw [->-=0.6, very thick] (botlabel) to (top);
\end{tikzpicture} \quad := \quad
\begin{tikzpicture}[baseline={([yshift=-.5ex]current bounding box.center)}]
\node (toplabel) at (0,3) {$M$};
\node (botlabel) at (0,0) {$A_\Delta$};
\node [vertex] (top) at (0,2) {$F$};
\node [vertex] (bot) at (0,1) {$\ev$};
\draw [->-=0.9,very thick] (toplabel) to (top);
\draw [->-=0.9, very thick] (top) to (bot);
\draw [->-=0.9, very thick] (botlabel) to (bot);
\end{tikzpicture}
\]
of a part of some larger diagram induces an quasi-isomorphism.

We also have the \emph{inverse dualizing bimodule} $A_\Delta^!$, defined following Ginzburg's definition in \cite{ginzburg2006calabiyau} for the dg case. We recall its definition in the case of an $A_\infty$-algebra; the $A_\infty$-category case is analogous, and the relevant formulas can be found in \cite{ganatra2013symplectic}. As a graded vector space, $A_\Delta^!$ is given by
\[ A_\Delta^! = \Hom_\kk(T(A[1]) \otimes A \otimes T(A[1]), A \otimes A) \]
with structure maps given schematically by
\[ \mu^{r|1|0} =
\begin{tikzpicture}[baseline={([yshift=-.5ex]current bounding box.center)}]
\node (toplabel) at (0,3) {$A$};
\node (botlabel) at (0,0) {$A \otimes A$};
\node [vertex] (mid) at (0,2) {$F$};
\draw [->-=0.6,very thick] (toplabel) to (mid);
\draw [->-] (0.17,1.8) to (0.17,0.2);
\draw [->-] (-0.17,1.8) to (-0.17,0.2);
\end{tikzpicture} \quad \rightsquigarrow \quad
\begin{tikzpicture}[baseline={([yshift=-.5ex]current bounding box.center)}]
\node (toplabel) at (0,3) {$A$};
\node (botlabel) at (0,0) {$A \otimes A$};
\node [vertex] (mid) at (0,2) {$F$};
\node [vertex] (mu) at (-0.17,1) {$\mu$};
\draw [->-=0.9,very thick] (toplabel) to (mid);
\draw [->-] (0.17,1.8) to (0.17,0.2);
\draw [->-] (-0.17,1.8) to (mu);
\draw [->-] (mu) to (-0.17, 0.2);
\end{tikzpicture}, \qquad
\mu^{0|1|s} =
\begin{tikzpicture}[baseline={([yshift=-.5ex]current bounding box.center)}]
\node (toplabel) at (0,3) {$A$};
\node (botlabel) at (0,0) {$A \otimes A$};
\node [vertex] (mid) at (0,2) {$F$};
\draw [->-=0.6,very thick] (toplabel) to (mid);
\draw [->-] (0.17,1.8) to (0.17,0.2);
\draw [->-] (-0.17,1.8) to (-0.17,0.2);
\end{tikzpicture} \quad \rightsquigarrow \quad
\begin{tikzpicture}[baseline={([yshift=-.5ex]current bounding box.center)}]
\node (toplabel) at (0,3) {$A$};
\node (botlabel) at (0,0) {$A \otimes A$};
\node [vertex] (mid) at (0,2) {$F$};
\node [vertex] (mu) at (0.17,1) {$\mu$};
\draw [->-=0.9,very thick] (toplabel) to (mid);
\draw [->-] (-0.17,1.8) to (-0.17,0.2);
\draw [->-] (0.17,1.8) to (mu);
\draw [->-] (mu) to (0.17, 0.2);
\end{tikzpicture}
\]
and $\mu^{r|1|s} = 0$ when $r,s > 0$. \footnote{Note we can analogously define the bimodule dual of any bimodule $M$ by substituting $M$ on top instead of $A$.} This is an $A_\infty$-analog of the definition of the bimodule dual $M^! = \Hom_{A^e}(M, A^e)$ over an associative algebra $A$; and can also be defined by this formula using the formalism of $n$-modules over an $A_\infty$-algebra.

Recall that an $A_\infty$-algebra $A$ is \emph{(homologically) smooth} \cite{kontsevich2009notes} if the diagonal bimodule $A_\Delta$ is perfect as a $(A,A)$-bimodule, i.e. if it is quasi-isomorphic to a direct summand of a finite extension of copies of the bimodule $A \otimes A$.

\begin{lemma}\cite[Prop.2.16]{ganatra2013symplectic} \label{lemma:IsoSmooth}
	If $A$ is homologically smooth and $M$ is perfect, then there is a natural quasi-isomorphism
	\[ A_\Delta^! \otimes_{A\mh A} M \xrightarrow{\simeq} \Hom_{A\mh A}(A_\Delta, M) \]
	for any $(A,A)$-bimodule $M$.
\end{lemma}

For $M = A_\Delta$, we pick an inverse of the quasi-isomorphism above to obtain a distinguished element $\ev^!$ in the graded vector space $A_\Delta^! \otimes_{A\mh A} A_\Delta$. We then have a map $\Hom_{A\mh A}(A_\Delta, M) \to A_\Delta^! \otimes_{A\mh A} A_\Delta$ given graphically by:
\[
\begin{tikzpicture}[baseline={([yshift=-.5ex]current bounding box.center)}]
\node (toplabel) at (0,3.2) {$A_\Delta^!$};
\node (botlabel) at (0,0) {$M$};
\node [vertex] (mid) at (0,1.5) {$F$};
\draw [->-=0.6,very thick] (toplabel) to (mid);
\draw [->-=0.6, very thick] (mid) to (botlabel);
\end{tikzpicture} \qquad \rightsquigarrow \qquad
\begin{tikzpicture}[baseline={([yshift=-.5ex]current bounding box.center)}]
\node (toplabel) at (0,3.2) {$A_\Delta$};
\node (botlabel) at (0,0) {$M$};
\node [vertex] (mid) at (0,1.5) {$\tilde F$};
\draw [->-=0.6,very thick] (mid) to (toplabel);
\draw [->-=0.6, very thick] (mid) to (botlabel);
\end{tikzpicture} \quad := \quad
\begin{tikzpicture}[baseline={([yshift=-.5ex]current bounding box.center)}]
\node (toplabel) at (0,3.2) {$A_\Delta$};
\node (botlabel) at (0,0) {$M$};
\node [vertex] (top) at (0,2.1) {$\ev^!$};
\node [vertex] (bot) at (0,1) {$F$};
\draw [->-=0.9,very thick] (top) to (toplabel);
\draw [->-=0.9, very thick] (top) to (bot);
\draw [->-=0.9, very thick] (bot) to (botlabel);
\end{tikzpicture}
\]
which gives a quasi-inverse to the map in the Lemma above.

\subsubsection{Compact and smooth Calabi-Yau structures}\label{sec:CYstructures}
We now recall two notions of Calabi-Yau structures on $A_\infty$-categories.

\begin{definition}
	If $\cA$ is compact, a dual cycle $\theta: C_*(\cA) \to \kk[-d]$ is a \emph{weak compact Calabi-Yau structure} of dimension $d$ on $\cA$ if it maps to an isomorphism of bimodules $\cA_\Delta \to \cA_\Delta^\vee$ under the map of \cref{lemma:IsoCompact}.

	If $\cA$ is smooth, a cycle $\omega: \kk[d] \to C_*(\cA)$ is a \emph{weak smooth Calabi-Yau structure} of dimension $d$ on $\cA$ if it maps to an isomorphism of bimodules $\cA_\Delta^! \to \cA_\Delta$ under the inverse to the map in \cref{lemma:IsoSmooth}.
\end{definition}
See e.g. \cite{ginzburg2006calabiyau,kontsevich2009notes,brav2019relative,ganatra2019cyclic}.

There is a canonical $S^1$-action on the Hochschild complex, whose homotopy fixed points are modeled by the negative cyclic complex $CC^-_*(\cA) = (C_*(\cA)[[u]], b+uB)$ and whose homotopy orbits are modeled by the (positive) cyclic complex $CC_*(\cA) = (C_*(A)[u,u^{-1}]/\kk[u], b+uB)$; there are canonical maps \[ CC^-_*(\cA) \to C_*(\cA), \qquad C_*(\cA) \to CC_*(\cA) \]

\begin{definition}\label{def:strongCYstructure}
	A \emph{(strong) compact Calabi-Yau structure} is a lift of a weak compact Calabi-Yau structure to a dual class in (positive) cyclic homology
	\[ CC_*(\cA) \to \kk[-d], \]
	and a \emph{(strong) smooth Calabi-Yau structure} is a lift of a weak smooth Calabi-Yau structure to a class in negative cyclic homology
	\[ \kk[d] \to CC^-_*(\cA). \]
\end{definition}

\subsection{Cyclic A-infinity structures}\label{sec:cyclicAinfty}
There is another notion of Calabi-Yau structure on an $A_\infty$-category, which is closely related with the definition of compact Calabi-Yau structure given above; this was defined in \cite{kontsevich2009notes} under the name of $A_\infty$-algebra/category with nondegenerate scalar product.

\begin{definition}\label{def:cyclicCY}
	A \emph{cyclic $A_\infty$-structure} of degree $d$ on an $A_\infty$-category $\cA$ is a collection of (chain-level) nondegenerate $\kk$-linear pairings
	\[ \langle -,-\rangle: \cA(X,Y) \otimes \cA(Y,X) \to \kk[-d] \]
	for any objects $X,Y$ of $\cA$, such that
	\[ \langle \mu^n(a_1,\dots,a_n), a_{n+1} \rangle =  (-1)^{\bar a_1 + 1} \langle a_1, \mu^n(a_2,\dots,a_n,a_0,\dots,a_{n+1}) \rangle \]
	for any collection of objects $X_1,\dots,X_n, X_{n+1} = X_0$ and morphisms $a_i \in \hom(X_i,X_{i+1})$.
\end{definition}
Note that nondegeneracy of the pairing implies that $A$ is finite-dimensional. A cyclic $A_\infty$-structure should be seen a strictification of a compact CY structure, in the sense that given any compact CY structure on a compact $A_\infty$-algebra, one can find a quasi-isomorphic finite-dimensional cyclic $A_\infty$-algebra; we discuss more about this equivalence in \cref{sec:ncLagrangian}, see also \cite{ganatra2019cyclic}.

\subsubsection{The necklace bracket for cyclic A-infinity algebras}
Recall that the data of an $A_\infty$-structure $\mu$ on $\cA$ is given by a solution to a Maurer-Cartan equation on the space $C^*(\cA)$ of its Hochschild cochains. We now explain an analogous description of cyclic $A_\infty$-structures.

Recall that the space of Hochschild chains of $\cA$ is defined as $C_*(\cA) = \prod_{s \ge 1} C_*(\cA)^s$, where
\[ C_*(\cA)^s = \prod_{X_0,\dots, X_s \in \Ob(\cA)} \cA(X_0,X_1)[1] \otimes \dots \otimes \cA(X_{s-1},X_s)[1] \otimes \cA(X_s,X_0)[1] [-1] \]
This space has a $\ZZ_s$ action rotating the factors of $\cA(\dots)[1]$ with a Koszul sign; let us denote by $C^{cyc}_*(\cA)^s = \left( C_*(\cA)^s \right)^{\ZZ_s}$, and $C^{cyc}_*(\cA) = \bigoplus_{s \ge 1} CC_*(\cA)^s$.

\begin{definition}
	Let $\cA$ have a pairing $\langle-,-\rangle$ of degree $d$. The \emph{necklace bracket} of $(\cA,\langle-,-\rangle)$ is the map
	\[ [-,-]_\nec: C^{cyc}_*(\cA)[1] \otimes C^{cyc}_*(\cA)[1] \to C^{cyc}_*(\cA)[1] \]
	defined by summing up over pairings, using $<,>$, of two cyclic words in elements of $A$, in all possible ways. Explicitly,
	\begin{align*} 
	[(a_0 &\otimes \dots \otimes a_r)_\mathrm{cyc}, (b_0 \otimes \dots \otimes b_s)_\mathrm{cyc}]_\nec = \\
	&\sum_{i,j} (-1)^\# \langle a_i,b_j \rangle (a_0 \otimes \dots\otimes a_{i-1},b_{j+1},\otimes \dots \otimes b_s \otimes b_0 \otimes \dots \otimes b_{j-1} \otimes a_{i+1} \otimes \dots \otimes a_{r})_\mathrm{cyc},
	\end{align*}
	where $(-)_\mathrm{cyc}$ denotes we take the sum over all cyclic permutations, and $(-1)^\#$ is the Koszul sign associated with the permutation in each term.
\end{definition}

We then have the following characterization of a cyclic $A_\infty$-structure.
\begin{lemma}
	If $\langle -,-\rangle$ is non-degenerate, there is an equivalence between the data of a cyclic $A_\infty$-structure on $\cA$ and the data of a solution $\omega \in C^{cyc}_*(\cA)$ of homogeneous degree $2$ of the Maurer-Cartan equation $[\omega,\omega]_\nec = 0$.
\end{lemma}
\begin{proof}
	The equivalence between the $A_\infty$-structure maps and the solution $\omega$ is given by dualizing the inputs:
	\[
	\begin{tikzpicture}[baseline={([yshift=-.5ex]current bounding box.center)}]
	\node [vertex] (mu) at (0,-0.2) {$\mu^s$};
	\draw [->-] (mu) to (0,-1);
	\draw [->-] (-1.5,1.6) to (mu);
	\draw [->-] (1.5,1.6) to (mu);
	\node at (0,0.7) {$\dots$};
	\end{tikzpicture} \quad = \quad
	\begin{tikzpicture}[baseline={([yshift=-.5ex]current bounding box.center)}]
	\node [vertex] (omega) at (0,-0.2) {$\omega^s$};
	\node [vertex] (p1) at (-0.8,0.7) {$\langle,\rangle$};
	\node [vertex] (pn) at (0.8,0.7) {$\langle,\rangle$};
	\draw [->-] (omega) to (0,-1);
	\draw [->-] (-1.5,1.6) to (p1);
	\draw [->-=0.8] (omega) to (p1);
	\draw [->-] (1.5,1.6) to (pn);
	\draw [->-=0.8] (omega) to (pn);
	\node at (0,0.7) {$\dots$};
	\end{tikzpicture}
	\]
	i.e., using the pairing on the first $s$ of the $s+1$ outgoing legs of $\omega^s$. The fact that such maps $\mu^s$ satisfy the $A_\infty$-relations and are compatible with the pairing follows from the non-degeneracy of the pairing and the cyclic symmetry of $\omega$.
\end{proof}

\section{Pre-Calabi-Yau algebras and categories}
We now present the main definition of this paper. A pre-CY structure on an $A_\infty$-algebra, or more generally on an $A_\infty$-category, is an extension of its $A_\infty$-structure given by a solution to a Maurer-Cartan equation on a certain dg Lie algebra containing the Hochschild cochains as a subcomplex. This equation is defined by a necklace bracket, generalizing the case of cyclic $A_\infty$-structures we saw in \cref{sec:cyclicAinfty} to algebras/categories that are not finite-dimensional.

\subsection{Higher Hochschild cochains}\label{sec:higherHochschildCohomology}
For any integer $\ell \ge 1$, and any $\ell$-tuple of non-negative integers $n_1,\dots,n_\ell$, we denote by $\{X^i_{j_i}\}$ a collection of $\sum (n_i+1)$ objects of $\cA$, indexed by $i= 1,\dots,\ell$ and $j_i = 0,\dots, n_i$.

We then sum over all such collections to define a graded vector space
\[ \hspace{-1.5cm} C^*_\ell(\cA;n_1,\dots,n_\ell) := \prod_{\{X^i_{j_i}\}} \Hom\left( \bigotimes_{i=1}^\ell (\cA(X^i_0, X^i_1)[1] \otimes \dots \otimes \cA(X^i_{n_i - 1}, X^i_{n_i})[1]), \ \bigotimes_{i=1}^\ell \cA(X^i_0,X^{i-1}_{n_i}) \right)  \]
Note the indices in the target of the $\Hom$ above. To make sense of this when $i=0$, we set $X^{0}_j = X^{\ell}_j$ and $n_0=n_\ell$.

\begin{example}
	Let us give an example to demonstrate how to organize the objects $X^i_j$. Consider the case $\ell = 2$ with $n_1=2, n_2 = 1$. We then pick five objects $X^1_0,X^1_1,X^1_2,X^2_0,X^2_1$, and organize them around the circle in clockwise fashion.
	
	We can depict an element $\phi \in C^*_{(3)}$ as a vertex inside of this circle which takes in elements of $\cA(X^i_0,X^i_{1}),\dots$ and then outputs an element of $\cA(X^2_0,X^1_2) \otimes \cA(X^1_0,X^2_1)$.
	\[
	\begin{tikzpicture}[baseline={([yshift=-.5ex]current bounding box.center)}]
	\node [vertex] (a) at (0,0) {$\phi$};
	\draw [->-] (a) to (0,1);
	\draw [->-] (a) to (0,-1);
	\draw [->-] (-1,0.6) to (a);
	\draw [->-] (-1,-0.6) to (a);
	\draw [->-] (1,0) to (a);
	\node at (-0.4,-0.7) {$X^1_0$};
	\node at (-0.6,0) {$X^1_1$};
	\node at (-0.4,+0.7) {$X^1_2$};
	\node at (0.5,0.5) {$X^2_0$};
	\node at (0.5,-0.5) {$X^2_1$};
	\node at (0,-1.3) {$\cA(X^1_0,X^2_1)$};
	\node at (0,1.3) {$\cA(X^2_0,X^1_2)$};
	\node at (-2,-0.6) {$\cA(X^1_0,X^1_1)$};
	\node at (-2,0.6) {$\cA(X^1_1,X^1_2)$};
	\node at (2,0) {$\cA(X^2_0,X^2_1)$};
	\end{tikzpicture}\]
	The objects $X^i_j$ label regions around the vertex, and each arrow with an $X$ region to its right and a $Y$ region to its left carries an element of $\cA(X,Y)$.
\end{example}

\begin{definition}\label{def:HigherHochschildChains}
	For any integer $\ell \ge 1$, the space of $\ell$-higher Hochschild cochains is the graded vector space
	\[ C^*_{(\ell)}(\cA) := \prod_{n_1,\dots,n_\ell} C^*_\ell(\cA;n_1,\dots,n_\ell)  \]
\end{definition}
Note that when $\ell = 1$ this is the usual space of Hochschild cochains, and when $\cA = A$ is some $A_\infty$-algebra we have
\[ C_{(\ell)}^*(A) = \prod_{n_1,\dots, n_k} \Hom_\kk\left(A[1]^{\otimes n_1} \otimes \dots \otimes A[1]^{\otimes n_k}, A^{\otimes k}\right). \]

\begin{remark}
	There is another concept of generalized Hochschild invariants which sometimes appears with this same name, described by \cite{pirashvili2000hodge} under the name of `higher order Hochschild homology'. This is a distinct notion from our definitions.
\end{remark}

\subsubsection{Graphical representation of higher Hochschild cochains}\label{sec:graphicalhigher}
We now extend the graphical calculus for (ordinary) Hochschild cochains to include higher Hochschild cochains; this will allow us to evaluate any directed tree with $k$ outgoing arrows, and higher Hochschild cochains in each vertex, into an element of $C^*_(k)(\cA)$. 

In analogy with usual Hochschild cochains, we visualize an element $\phi \in C^*_{(k)}(A)$ as a vertex drawn on a plane, with $k$ arrows coming out of the vertex. As an example, a cochain $\phi \in C^*_{(3)}(A)$ is drawn as a vertex:
\[
\begin{tikzpicture}[baseline={([yshift=-.5ex]current bounding box.center)}]
\node [vertex] (a) at (0,0) {$\phi$};
\draw [->-] (a) to (1,0.6);
\draw [-w-] (a) to (0,-1);
\draw [->-] (a) to (-1,0.6);
\end{tikzpicture}\]
with the white arrow marking the first factor of the (output) tensor product, and the other factors are read from the outgoing arrows in \emph{clockwise} direction. The vertex above represents a collection of maps
\[ \phi^{(n_1,n_2,n_3)} : A[1]^{\otimes n_2} \otimes A[1]^{\otimes n_2} \otimes A[1]^{\otimes n_3} \to A[1] \otimes A[1] \otimes A[1] \]
(note the shifts on the outputs) for all choices of $n_i \ge 0$. We visualize the input factors as arrows incoming into the vertex, starting next to the white arrow, and the outputs as arrows going out of the vertex, starting on the white arrow, all clockwise. For example, if the cochain $\phi$ evaluates
\[ \phi^{(2,1,3)}(a^1_1,a^1_2;\;  a^2_1;\; a^3_1, a^3_2, a^3_3) = (b_1,b_2,b_3),  \]
we visualize this operation as
\[\begin{tikzpicture}
\node [vertex] (phi) at (0,0) {$\phi$};
\node (a11) at (-.7,-1) {$a^1_1$};
\node (a12) at (-1.2,-0) {$a^1_2$};
\node (a21) at (-0,1.3) {$a^2_1$};
\node (a31) at (1.4,0.2) {$a^3_1$};
\node (a32) at (1.3,-0.6) {$a^3_2$};
\node (a33) at (0.7,-1) {$a^3_3$};
\node (b1) at (0,-1.5) {$b_1$};
\node (b2) at (-1.5,0.9)  {$b_2$};
\node (b3) at (1.5,0.9) {$b_3$};

\draw [-w-] (phi) to (b1);
\draw [->-] (phi) to (b2);
\draw [->-] (phi) to (b3);
\draw [->, shorten <=-2mm, shorten >=2mm] (a11) to (phi);
\draw [->, shorten <=-2mm, shorten >=2mm] (a12) to (phi);
\draw [->, shorten <=-1mm, shorten >=2mm] (a21) to (phi);
\draw [->, shorten <=-2mm, shorten >=2mm] (a31) to (phi);
\draw [->, shorten <=-1mm, shorten >=2mm] (a32) to (phi);
\draw [->, shorten <=-2mm, shorten >=2mm] (a33) to (phi);
\end{tikzpicture}\]
This notation extends by linear combination to inputs/outputs that are general tensors, and also to $A_\infty$-categories; in that case we must label the regions between the arrows by objects of the category.

We now describe how to evaluate a directed tree diagram composed of vertices like the one above, extending the procedure from \cref{sec:graphical}. For example, for higher cochains
\[ \phi \in C^*_{(3)}(A), \quad \psi \in C^*_{(2)}(A), \quad \lambda \in C^*_{(3)}(A) \]
we would like to interpret a diagram such as
\[\begin{tikzpicture}[baseline={([yshift=-.5ex]current bounding box.center)}]
\draw (0,0) circle (2);
\node [vertex] (phi) at (-0.9,0.8) {$\phi$};
\node [vertex] (psi) at (0,0) {$\psi$};
\node [vertex] (lambda) at (0,-1) {$\lambda$};
\draw  [->-] (phi) to (-2.2,-0.1);
\draw [->-] (phi) to (-0.6,+2.2);
\draw [->-] (phi) to (psi);
\draw [->-] (psi) to (+1.7,1.55);
\draw [->-] (psi) to (lambda);
\draw [->-] (lambda) to (0,-2.2);
\draw [->-] (lambda) to (-1.7,-1.55);
\draw [->-] (lambda) to (2.4,0.1);
\draw [-w-=1] (lambda) to (2.4,0.1);
\end{tikzpicture}\]
as giving an element in $C^*_{(6)}(A)$. The white arrow to the right denotes that we want to read the output of this diagram starting clockwise from there.

We again must order our diagram, by choosing an ordering of the vertices compatible with the partial order given by the arrows. For example, on the diagram above we can choose the ordering $(\lambda \psi \phi)$ (in ascending order as before).

We interpret the diagram above as giving maps
\[ A[1]^{\otimes n_1} \otimes \dots \otimes A[1]^{\otimes n_6} \to A^{\otimes 6} \]
as follows:
\begin{enumerate}
	\item We draw incoming lines around the circle for each of the factors in the source, starting right after the marked white arrow tip, then connect them to the internal vertices of the tree in all possible ways \emph{without crossing}. For example, here is one such diagram:
\[\begin{tikzpicture}[baseline={([yshift=-.5ex]current bounding box.center)}]
\draw (0,0) circle (2);
\node [vertex] (phi) at (-0.9,0.8) {$\phi$};
\node [vertex] (psi) at (0,0) {$\psi$};
\node [vertex] (lambda) at (0,-1) {$\lambda$};

\draw  [->-] (phi) to (-2.2,-0.1);
\draw [->-] (phi) to (-0.6,+2.2);
\draw [->-] (phi) to (psi);
\draw [->-] (psi) to (+1.7,1.55);
\draw [->-] (psi) to (lambda);
\draw [->-] (lambda) to (0,-2.2);
\draw [->-] (lambda) to (-1.7,-1.55);
\draw [->-] (lambda) to (2.4,0.1);
\draw [-w-=1] (lambda) to (2.4,0.1);

\draw [->-,bend left=20] (0,2) to (phi);
\draw [->-,bend left=20] (0.5,1.945) to (phi);
\draw [->-,bend right=20] (1,1.75) to (psi);
\draw [->-,bend left=20] (1.88,0.7) to (psi);
\draw [->-,bend left=10] (1.88,-0.7) to (lambda);
\draw [->-,bend left=40] (1.05,-1.7) to (lambda);
\draw [->-,bend right=20] (-1.05,-1.7) to (lambda);
\draw [->-,bend left=10] (-1.75,-1) to (lambda);
\draw [->-] (-1.95,-0.5) to (psi);
\draw [->-,bend right=30] (-1.5,1.33) to (phi);
\node at (2.15,-0.7) {$a^1_1$};
\node at (1.2,-1.9) {$a^1_2$};
\node at (-1.2,-2) {$a^2_1$};
\node at (-2.05,-1.1) {$a^3_1$};
\node at (-2.15,-0.5) {$a^3_2$};
\node at (-1.7,1.55) {$a^4_1$};
\node at (0,2.3) {$a^5_1$};
\node at (0.6,2.2) {$a^5_2$};
\node at (1.2,2) {$a^5_3$};
\node at (2.15,0.8) {$a^6_1$};
\end{tikzpicture}\]

\item We then write the vertices in their order next to the inputs:
\[ (\lambda\psi\phi)(a^1_1 a^1_2 \dots a^6_1)  \]

\item Permute the entries with Koszul signs until the inputs of the last vertex are next to it, in order:
\[ (-1)^\# (\lambda\psi\phi)(a^4_1 a^5_1 a^5_2 a^1_1 \dots a^6_1) \]
for the appropriate sign $\#$.

\item Evaluate the vertex and then write its outputs, also in order. Note that in general the output is not a simple tensor, but instead a combination, such as
\[ \phi(\varnothing; a^4_1; a^5_1, a^5_2) = b_1 \otimes b_2 \otimes b_3 + b'_1 \otimes b'_2 \otimes b'_3 \]
so we sum over all terms. So we have
\[ (-1)^\# \left((\lambda\psi)(b_1 b_2 b_3 a^1_1 \dots  a^6_1) + (\lambda\psi)(b'_1 b'_2 b'_3 a^1_1 \dots  a^6_1) + \dots \right) \]

\item Repeat steps (3) and (4) until all vertices are gone.
\end{enumerate}

Note that in the drawing above we marked the rightmost outgoing arrow with a white arrowhead; this indicates the first factor of $A$ in the output $A^{\otimes 6}$. After the steps above we ended up with 6 elements of $A[1]$; so we  permute them with Koszul sign to match the outgoing order, going clockwise and starting from the white arrow, and then applying an overall shift to read the output in $A^{\otimes 6}$.

By construction, the description above matches what we already established for ordinary Hochschild cochains, when all vertices have exactly one outgoing arrow.

\subsubsection{Cyclic actions}
We now describe cyclic actions on spaces of higher Hochschild cochains. Recall that an element $\phi \in C^*_{(k)}(\cA)$ is a collection of maps given by
\[ \phi(a^1_1, \dots, a^1_{n_1}; \dots; a^k_1,\dots, a^k_{n_k}) = b_1 \otimes \dots b_k + b'_1 \otimes \dots b'_k + \dots \]
We now choose an integer $d$ and write $t$ for the generator $1$ in the cyclic group $\ZZ_k$ of order $k$.
\begin{definition}\label{def:actionDimensionD}
	The action of dimension $d$ of the cyclic group $\ZZ_k$ on $C^*_{(k)}$ is given by
	\[ \hspace{-1cm} (t \phi)(a^2,\dots, a^2_{n_2}; \dots;a^k_1,\dots, a^k_{n_k}; a^1_1, \dots, a^1_{n_1}) = (-1)^{\#_a} (-1)^{(d-1)(k-1)}\left((-1)^{\#_b} b_2 \otimes b_3 \otimes b_3 + \dots \right) \]
	where $\#_a = (\bar a^1_1 + \dots + \bar a^1_{n_1})(\bar a^2_1 +\dots + \bar a^k_{n_k})$ and $\#_b = \bar b_1(\bar b_2 + \dots + \bar b_k)$ are the Koszul signs for permuting the factors of $a$ and $b$, seen as elements of $\cA[1]$. That is, the action of dimension $d$ has an extra sign of $(d-1)(k-1)$ compared to the usual Koszul sign.
\end{definition}

We denote the action of dimension $d$ by $(\ZZ_k,d)$. It only depends on the parity of $d$.
\begin{definition}\label{def:cyclicCochains}
	The space of \emph{cyclic $k$-cochains} of dimension $d$ on $\cA$ is defined as
	\[ C_{(k, d)}^*(\cA) := (C_{(k)}^*(\cA))^{(\ZZ_k,d)}[(d-2)(k-1)] \]
	that is, the $(d-2)(k-1)$ shift of the higher Hochschild cochains that are invariant under the action of dimension $d$.

	We assemble all these spaces into the \emph{tangent complex}
	\[ C_{[d]}^*(\cA) := \prod_{k \ge 1} C_{(k, d)}^*(\cA) \]
\end{definition}

Now that we have all these different complexes, related by shifts, let us define some notation that will simplify the calculation of signs later.
\begin{definition}
	Given a cochain $\phi \in C_{(k, d)}^*(\cA)$, we denote by
	\begin{itemize}
		\item $\deg(\phi)$ its degree in $C^*_{(k)}(\cA)$, i.e. as a map from copies of $\cA[1]$ to copies of $\cA$.
		\item $\bar \phi$ its degree as a map from copies of $\cA[1]$ to copies of $\cA[1]$.
		\item $|\phi|$ its degree in $C_{(k, d)}^*(\cA)$, or equally, in $C_{[d]}(\cA)$.
	\end{itemize}
	These degrees are related by $\deg(\phi) = \bar \phi + k = |\phi| + (d-2)(k-1)$.
\end{definition}

\subsubsection{The necklace bracket}
We now use the graphical notation to define some operations on higher cyclic cochains. Let $\phi \in C^*_{(k,d)}(\cA), \psi \in C^*_{(\ell,d)}(\cA)$ be two higher cyclic cochains of dimension $d$ on $\cA$.

\begin{definition}\label{def:NecklaceProduct}
	The \emph{necklace product} $\phi \circnec \psi$ is the element of $C^*_{(k+\ell-1)}(\cA)$ given by the following expression:
	\[\hspace{-1cm} \phi \circnec \psi = \sum_{n=1}^{k} (-1)^{r_n} \begin{tikzpicture}[baseline={([yshift=-.5ex]current bounding box.center)}]
	\draw (0,0) circle (1.5);
	\node [vertex] (psi) at (0,0.6) {$\psi$};
	\node [vertex] (phi) at (0,-0.5) {$\phi$};
	\draw [-w-=1] (phi) to (0,-1.8);
	\draw [-w-] (phi) to (0,-1.5);
	\node at (0,-2) {$1$};
	\draw [->-] (phi) to (-1.4,-1);
	\node at (-1.5,-1.1) {$2$};
	\draw [->-] (phi) to (-1.6,0.5);
	\node at (-1.8,0.55) {$n$};
	\node at (-0.8,-0.3) {$\vdots$};
	\draw [->-] (psi) to (-1.1,1.1);
	\node at (-1.6,1.2) {$n+1$};
	\node at (0,1.1) {$\dots$};
	\draw [->-] (psi) to (1.1,1.1);
	\node at (2,1.2) {$n+\ell-1$};
	\draw [->-] (phi) to (1.4,-1);
	\node at (2.2,-1.1) {$k+\ell-1$};
	\node at (0.8,-0.3) {$\vdots$};
	\draw [->-] (phi) to (1.6,0.5);
	\node at (2.1,.55) {$n+\ell$};
	\draw [-w-] (psi) to (phi);
	\end{tikzpicture} + \sum_{m=1}^{k-1} (-1)^{s_n}
	\begin{tikzpicture}[baseline={([yshift=-.5ex]current bounding box.center)}]
	\draw (0,0) circle (1.5);
	\node [vertex] (phi) at (0,0.6) {$\phi$};
	\node [vertex] (psi) at (0,-0.5) {$\psi$};
	\draw [-w-=1] (psi) to (0,-1.8);
	\draw [-w-] (psi) to (0,-1.5);
	\node at (0,-2) {$1$};
	\draw [->-] (psi) to (-1.4,-1);
	\node at (-1.5,-1.1) {$2$};
	\draw [->-] (psi) to (-1.6,0.5);
	\node at (-1.8,0.55) {$m$};
	\node at (-0.8,-0.3) {$\vdots$};
	\draw [-w-] (phi) to (-1.1,1.1);
	\node at (-1.6,1.2) {$m+1$};
	\node at (0,1.1) {$\dots$};
	\draw [->-] (phi) to (1.1,1.1);
	\node at (1.6,1.2) {$m+\ell$};
	\draw [->-] (psi) to (1.4,-1);
	\node at (2.2,-1.1) {$k+\ell-1$};
	\node at (0.8,-0.3) {$\vdots$};
	\draw [->-] (psi) to (1.6,0.5);
	\node at (2.5,.55) {$m+\ell+1$};
	\draw [->-] (psi) to (phi);
	\end{tikzpicture}\]
	where the sign exponents are given by
	\[ r_n = (\ell -1)(d-1)(k-n + |\phi|+1), \qquad s_m = (k-1)(d-1)(|\phi|+1) + (\ell-1)(d-1)n \]
\end{definition}
Intuitively, the necklace product $\phi \circ \psi$ is given by placing $\psi$ in all possible ways around $\phi$ (making a `necklace') connecting $\psi \to \phi$, summing over all possibilities with appropriate signs.

\begin{example}
	Let us discuss the cases where $k$ or $\ell = 1$. When $k = 1$, i.e., $\phi$ is an ordinary Hochschild chain, we sum over putting $\phi$ along all the $\ell$ outgoing arrows of $\psi$, always with the same sign $(-1)^{(\ell-1)(d-1)(|\phi|+1)}$.

	When $\ell = 1$, then we get a sum over all ways of putting $\psi$ in the regions around $\phi$, all with sign $+1$. Finally, when $k=\ell=1$ this is just the ordinary Gerstenhaber product.
\end{example}

Note that the necklace product has degree $-1$, since it involves interpreting an output (element in $\cA(X,Y)$) as an input (element in $\cA(X,Y)[1]$).
\begin{definition}\label{def:NecklaceBracket}
	The \emph{necklace bracket} of dimension $d$ is the map
	\[ [-,-]_\nec: C^*_{(k,d)}(\cA)[1] \otimes C^*_{(\ell,d)}(\cA)[1] \to C^*_{(k+\ell-1,d)}(\cA)[1] \]
	defined by $[\phi,\psi]_\nec = \phi \circnec \psi - (-1)^{(|\phi|-1)(|\psi|-1)} \psi \circnec \phi$.
\end{definition}
When restricted to ordinary Hochschild cochains $C^*(\cA)$, this gives the usual notion of Gerstenhaber bracket.

The following proposition could be proven by doing an explicit computation of the signs involved. However, there is a more conceptual way of organizing the signs which we will discuss in \cref{sec:action}, so we postpone the proof until then.
\begin{proposition}\label{prop:necklaceDGLie}
	The map defined above does land in $C^*_{(k+\ell-1,d)}(\cA)[1]$, that is, its image satisfies the appropriate cyclic invariance under $\ZZ_{k +\ell-1}$, and also gives $C^*_{[d]}(\cA)[1]$ the structure of a dg Lie algebra.
\end{proposition}

\subsection{Pre-Calabi-Yau structures}
Using the necklace bracket we now come to the main definition of this paper.
\begin{definition}
	A pre-Calabi Yau structure of dimension $d$ on $\cA$ is an element
	\[ m = \sum_{k = 1}^\infty m_{(k)} \in C_{[d]}^*(\cA) \]
	of degree $|m|=2$ (that is, of degree 1 in the dg Lie algebra $C_{[d]}^*(\cA)[1]$) solving the Maurer-Cartan equation $m \circnec m = 0$.
\end{definition}

We will say that $\cA$ is a pre-CY category to mean that there exists a pre-CY structure $m$ as above. Restricting the equation $m \circnec m = 0$ to the component $C^*_{(1)}(\cA)[1] = C^*(\cA)[1]$ gives the equation $m_{(1)} \circ m_{(1)} = 0$, whose solution $\mu = m_{(1)}$ is an $A_\infty$ structure on $\cA$.

For concreteness, let us repeat the definition above in more detail, for the case of an $A_\infty$-algebra $A$. The data of a pre-CY structure of dimension $d$ on $A$ is then a collection of maps
\[ m_{(k)}^{n_1,\dots,n_k}: A[1]^{\otimes n_1} \otimes \dots \otimes A[1]^{\otimes n_k} \to  A^{\otimes k} \]
of degree $d k - d - 2k +4$, cyclically invariant or anti-invariant (depending on the parity of $(k-1)(d-1)$), satisfying
\[ \sum_{k + \ell = n+1} m_{(k)} \circnec m_{(\ell)} = 0 \]
for every $n \ge 1$.

\subsubsection{Unitality}
Recall that if $(\cA,\mu)$ is a nonunital $A_\infty$-category, one can adjoin an unit $1_X$ to the endomorphism space of each object $X$ to get a strictly unital $A_\infty$-category $\cA^+$, which is moreover quasi-equivalent to $\cA$ when $\cA$ is homologically unital \cite{lefevre2003}.

\begin{definition}
	A pre-CY category $(\cA,m=\sum_{k \geq 1} m_{(k)})$ is called strictly unital if and only if, for every object $X$ of $\cA$, there is an element $1_X \in \hom_\cA(X,X)$ such that $m_{(1)}^2(1_X,a)=a$, $(-1)^{\bar b} m_{(1)}^2(b,1_X)=b$ for all $a\in  \cA(X,Y), b\in  \cA(Y,X)$ and any object $Y$, and every higher structure map $m_{(k)}, k \ge 2$, evaluates to zero on any sequence containing $1_X$.
\end{definition}
Note that by the definition above a strictly unital pre-CY category is also a strictly unital $A_\infty$-category. The following lemma follows directly from the definitions.
\begin{proposition}
	Let $(\cA, \mu = m_{(1)})$ be an $A_\infty$-category, not necessarily strictly unital, and $(\cA^+, \mu^+)$ its strictly unital augmentation. Then any pre-CY structure $m = \{m_{(k)}\}$ of dimension $d$ on $\cA$ extends to a pre-CY structure $m^+  = \{m^+_{(k)}\}$ of dimension $d$ on $\cA^+$, given by setting
	\[ m^+_{(1)} = \mu^+, \quad m^+_{(k)}|_\cA = m_{(k)}, \forall k \ge 2 \]
	and $m_{(k \ge 2)}$ evaluated on any sequence containing a unit $1_X$ gives zero.
\end{proposition}

\subsubsection{The category of pre-CY algebras}\label{sec:categoryPreCYalgebras}
Recall that $A_\infty$-algebras over $\kk$ form a category, with morphisms $(A,\mu_A) \to (B,\mu_B)$ given by $A_\infty$-functors $f = \{f^n\}$, that is, collections of maps
\[ f^n: A[1]^{\otimes n} \to B[1], \]
(of degree zero) satisfying a compatibility condition with the maps $\mu_A,\mu_B$, which can be expressed graphically as the following equation in $\prod_n \Hom(A[1]^{\otimes n},B[2])$.
\[ \begin{tikzpicture}[baseline={([yshift=-.5ex]current bounding box.center)}]
\draw (0,0) circle (1);
\node [vertex] (mu) at (0,0.55) {$\mu$};
\node [vertex] (f) at (0,-0.4) {$f$};
\draw [->-=0.8] (mu) to (f);
\draw [->-=0.8] (f) to (0,-1);
\end{tikzpicture} \quad = \quad \sum_{i\ge 1} \begin{tikzpicture}[baseline={([yshift=-.5ex]current bounding box.center)}]
\draw (0,0) circle (1.2);
\node [vertex] (f1) at (-0.85,0.2) {$f$};
\node [vertex] (f2) at (-0.25,0.2) {$f$};
\node at (0.3,0.2) {$\dots$};
\node [vertex] (f3) at (0.85,0.2) {$f$};
\node [vertex] (mu) at (0,-0.7) {$\mu$};
\draw [->-=0.8] (f1) to (mu);
\draw [->-=0.8] (f2) to (mu);
\draw [->-=0.8] (f3) to (mu);
\draw [->-=0.8] (mu) to (0,-1.2);
\end{tikzpicture}
\]
where the $i$th term on the right-hand side has $i$ vertices labeled $f$. As for the signs in this equation, there is a single possible ordering for the diagram on the left-hand side, and for the right-hand side, the ordering of the $f$s does not matter since they all have even degree (zero) as maps $A[1]^{\otimes n} \to B[1]$. Likewise, composition of $A_\infty$-morphisms $f$ and $g$ is given graphically by
\[ g \circ f = \sum_{i\ge 1} \begin{tikzpicture}[baseline={([yshift=-.5ex]current bounding box.center)}]
\draw (0,0) circle (1.2);
\node [vertex] (f1) at (-0.85,0.2) {$f$};
\node [vertex] (f2) at (-0.25,0.2) {$f$};
\node at (0.3,0.2) {$\dots$};
\node [vertex] (f3) at (0.85,0.2) {$f$};
\node [vertex] (mu) at (0,-0.7) {$g$};
\draw [->-=0.8] (f1) to (mu);
\draw [->-=0.8] (f2) to (mu);
\draw [->-=0.8] (f3) to (mu);
\draw [->-=0.8] (mu) to (0,-1.2);
\end{tikzpicture} \]
where the ordering of the $f$ also does not matter.

We now extend this description to pre-CY algebras. Let $(A,m), (B,n)$ be two pre-CY algebras of dimension $d$ over $\kk$. Let us consider collections of maps $f = \{f_{(k)}^{n_1,\dots,n_k}\}$, for all $k \ge 1$ and $n_i \ge 1$, where
\[ f_{(k)}^{n_1,\dots,n_k} \in \Hom(A[1]^{\otimes n_1} \otimes \dots A[1]^{\otimes n_k}, B[1]^{\otimes k})^{(\ZZ/k,d)}[(d-3)(k-1)] \]
where as before $(\ZZ_k,d)$ indicates the dimension $d$ action of the cyclic group, as in \cref{def:actionDimensionD}.

A morphism $f: (A,m) \to (B,n)$ is a collection of maps as above that satisfy a compatibility condition with respect to the structure maps $m$ and $n$. To describe this condition, and also the composition of morphisms, we use \emph{planar tree quivers with height structure} in the disc, with outgoing arrows. The \emph{height} of such a planar tree quiver is the number of edges of the longest directed path. We require that given a planar tree quiver of height $N$, its vertices are partitioned into vertices of \emph{height} 1 through $N$, such that an arrow $v_i \to v_j$ means that $\mathrm{height}(v_j) > \mathrm{height}(v_i)$. This is what we call `height structure'. We will not allow any vertex to be a sink.

We are now ready to state the compatibility condition, with one caveat. Unlike in the case of $A_\infty$-morphisms, where the maps $f^n$ are all degree zero (seen as maps $A[1]^{\otimes\dots} \to B[1]^{\otimes\dots}$) and their ordering of application does not matter for signs, for pre-CY morphisms that is not the case when $d$ is even, since the degree $(d-3)(k-1)$ can be odd. It turns out that there is an easy way of correctly specifying the signs for the equation below, but it requires the formalism of orientations on ribbon quivers, that will be discussed later in \cref{sec:orientations}. So we postpone the definition of signs in \cref{def:compatibility,def:composition}, and the proof of \cref{prop:preCYcategory}, to \cref{sec:signsNecklace}.
\begin{definition}\label{def:compatibility}
	The collection of maps $f = \{f_{(k)}^{n_1,\dots,n_k}\}$ is a pre-CY morphism $(A,m) \to (B,n)$ if the following equation is satisfied
	\[ \sum \pm \begin{tikzpicture}[baseline={([yshift=-.5ex]current bounding box.center)}]
		\draw (0,0) circle (2.1);
		\node [vertex] (m) at (0,0) {$m$};
		\node [vertex] (f1) at (-1.3,-0.5) {$f$};
		\node [vertex] (f3) at (0,+1.2) {$f$};
		\node [vertex] (f4) at (1.1,0) {$f$};
		\node [vertex] (f6) at (0,-1.3) {$f$};
		\draw [->-] (m) to (f1);
		\draw [->-] (f1) to (-1.35,-1.6);
		\draw [->-] (f1) to (-2,-0.6);
		\draw [->-] (f1) to (-2,0.6);
		\draw [->-] (m) to (f3);
		\draw [->-] (f3) to (-1,1.85);
		\draw [->-] (f3) to (1,1.85);
		\draw [->-] (m) to (f4);
		\draw [->-] (f4) to (1.85,1);
		\draw [->-] (f4) to (1.85,-1);
		\draw [->-] (m) to (f6);
		\draw [->-] (f6) to (0,-2.1);
	\end{tikzpicture} = \sum \pm \quad
	\begin{tikzpicture}[baseline={([yshift=-.5ex]current bounding box.center)}]
		\draw (0,0) circle (2.1);
		\node [vertex] (n) at (0,0) {$n$};
		\node [vertex] (f1) at (-1,-0.3) {$f$};
		\node [vertex] (f3) at (0,+1.2) {$f$};
		\node [vertex] (f4) at (1,0.5) {$f$};
		\node [vertex] (f6) at (0,-1.3) {$f$};
		\draw [->-] (n) to (-2, 0.6);
		\draw [->-] (f1) to (n);
		\draw [->-] (n) to (-1.4,-1.6);
		\draw [->-] (n) to (1.85,-1);
		\draw [->-] (f1) to (-2,-0.6);
		\draw [->-] (f3) to (n);
		\draw [->-] (f3) to (-1,1.85);
		\draw [->-] (f3) to (1,1.85);
		\draw [->-] (f4) to (n);
		\draw [->-] (f4) to (1.85,1);
		\draw [->-] (f6) to (n);
		\draw [->-] (f6) to (0,-2.1);
	\end{tikzpicture}
	\]
	where the sum on the left-hand side is over all planar tree quivers of height two with \emph{a single vertex at height one}; we place an $m$ on that vertex and $f$s on the vertices at height two. On the right-hand side, the sum is over all planar tree quivers of height two with \emph{a single vertex at height two}; we place an $n$ on that vertex and $f$s on the vertices at height one.
\end{definition}

Using the degrees of $m,n$ and $f$ we calculate that each term in the expression above has degree only depending on the overall number $k$ of outgoing arrows, given by $(d-3)(k-1)+1$. So the condition above is an equation in the space
\[ \prod_{k=1}^\infty \prod_{n_1,\dots n_k \ge 0} \Hom(A[1]^{\otimes n_1} \otimes \dots A[1]^{\otimes n_k}, B[1]^{\otimes k})[(d-3)(k-1) + 1] \]
and it can be seen as many equations, one for each $k \ge 1$. For instance, the two diagrams we depicted for the sake of example are elements of
\[ \prod_{n_1,\dots n_8 \ge 0} \Hom(A[1]^{\otimes n_1} \otimes \dots A[1]^{\otimes n_8}, B[1]^{\otimes 8})[7(d-3) + 1].\]

We now define the composition of two such pre-CY morphisms $f:(A,m) \to (B,n)$ and $g:(B,n) \to (C,p)$.
\begin{definition}\label{def:composition}
	The composition $h = g \circ f$ is given by a sum
	\[ h = \sum \pm \quad \begin{tikzpicture}[baseline={([yshift=-.5ex]current bounding box.center)}]
		\draw (0,0) circle (2);
		\node [vertex] (f1) at (-0.5,-0.2) {$f$};
		\node [vertex] (f2) at (0.5,-0.2) {$f$};
		\node [vertex] (f3) at (1,-1) {$f$};
		\node [vertex] (f4) at (0,1.5) {$f$};
		\node [vertex] (g1) at (0,-1) {$g$};
		\node [vertex] (g2) at (-1.5,0) {$g$};
		\node [vertex] (g3) at (0,0.7) {$g$};
		\node [vertex] (g4) at (1.3,0) {$g$};
		\draw [->-] (f1) to (g1);
		\draw [->-] (f2) to (g1);
		\draw [->-] (f1) to (g2);
		\draw [->-] (f1) to (g3);
		\draw [->-] (f3) to (g4);
		\draw [->-] (f2) to (g4);
		\draw [->-] (f4) to (g3);
		\draw [->-] (g1) to (0,-2);
		\draw [->-] (g2) to (-1.75,-1);
		\draw [->-] (g2) to (-1.75,1);
		\draw [->-] (g3) to (1,1.75);
		\draw [->-] (g3) to (-1,1.75);
		\draw [->-] (g4) to (2,0);
	\end{tikzpicture}\]
	where the sum is over all planar tree quivers of height two where there are no arrows leaving the disc from vertices of height one, which we label with $f$s, just from vertices of height two, which we label with $g$.
\end{definition}

For each diagram appearing in the definition of $h$, we see that the degree of the resulting map is independent of the number of $f$ and $g$ vertices, and only depends on the number of arrows $k$ leaving the disc, being given by $(d-3)(k-1)$, as desired. 
\begin{proposition}\label{prop:preCYcategory}
	For each integer $d$, the compatibility relation for pre-CY morphisms and the composition defined above give a category of pre-CY algebras of dimension $d$.
\end{proposition}

We note that restricting attention to the structure maps $m_{(1)}$ gives a definition of morphism that agrees with the definition of $A_\infty$-morphism, so we have the following result.
\begin{proposition}
	The functor $(A,m) \mapsto (A, \mu = m_{(1)})$ gives a forgetful functor from the category of pre-CY algebras of any dimension $d$ over $\kk$ to the category of $A_\infty$-algebras over $\kk$.
\end{proposition}
Finally we note that the results above also generalize immediately to the setting of pre-CY categories (i.e. with multiple objects), in the same way that functors between $A_\infty$-categories are defined.

\begin{remark}
	Leray and Vallette \cite{leray2022pre} have recently used the formalism of properads to describe a certain notion of morphism of (curved) pre-CY algebras. This morphism is given in terms of the decomposition map of some ``Koszul dual coproperad'', which they describe using certain decompositions of the disc into two types of regions, which they call `partitioned bangles'. In the non-curved case, their definition agrees with the category of pre-CY algebras we defined above, as shown in a more recent version of their preprint \cite[Prop.4.7]{leray2022pre}.
\end{remark}

\subsection{Pre-CY algebras in noncommutative geometry}
We argue now that pre-CY algebras should be seen as giving a notion of Lagrangian subspaces inside a noncommutative symplectic space.

\subsubsection{The finite-dimensional case}
We start with the case where $A$ is a finite dimensional graded vector space. In that case, the notion of a pre-CY structure can be rephrased in terms of cyclic $A_\infty$ structures.

\begin{proposition}\label{prop:finiteDimensional}
	Let $A$ be a finite-dimensional graded vector space. Then the data of a pre-Calabi Yau structure of dimension $d$ on $A$ is equivalent to the data of a cyclic $A_\infty$ structure of dimension $d-1$ on the space $A \oplus A^\vee[1-d]$, such that the subspace $A$ is an $A_\infty$-subalgebra.
\end{proposition}

\begin{proof}
	Recall that by definition a cyclic $A_\infty$-structure of dimension $d-1$ on a graded vector space $B$ is a pair $(\mu_B, \langle,\rangle)$ of an $A_\infty$-structure on $B$ and a nondegenerate pairing
	\[ \langle,\rangle: B \otimes B \to \kk[1-d] \]
	such that the tensor $\langle \mu_B^n(-,\dots,-),-\rangle$ is (graded) invariant under the cyclic action.

	Let $m$ be a pre-CY structure of degree $d$ on $A$. We now produce an $A_\infty$-structure $\mu$ on $B = A \oplus A^\vee[1-d]$; by definition this is the data of maps
	\[ \mu^N: ((A \oplus A^\vee[1-d])[1])^{\otimes N} \to (A \oplus A^\vee[1-d])[2] \]
	We produce all these maps from $m$ by dualizing the appropriate map $m^{n_1,\dots,n_k}_{(k)}$. Let us be explicit: consider the component
	\[ m^{n_1,\dots,n_k}_{(k)}: A[1]^{\otimes n_1} \otimes \dots \otimes A[1]^{\otimes n_k} \to A^{\otimes k} \]
	We then produce components of $\mu$ from it, in the following way. For simplicity we denote
	\[ m^{n_1,\dots,n_k}_{(k)}(a^1_1,\dots,a^k_{n_k}) = b_0 \otimes \dots \otimes b_{k-1} \]
	and regard all factors as living in $A[1]$.
	\begin{enumerate}
		\item We make a map
		\[ A[1]^{\otimes n_1} \otimes (A^\vee[2-d]) \otimes A[1]^{\otimes n_2} \otimes \dots \otimes A^\vee[2-d] \otimes A[1]^{\otimes n_k} \to A \]
		which on $(a^1_1,\dots,a^1_{n_1},c_1,\dots, c_{k-1},\dots, a^k_{n_k})$ first permutes all the $c_i$ factors (elements of $A^\vee[2-d]$) to the end, evaluates $ m^{n_1,\dots,n_k}_{(k)}$ on the $a$ factors, then permutes the outputs $b_i$ to pair $b_i$ with $c_i$, giving the result
		\[ \langle b_1, c_1 \rangle \dots \langle b_{k-1} c_{k-1} \rangle b_0 \]
		with the Koszul sign coming from all the permutations.
		\item If $n_k \ge 1$, we also make a map
		\[ A[1]^{\otimes n_1} \otimes (A^\vee[2-d]) \otimes A[1]^{\otimes n_2} \otimes \dots \otimes A^\vee[2-d] \otimes A[1]^{\otimes n_k} \to A^\vee[1-d] \]
		in the same way as above, but dualizing the last incoming factor of $A[1]$ instead of the first outgoing factor of $A$.
	\end{enumerate}

	One then has to check that the resulting structure is cyclic with respect to the canonical pairing of degree $d-1$ on $B$, and that it satisfies the $A_\infty$-relations, by performing a computation of the signs. Cyclicity follows from the cyclic invariance of the $m$ maps, and the $A_\infty$-relations follow from the necklace Maurer-Cartan equation. Recall that in the definition of the necklace product there are two sums; these correspond respectively to the terms in the $A_\infty$-relation for $\mu$ given by types (1) and (2) above.
	
	To see that this map is an equivalence, one constructs an inverse map to the one above; for each structure map of $A \oplus A^\vee[1-d]$, we reverse each $A^\vee$ arrow into a $A$ arrow with opposite orientation; cyclicity implies that the elements of $C^*_{(k)}(\cA)$ so obtained, for all $k \ge 2$, have the right behavior under the cyclic action; one can then check that these maps compose to the identity.
\end{proof}

\begin{remark}
	It is instructive to
	consider the differential $\mu^1:A \oplus A^\vee[1-d] \to (A \oplus A^\vee[1-d])[1]$: the component $A \to A^\vee[-d]$ (or equivalently, a pairing $A \otimes A[d] \to \kk$) is identically zero because we require that $A$ be an $A_\infty$-subalgebra. The component $A^\vee[-d] \to A$ may be nonzero; in terms of the pre-CY structure this is the copairing on $A$ given by the component $m^{0,0}_{(2)}: \kk \to A \otimes A[d]$.
\end{remark}

Another way of stating \cref{prop:finiteDimensional} is to say that a finite-dimensional pre-CY algebra is a Lagrangian $A_\infty$-subalgebra. As remarked in the Introduction, the definition of pre-CY structure already appeared in the work of Tradler and Zeinalian \cite{tradler2007infinity} (under the name of $\cV_\infty$-algebras), and in the work of Seidel \cite{seidel2012fukaya, seidel2021fukaya} under the name of boundary algebras; these definitions do not apply to infinite-dimensional algebras. In the case where $A$ is not finite-dimensional, but still compact (that is, $H^*A$ is finite dimensional), we will relate this notion to cyclic $A_\infty$-structures in \cref{sec:ncLagrangian}.

We have yet another simple relation between pre-CY structures and cyclic $A_\infty$-structures:
\begin{proposition}\label{prop:cyclictopre}
A cyclic $A_\infty$-structure of dimension $d$ on a finite-dimensional graded vector space $A$ also defines a pre-CY structure of dimension $d$ on $A$.
\end{proposition}
\begin{proof}
We simply set $m_{(1)} = \mu_A$ to be the $A_\infty$-structure on $A$, $m_{(2)} = m^{0,0}_{(2)}: \kk \to A \otimes A[d]$ to be the inverse of the pairing coming from the cyclic $A_\infty$-structure, and $m_{(k \ge 3)} = 0$. The cyclic relation for $\mu$ then implies that $m \circnec m = 0$, as desired.
\end{proof}
As usual, the proposition above also holds in the setting with multiple objects, where we replace the graded vector space $A$ with a `pre-$A_\infty$-category' $\cA$, that is a set of objects and graded vector spaces $\cA(X,Y)$ for every pair of objects.

\subsubsection{Cyclic forms}
The result of \cref{prop:finiteDimensional} is not compatible with quasi-equivalences of $A_\infty$-structures; for applications it will be useful to relax the cyclicity condition so that it is only required to hold up to homotopy. Since we will make use of $A_\infty$-minimal models, in this section we will restrict our attention to $A_\infty$-algebras, which is the setting where the theory of minimal models is more readily available in the literature.

These results are more naturally understood in the language of noncommutative formal manifolds, as developed in \cite{kontsevich2009notes}. Recall that the data of a nonunital $A_\infty$-algebra $A$ is the same as the data of noncommutative formal pointed dg-manifold $(X,x_0)$ with a homological vector field $Q$ and an isomorphism $A[1] \cong T_{x_0}X$.

The space of functions $\cO(X)$ on $X$ is then identified with the tensor algebra $\overline T(A[1]^\vee)$; its space of \emph{cyclic 0-forms} is then
\[ \Omega^0_\cyc(X) = \cO(X)/[\cO(X),\cO(X)]_\mathrm{top}, \]
where $[,]_\mathrm{top}$ is the topological completion of the algebraic commutator. Roughly, if $\{x_i\}$ are coordinates on $A[1]$, $\Omega^0_\cyc(X)$ is composed of cyclic formal series $f(\{x_i\})$ on free variables $\{x_i\}$.

The functions on the odd tangent bundle $\cO(\overline T[1]X)$ are then formal series on free variables $\{x_i, dx_i\}$, with $\deg(dx_i) = \deg(x_i)+1$. The spaces of cyclic nc differential forms $\Omega^m_\cyc(X)$ are given by the decomposition
\[ \Omega^0_\cyc(\overline T[1]X) = \prod_{m \ge 0} \Omega^m_\cyc(X) \]
into spaces spanned by expressions with exactly $m$ variables of the type $dx_i$.

We have two distinct differentials acting on cyclic nc differential forms: the cyclic de Rham differential $d_\cyc$ and the Lie derivative $Lie_Q$ with respect to the homological vector field $Q$. We denote by $\Omega^{m,\cl}_\cyc(X)$ the closed forms with respect to $d_\cyc$; the Lie derivative descends to this subcomplex so we can consider the complexes $(\Omega^{m,\cl}_\cyc(X), Lie_Q)$.

We can express the action of the Lie differential graphically by relating to the definitions of \cref{sec:Background} in terms of bimodules. Translating the definitions, we have an isomorphism
\[ \Omega^m_\cyc(X) = \left( (\overbrace{A_\Delta \otimes_A A_\Delta}^{m-1} \otimes_{A-A} A_\Delta)^\vee \right)^{\ZZ_m} \]
that is, a element $\omega \in \Omega^k_\cyc(X)$ can be seen as a vertex receiving $m$ cyclically ordered $A_\Delta$ arrows and any $k$ numbers of $A[1]$ arrows between them, for example:
\[\begin{tikzpicture}
\node [vertex] (phi) at (0,0) {$\omega$};
\node (a11) at (.7,1) {};
\node (a12) at (1.2,0) {};
\node (a21) at (0,-1.3) {};
\node (a31) at (-1.4,-0.2) {};
\node (a32) at (-1.3,0.6) {};
\node (a33) at (-0.7,1) {};
\node (b1) at (0,1.5) {};
\node (b2) at (1.5,-0.9) {};
\node (b3) at (-1.5,-0.9) {};

\draw [->-=0.7, very thick] (b1) to (phi);
\draw [->-=0.7, very thick] (b2) to (phi);
\draw [->-=0.7, very thick] (b3) to (phi);
\draw [->, shorten <=2mm, shorten >=2mm] (a11) to (phi);
\draw [->, shorten <=2mm, shorten >=2mm] (a12) to (phi);
\draw [->, shorten <=1mm, shorten >=2mm] (a21) to (phi);
\draw [->, shorten <=2mm, shorten >=2mm] (a31) to (phi);
\draw [->, shorten <=1mm, shorten >=2mm] (a32) to (phi);
\draw [->, shorten <=2mm, shorten >=2mm] (a33) to (phi);
\end{tikzpicture}\]
where the bold arrows label the $A_\Delta$ arrows. The Lie derivative $Lie_Q$ is then given by `circling' this vertex with a vertex corresponding to the $A_\infty$-structure maps $\mu_A$ and $\mu_{A_\Delta}$, as in \cref{sec:TensorProductsBimodules}.

\subsubsection{Symplectic structures and minimal models}\label{sec:SymplecticStructures}
From now on we assume the $A_\infty$ algebra $A$ is homologically unital and compact. Recall from \cref{def:strongCYstructure} that a compact CY structure on $A$ is a morphism of complexes
\[ \omega: CC_*(A) \to \kk[-d] \]
satisfying a nondegeneracy condition.

With the above assumptions on $A$, we have quasi-isomorphisms between the following three complexes, all of which calculate the cyclic cohomology  $HC^*(A)$:
\begin{enumerate}
	\item The dual of the `cyclic Cuntz-Quillen complex' \cite{kontsevich2009notes}
	\[
		CC^*_\mathrm{mod}(A) = ((C_*(A))^\vee[u^{-1}],b^* + u^{-1} B^*)
	\]
	The dual of the canonical map $HH_*(A) \to HC_*(A)$ is realized by the map $CC^*_\mathrm{mod}(A) \to (C_*(A))^\vee$ sending $u^{-1} \mapsto 0$, and a class
	\[
		\omega = \sum_{n\ge0} \omega_n u^{-n} \in CC^*(A)
	\]
	represents a compact CY structure if when $\omega_0$ induces a quasi-isomorphism $A \to A^\vee[-d]$.
	\item The complex of closed cyclic nc 2-forms $(\Omega^{2,\cl}_{\cyc}(X),Lie_Q)$ defined above. Recall that elements of this space are cyclically symmetric combinations of expressions of the form $f(x_1,\dots)dx_i g(x_1,\dots)dx_j$ where $f,g$ are formal power series in the free variables $x_i$. Recall that basis one-forms $dx_i$ give functions on the shifted tangent space $T_{x_0}X[-1] \cong A$; a class $\omega \in \Omega^{2,\cl}_{\cyc}(X)$ is a compact CY structure if its evaluation at zero $\omega|_0$ induces a quasi-isomorphism $A \to A^\vee[-d]$.
	\item Finally, we have the complex $(\Omega^0_\cyc(X)/\kk,Lie_Q)$ of cyclic nc 0-forms modulo constants. Taking the length one part gives a map $\Omega^0_\cyc(X)/\kk \to (A/[A,A])^\vee$ which we compose with the natural map $(A/[A,A])^\vee \to (\mathrm{Sym}^2 A)^\vee$ given by $\phi \mapsto \phi(\mu^2(-,-))$. A class $\omega \in \Omega^0_\cyc(X)/\kk$ gives a compact CY structure if its image under this map induces a quasi-isomorphism $A \to A^\vee[-d]$.
\end{enumerate}

We now recall the theory of $A_\infty$-minimal models. An $A_\infty$-algebra $(A_0,\mu_0)$ is minimal if the differential $\mu_0^1$ is zero; if there is an $A_\infty$ quasi-isomorphism $(A_0,\mu_0) \to (A,\mu)$, we say that $A_0$ is a \emph{minimal model} of $A$.

One can prove that any $A_\infty$-algebra has a minimal model $A_0$, which as a vector space is $H^*(A,\mu^1)$. Moreover, the structure maps on $A_0$ can be algorithmically constructed by the procedure known as \emph{homological perturbation}. Given a section $i:A_0 \hookrightarrow A$ of the projection $\pi:A \to A_0$, together with a homotopy $H: A \to A[-1]$, satisfying
\[ \id_A - i\circ \pi = \mu^1 \circ H - H \circ \mu^1, \]
there is a minimal $A_\infty$-structure $\mu_{A_0} = \{ \mu_{A_0}^{k \ge 2}\}$ and a quasi-isomorphism $i_A = \{i_A^{k \ge 1}: A_0 \to A$, extending $i_A^1 = i$. These maps can be constructed from $i,\pi,H$ by an appropriate sum over tree diagrams with those maps along the edges.

In geometric terms, the quasi-isomorphism $A_0 \simeq A$ corresponds to performing a change of coordinates around the base point $x_0 \in X$ in the corresponding formal noncommutative manifold. The induced action of $i_A^*$  on $(\Omega^{2,\cl}_{\cyc}(X),Lie_Q)$ is the transformation of forms induced by that change of coordinates. The following result says that one can always find such a change of coordinates which makes a given nondegenerate cyclic nc 2-form $\omega$ constant.
\begin{proposition}\cite[Cor.10.10]{kontsevich2009notes}
	Let $(X,x_0)$ be a formal noncommutative pointed dg manifold with $\dim(H^*(T_{x_0}X)) < \infty$. Then any class $\omega \in H^*(\Omega^{2,\cl}_{\cyc}(X),Lie_Q)$ which is nondegenerate (i.e., a compact CY structure on $A = T_{x_0}X[-1]$) gives a constant nondegenerate class $\omega_0$ on a minimal model (i.e., a cyclic $A_\infty$-structure on $A_0$).
\end{proposition}

\subsubsection{A noncommutative Lagrangian neighborhood theorem}\label{sec:ncLagrangian}
We now come to the main result of this Subsection, an extension of the proposition above which can be seen as a noncommutative version of the Lagrangian neighborhood theorem.
\begin{theorem}\label{thm:Minimal}
	Let $(A,\mu_A)$ be a (not necessarily unital) compact $A_\infty$-algebra, and $A_0 \to A$ a minimal model of $A$. Then $A_0$ has a pre-CY structure of dimension $d$, compatible with the transferred $A_\infty$-structure, if and only if there is a tuple $((B,\mu_B),f, \omega_A,\omega_B)$ where
	\begin{enumerate}
		\item $(B,\mu_B)$ is an $A_\infty$-algebra, $f = \{f^n\}$ is an $A_\infty$-morphism $f:A \to B$,
		\item $\omega_A,\omega_B$ are elements of $\Omega^{2,cl}_{cyc}(A),\Omega^{2,cl}_{cyc}(B)$ such that
		\[ Lie_{Q_B} \omega_B = 0, \quad f^* \omega_B = Lie_{Q_A} \omega_A, \]
		\item $\omega_B$ is a compact Calabi-Yau structure on $B$; inducing a symplectic form $\omega_{B,0}$ on $H^*(B,m_B)$,
		\item The image $f^1(H^*(A,\mu^1_A))$ is Lagrangian in $H^*(B,\mu^1_B)$ with respect to $\omega_{B,0}$, and
		\item The quadratic form $\omega_{A,0}$ on $H^*(A,\mu^1_A)$ is a perfect pairing when restricted to $\ker(f^1) \otimes \ker(f^1)$.
	\end{enumerate}
	Moreover, if $A$ is homologically unital then
	$(B,f)$ is also homologically unital and $f$ is a morphism of homologically unital $A_\infty$ algebras.
\end{theorem}
Note that the equations on the Lie derivatives of $\omega_A,\omega_B$ above are equivalent to the statement that the element
\[ \omega = (\omega_A, \omega_B) \in \Cone\left(f^*:(\Omega^{2,cl}_{cyc}(B),Lie_{Q_B}) \to (\Omega^{2,cl}_{cyc}(A),Lie_{Q_A}) \right) \]
of the cone on cyclic nc 2-forms is closed.

\begin{proof}
Let us prove the easy direction first. We pick $B_0 = A_0 \oplus A_0^\vee[1-d]$; by \cref{prop:finiteDimensional} the pre-CY structure of dimension $d$ on $A_0$ gives a cyclic $A_\infty$-structure of dimension $(d-1)$ on $B_0$, or equivalently a constant symplectic form $\omega_0$ giving a perfect pairing $B_0 \otimes B_0 \to \kk[1-d]$, and a compatible $A_\infty$-structure $\mu_{B_0}$ such that $A_0$ is an $A_\infty$-subalgebra. We then pick a quasi-inverse $s = \sum s^n$ to the $A_\infty$-morphism $i_A:A_0 \to A$, and declare $f$ to be the $A_\infty$ composition $j \circ s$, where $j = j^1$ is the inclusion of $A_0$. By assumption, $j^*(\omega_0) = 0$, so the tuple $((B_0,\mu_{B_0}), f, \omega_A = 0, \omega_0)$ gives the desired structure.

For the other direction, we must produce a pre-CY structure on $A_0$ from a tuple satisfying conditions $(1)-(5)$. We start by noting that these conditions are invariant under quasi-isomorphism, so we take $A$ and $B$ to be already minimal, and use an automorphism of the minimal $A_\infty$-algebra $B$ to make $\omega_B$ constant. We then have a morphism of graded vector spaces $f^1:A \to B$. Let us denote $K = \ker(f^1)$ and $L = f^1(A)$; we then have the short exact sequence of graded vector spaces
\[ K \overset{i}{\hookrightarrow} A \overset{f^1}{\twoheadrightarrow} L \subset B \]

We now split the proof into four steps:

\noindent\textbf{Step 1:} We find a non-minimal $A_\infty$-algebra $D$ quasi-isomorphic to $B$ such that $f^1$ lifts to an injective map $A \to D$. We construct this explicitly as follows. The quadratic form $\omega_{A,0}$ defines maps
\[ A \to A^\vee[-d], \qquad K \xrightarrow{\sim} K^\vee[-d], \]
the latter being an isomorphism by condition (5). We use these maps to define a projection $\pi: A \twoheadrightarrow K$ given by
\[ A \to A^\vee[-d] \xrightarrow{\mathrm{res}} K^\vee[-d] \to K \]
such that $\pi \circ i = \id_K$. This gives a section $s:L \to A$ and a splitting
\[ K \substack{i \\ \hookrightarrow \\ \twoheadleftarrow \\\pi} A \substack{f^1 \\ \twoheadrightarrow \\ \hookleftarrow \\ s} L \]
We also pick any complementary subspace to the Lagrangian $L$; this gives us a decomposition $B \cong L \oplus L^\vee[1-d]$.

We now define a nonminimal $A_\infty$-algebra
\[ D := B \oplus K \oplus K^\vee[1-d] = L \oplus L^\vee[1-d] \oplus K \oplus K^\vee[1-d] \]
with a differential given by $m^1_D: K^\vee[1-d] \to K[1]$ given by $\omega_{A,0}$, and zero on $B \oplus K$. The higher structure maps are given by $\mu^n_D = \mu^n_B$ on $B$ and zero on $K \oplus K^\vee[1-d]$; one checks promptly that these maps satisfy the $A_\infty$-relations. By definition, the inclusion $B \hookrightarrow D$ is a quasi-isomorphism of $A_\infty$-algebras.

We now define a constant 2-form $\omega_D: D \otimes D \to \kk[1-d]$ on $D$ by using $\omega_B$ on the subspace $B \otimes B$ plus the standard pairing between $K$ and $K^\vee$. We can pick the sign for the differential $\mu^1_D$ to satisfy the following relation for any $x,y \in K^\vee[1-d]$
\[ \omega_A(\mu^1_D(x),\mu^1_D(y)) = -\omega_D(\mu^1_D(x),y) = (-1)^{\bar x} \omega_D(x, \mu^1_D(y)) \]
Together with cyclicity of $\omega_B$ with respect to $\mu_B$, we then get
\[ \omega_D(\mu_D(a_1,\dots,a_n),a_{n+1}) - (-1)^{\bar a_1}\omega_D(a_1,\mu_D(a_2,\dots a_{n+1})) = 0 \]
for any $n$, establishing cyclicity of $\omega_D$. Using the graphical calculus for signs from \cref{sec:graphical} we can concisely express this relation as
\[ Lie_{D} \omega_D =
\begin{tikzpicture}[baseline={([yshift=-.5ex]current bounding box.center)}]
\node [vertex] (mu) at (0,0) {$\mu$};
\node [vertex] (omega) at (0.4,-0.8) {$\omega$};
\draw [->-] (mu) to (omega);
\draw [->-] (-0.6,0.8) to (mu);
\draw [->-] (0.6,0.8) to (mu);
\draw [->-] (1.4,0.8) to (omega);
\node at (0,1) {$\dots$};
\node at (-0.6, 1) {$a_1$};
\node at (0.6, 1) {$a_n$};
\node at (1.4, 1) {$a_{n+1}$};
\end{tikzpicture}-
\begin{tikzpicture}[baseline={([yshift=-.5ex]current bounding box.center)}]
\node [vertex] (mu) at (0.8,0) {$\mu$};
\node [vertex] (omega) at (0.4,-0.8) {$\omega$};
\draw [->-] (mu) to (omega);
\draw [->-] (-0.6,0.8) to (omega);
\draw [->-] (0.2,0.8) to (mu);
\draw [->-] (1.4,0.8) to (mu);
\node at (0.8,1) {$\dots$};
\node at (-0.6, 1) {$a_1$};
\node at (0.2, 1) {$a_2$};
\node at (1.5, 1) {$a_{n+1}$};
\end{tikzpicture} = 0 \]

We now define an $A_\infty$-morphism $g: A \to D$ such that $g^1$ is injective, explicitly as follows. The first map $g^1$ embeds $A$ as $L \oplus K$, in other words $g$ is equal to $f$ plus a correction in $K$:
\begin{align*}
	g^1: A &\to B \oplus K \oplus K^\vee[1-d] \\
	a &\mapsto (f^1(a),\pi(a), 0)
\end{align*}
and the higher maps are equal to $f$ but with a correction in $K^\vee[1-d]$:
\begin{align*}
g^n: A^{\otimes n} &\to B \oplus K \oplus K^\vee[1-d] \\
\vec a & \mapsto (f^n(\vec a),0, (\mu^1_D)^{-1} \pi \mu^n_A(\vec a))
\end{align*}
where $(\mu^1_D)^{-1}: K[1] \to K^\vee[1-d]$ is the inverse of the differential. Using the $A_\infty$-relations for $A$, one can check that this indeed defines a morphism of $A_\infty$-algebras $g: A \to D$.

\noindent\textbf{Step 2:} Now that we have an $A_\infty$-embedding, we further calculate an $A_\infty$-automorphism
\[ t: (D,\mu_D) \to (D, \nu) \]
such that the composition $h = t \circ g: (A,\mu_A) \to (D,\nu)$ satisfies the property that $\mathrm{Im}(h^n) \subseteq \mathrm{Im}(h^1)$ for all $n$. Note that the $A_\infty$-structure $\nu$ is different by quasi-isomorphic to $\mu_D$; also we will require $t^1 = \id_D$.

We must now calculate the higher maps $t^n$ and the structure maps $\nu^n$. The maps $t^n, n \ge 2$ are only nonzero on elements of $L \oplus K$, and map to elements of $L^\vee[1-d] \oplus K^\vee[1-d]$. We define them inductively; let us suppose that we have defined them up to $t^{n-1}$.

Now for $n$, first we define a map $\tau^n: A^{\otimes n} \to L \oplus L^\vee[1-d] \oplus K^\vee[1-d]$ given by
\[\begin{tikzpicture}[baseline={([yshift=-.5ex]current bounding box.center)}]
\node [vertex] (tau) at (0,0) {$\tau^n$};
\draw [->-] (tau) to (0,-1);
\draw [->-] (-0.6,0.8) to (tau);
\draw [->-] (0.6,0.8) to (tau);
\node at (0,1) {$\dots$};
\node at (-0.6, 1) {$a_1$};
\node at (0.6, 1) {$a_n$};
\end{tikzpicture} = -
\sum_{k = 2}^{n-1} \sum_{\substack{n_1,\dots, n_k \\ \sum_i n_i = n}}
\begin{tikzpicture}[baseline={([yshift=-.5ex]current bounding box.center)}]
\node [vertex] (a) at (-0.8,1) {$g^{n_1}$};
\node [vertex] (b) at (0.8,1) {$g^{n_k}$};
\node [vertex] (c) at (0,0) {$t^k$};
\node at (0,1) {$\dots$};
\node at (0,1.7) {$\dots$};
\node at (-1, 1.7) {$a_1$};
\node at (1, 1.7) {$a_n$};
\draw [->-] (-1,1.5) to (a);
\draw [->-] (-0.5,1.5) to (a);
\draw [->-] (0.5,1.5) to (b);
\draw [->-] (1,1.5) to (b);
\draw [->-] (a) to (c);
\draw [->-] (b) to (c);
\draw [->-] (c) to (0,-1);
\end{tikzpicture} -
\begin{tikzpicture}[baseline={([yshift=-.5ex]current bounding box.center)}]
\node [vertex] (tau) at (0,0) {$g^n$};
\draw [->-] (tau) to (0,-1);
\draw [->-] (-0.6,0.8) to (tau);
\draw [->-] (0.6,0.8) to (tau);
\node at (0,1) {$\dots$};
\node at (-0.6, 1) {$a_1$};
\node at (0.6, 1) {$a_n$};
\end{tikzpicture}\]
We then define $t^n$ to be zero on any sequence with factors in $L^\vee[1-d]$ and $K^\vee[1-d]$, and on powers of the subspace $L \oplus K$ to be given by
\[ \tau = \pi_{L^\vee [1-d] \oplus K^\vee[1-d]} \tau^n, \]
that is, using the identification $A = L \oplus K$ on the source, applying $\tau$ and projecting to $L^\vee[1-d] \oplus K^\vee[1-d]$. This definition implies that the all the maps of the composition $h = t \circ g$ have image in $L$.

We use the sequence of maps $t^n$ above to calculate the new $A_\infty$-structure $\nu$, also inductively, by requiring that $t$ be an $A_\infty$-morphism. Explicitly, the following formula (with Koszul signs coming from the graphical calculus) gives an $A_\infty$-structure $\nu$:
\[\begin{tikzpicture}[baseline={([yshift=-.5ex]current bounding box.center)}]
\node [vertex] (nu) at (0,0) {$\nu^n$};
\draw [->-] (nu) to (0,-1);
\draw [->-] (-0.6,0.8) to (nu);
\draw [->-] (0.6,0.8) to (nu);
\node at (0,1) {$\dots$};
\node at (-0.6, 1) {$d_1$};
\node at (0.6, 1) {$d_n$};
\end{tikzpicture} =
\sum_{k = 2}^{n-1} \sum_{\substack{n_1,\dots, n_k \\ \sum_i n_i = n}}
\begin{tikzpicture}[baseline={([yshift=-.5ex]current bounding box.center)}]
\node [vertex] (a) at (-0.8,1) {$t^{n_1}$};
\node [vertex] (b) at (0.8,1) {$t^{n_k}$};
\node [vertex] (c) at (0,0) {$\nu^k$};
\node at (0,1) {$\dots$};
\node at (0,1.7) {$\dots$};
\node at (-1, 1.7) {$d_1$};
\node at (1, 1.7) {$d_n$};
\draw [->-] (-1,1.5) to (a);
\draw [->-] (-0.5,1.5) to (a);
\draw [->-] (0.5,1.5) to (b);
\draw [->-] (1,1.5) to (b);
\draw [->-] (a) to (c);
\draw [->-] (b) to (c);
\draw [->-] (c) to (0,-1);
\end{tikzpicture} +
\sum_{k = 1}^{n-1} \sum_{0 \le i \le n-k}
\begin{tikzpicture}[baseline={([yshift=-.5ex]current bounding box.center)}]
\node [vertex] (tau) at (0,-0.1) {$t^n$};
\node [vertex] (mu) at (0,0.7) {$\mu^k$};
\draw [->-] (tau) to (0,-1);
\draw [->-] (-1.2,1.5) to (tau);
\draw [->-] (1.2,1.5) to (tau);
\draw [->-] (-0.6,1.5) to (mu);
\draw [->-] (0.6,1.5) to (mu);
\draw [->-] (mu) to (tau);
\node at (0,1.7) {$\dots$};
\node at (-1.2, 1.7) {$a_1$};
\node at (-0.6, 1.7) {$a_{i+1}$};
\node at (0.6,1.7) {$a_{i+k}$};
\node at (1.2, 1.7) {$a_n$};
\end{tikzpicture}\]
where the terms in the second sum have signs coming from the graphical calculus, namely $(-1)^{\bar a_1 + \dots \bar a_i}$ for the term shown.

\noindent\textbf{Step 3:} We now calculate that with the definitions above, the constant 2-form $\omega_D$ is indeed cyclic for $\nu$; this is a consequence of the relation $f^* \omega_B = Lie_{A}\omega_A$.

We first make an auxiliary calculation using the transformation $t$: we calculate the following identity for all elements $a_1, \dots, a_{n+1}$ in the image of $h^1: A \hookrightarrow D$.
\begin{align*}
t^* \omega_D(a_1,\dots, a_{n+1}) &=
\begin{tikzpicture}[baseline={([yshift=-.5ex]current bounding box.center)}]
\node [vertex] (t) at (0,0) {$t$};
\node [vertex] (omega) at (0.4,-0.8) {$\omega_D$};
\draw [->-] (t) to (omega);
\draw [->-] (-0.6,0.8) to (t);
\draw [->-] (0.6,0.8) to (t);
\draw [->-] (1.4,0.8) to (omega);
\node at (0,1) {$\dots$};
\node at (-0.6, 1) {$d_1$};
\node at (0.6, 1) {$d_n$};
\node at (1.4, 1) {$d_{n+1}$};
\end{tikzpicture} +
\begin{tikzpicture}[baseline={([yshift=-.5ex]current bounding box.center)}]
\node [vertex] (t) at (0.8,0) {$t$};
\node [vertex] (omega) at (0.4,-0.8) {$\omega_D$};
\draw [->-] (t) to (omega);
\draw [->-] (-0.6,0.8) to (omega);
\draw [->-] (0.2,0.8) to (t);
\draw [->-] (1.4,0.8) to (t);
\node at (0.8,1) {$\dots$};
\node at (-0.6, 1) {$d_1$};
\node at (0.2, 1) {$d_2$};
\node at (1.5, 1) {$d_{n+1}$};
\end{tikzpicture} \\
&= - \sum_{k = 2}^{n-1} \sum_{i = 0}^{n-k+1} (-1)^\# \begin{tikzpicture}[baseline={([yshift=-.5ex]current bounding box.center)}]
\node [vertex] (omega) at (0,-0.1) {$\omega_A$};
\node [vertex] (mu) at (0,0.7) {$\mu_A^k$};
\draw [->-] (-1.2,1.5) to (omega);
\draw [->-] (1.2,1.5) to (omega);
\draw [->-] (-0.6,1.5) to (mu);
\draw [->-] (0.6,1.5) to (mu);
\draw [->-] (mu) to (omega);
\node at (0,1.7) {$\dots$};
\node at (-1.2, 1.7) {$d_1$};
\node at (-0.6, 1.7) {$d_{i+1}$};
\node at (0.6,1.7) {$d_{i+k}$};
\node at (1.2, 1.7) {$d_n$};
\end{tikzpicture}
\end{align*}
where $\# = \deg(a_1)+ \bar a_2 + \dots + \bar a_i + 1$, which equals $\sum_{j=1}^i \bar a_j$ when $i \ge 1$ and $1$ when $i=0$. Note also that
\[ t^*\omega_D(\dots,x,\dots) = 0 \]
on a sequence of length $\ge 3$ containing an element $x \in L^\vee[1-d] \oplus K^\vee[1-d]$, since all the $t^{n\ge 2}$ vanishes on those elements.

By the bounds on the last sum, there are no terms when $n=1,2$; also for every $n$ the sum only depends only on the nonconstant part of $\omega_A$ (i.e. $\omega_A$ vertices with $\ge 3$ incoming arrows). The calculation above follows from the definition of $t^n$ and the relation between $\omega_A$ and $\omega_B$.

We then calculate
\[
\begin{tikzpicture}[baseline={([yshift=-.5ex]current bounding box.center)}]
\node [vertex] (mu) at (0,0) {$\nu$};
\node [vertex] (omega) at (0.4,-0.8) {$\omega_D$};
\draw [->-] (mu) to (omega);
\draw [->-] (-0.6,0.8) to (mu);
\draw [->-] (0.6,0.8) to (mu);
\draw [->-] (1.4,0.8) to (omega);
\node at (0,1) {$\dots$};
\node at (-0.6, 1) {$d_1$};
\node at (0.6, 1) {$d_n$};
\node at (1.4, 1) {$d_{n+1}$};
\end{tikzpicture}-
\begin{tikzpicture}[baseline={([yshift=-.5ex]current bounding box.center)}]
\node [vertex] (mu) at (0.8,0) {$\nu$};
\node [vertex] (omega) at (0.4,-0.8) {$\omega_D$};
\draw [->-] (mu) to (omega);
\draw [->-] (-0.6,0.8) to (omega);
\draw [->-] (0.2,0.8) to (mu);
\draw [->-] (1.4,0.8) to (mu);
\node at (0.8,1) {$\dots$};
\node at (-0.6, 1) {$d_1$};
\node at (0.2, 1) {$d_2$};
\node at (1.5, 1) {$d_{n+1}$};
\end{tikzpicture}\]
by using the inductive definition of $\nu^n$ in terms of $t^{k < n}, \nu^{k < n}$ and plugging in the equation for $t^*\omega_D$; this reduces the expression above to similar expressions for $k < n$, down to the base cases $k=1,2$ which can be calculated explicitly to be zero. Therefore $\omega_D$ is cyclic for $\nu$.

\noindent\textbf{Step 4:} Now we have a cyclic $A_\infty$-algebra $(D,\nu)$ with an $A_\infty$-map $h: A \to D$ all of whose components land on the Lagrangian subspace $L \oplus K$. It remains to prove that the structure maps $\nu$ preserve this subspace; we can express this as
\[ \omega_D(\nu^n(h^1(-),\dots, h^1(-)), h^1(-)) = 0 \]
This can also be proved by induction; the $A_\infty$-relation between $h$ and $\nu$ implies that the expression above can be reduced to similar expressions for $\nu^{k < n}$. The base case with $n=1$ then follows by assumption, since $\nu^1 = \mu^1$ lands in $K$, a subset of the Lagrangian subspace. By \cref{prop:finiteDimensional}, this is equivalent to a pre-CY structure of dimension $d$ on the minimal $A_\infty$-algebra $A$. \\

As for the last part of the statement, regarding unitality of the minimal model, it is a general fact that units in cyclic $A_\infty$-categories can always be strictified by a cyclic $A_\infty$-quasi-isomorphism; see \cite[Prop.4.8]{davison2024purity}.
\end{proof}

\begin{remark}
	Recall that when $A$ is homologically unital we have a quasi-isomorphism
	\[ (\Omega^{2,\cl}_{\cyc}(A),Lie_{Q_A}) \cong (\Omega^0_\cyc(A)/\kk,Lie_{Q_A}) \]
	so \cref{thm:Minimal} can be rephrased in terms of cyclic 0-forms with no constant term, with an entirely analogous statement.
\end{remark}

\subsubsection{Pre-CY structures as noncommutative integrable polyvector fields}
We continue in the setting of a compact $A_\infty$-algebra $(A,\mu)$. Recall that an $A_\infty$-structure is a homological vector field $Q$ on the pointed formal dg manifold $X$ corresponding to $A$. The extension of this $A_\infty$-algebra to a pre-CY algebra $(A, m)$ with $m_{(1)} = \mu$ should be seen as an extension of $Q$ to a polyvector field on $X$ satisfying an integrability condition expressed by the vanishing of a noncommutative Schouten-Nijenhuis bracket.

For that, we go to a minimal model $A_0$ of $A$; that is, to a particular coordinate system around $x_0$.
\begin{definition}
	The shifted degree $(2-d)$-cotangent bundle $\Pi T^*[2-d]X$ is the pointed formal dg manifold corresponding to the graded vector space $A_0 \oplus A_0^\vee[1-d]$.
\end{definition}
From \cref{thm:Minimal}, we get a cyclic $A_\infty$-structure on $A_0 \oplus A_0[1-d]$, with $A_0$ as an $A_\infty$-subalgebra. Thus we get a homological vector field $Q'$ on $\Pi T^*[2-d]X$ which preserves the zero section $X$, restricting to $Q$ on it.

The vector field $Q'$ is Hamiltonian with respect to the constant two-form $\omega_0$ given by the standard pairing, that is, $Lie_{Q'}\omega = 0$. As in \cite{kontsevich2009notes}, there is a cyclic function $H \in \Omega^0_\cyc(\Pi T^*[2-d]X)$ satisfying $i_Q \omega = dH$.

In analogy with the commutative world, note that the space of polyvector fields is identified with the space of functions on the shifted cotangent bundle. Thus, we set the space of noncommutative $(2-d)$-shifted polyvector fields to be given by this latter space of cyclic functions $ \Omega^0_\cyc(\Pi T^*[2-d]X)$. One sees that the Poisson bracket $\{,\}$ on this space coming from $\omega_0$ is the analogue of the Schouten-Nijenhuis bracket acting on polyvector fields; which gives the following characterization.
\begin{lemma}
	The data of a pre-CY structure of dimension $d$ on a compact $A_\infty$-algebra $A$ is a polyvector field $H$ on the degree $d$ cotangent bundle $T^*[2-d]X$, satisfying the Maurer-Cartan equation $\{H,H\} = 0$.
\end{lemma}

\section{Deformation theory of pre-Calabi-Yau structures}\label{sec:deformation}
Each infinitesimal deformation problem of algebraic structures such as $A_\infty$ and pre-CY structures is governed by some type of Maurer-Cartan equation in an appropriate dg Lie algebra.

For $A_\infty$-structures on $\cA$, that Lie algebra is the shifted Hochschild cochains $C^*(\cA)[1]$ with differential given by the Gerstenhaber bracket $[\mu,-]_G$ with the $A_\infty$-structure maps. More abstractly, there is a formal derived stack $\mathrm{Def}_{A_\infty}(\cA)$ over $\kk$ parametrizing $A_\infty$-structures on $\cA$ whose derived tangent complex at a given point $\mu$ is calculated by the Hochschild cohomology $HH^*(\cA, \mu)$.

There is a similar description for the deformation theory of pre-CY structures, which is the topic of this section. In the special cases where $A$ is smooth or compact, we present methods to effectively compute the relevant deformation spaces.

\subsection{Higher Hochschild invariants}
Let $(\cA, \mu)$ be an $A_\infty$-category. The $A_\infty$ structure $\mu$ is an element of $C^*_{(1)}(\cA)$, and therefore $[\mu,-]_\nec$ defines a map $C^*_{(\ell)}(\cA) \to C^*_{(\ell)}(\cA)[1]$ for every $\ell$, which squares to zero as a consequence of \cref{prop:necklaceDGLie}.
\begin{definition}\label{def:HigherHH}
	The $\ell$th \emph{higher Hochschild cohomology} of the $A_\infty$-category $(\cA,\mu)$ is the graded vector space
	\[ HH^*_{(\ell)}(\cA) := H^*(C^*_{(\ell)}(\cA), [\mu,-]_\nec). \]
\end{definition}
This definition agrees with the usual Hochschild cohomology of $A_\infty$-categories when $\ell=1$.

For fixed $\ell \ge 1, d \in \ZZ$, recall the space of higher cyclic cochains from \cref{def:cyclicCochains}. Taking the necklace bracket with the $A_\infty$-structure $\mu$ preserves $d$ and $\ell$ and cyclic invariance, so we define:
\begin{definition}
	The $(\ell,d)$-higher cyclic cohomology of the $A_\infty$-category $(\cA,\mu)$ is the graded vector space
	\[ HC^*_{(\ell,d)}(\cA) :=  H^*(C^*_{(\ell,d)}(\cA), [\mu,-]_\nec). \]
\end{definition}
Recall that we introduced a shift depending on $d$ between the gradings for higher cyclic cochains and higher Hochschild cochains: an element of degree $n$ in $C^*_{(\ell,d)}(\cA)$ is an element of degree $n + (d-2)(\ell-1)$ in $C^*_{(\ell)}(\cA)$, cyclically invariant or anti-invariant depending on the parity of $(d-1)(\ell-1)$.
\begin{proposition}\label{prop:symmetrization}
	If for some $n \ge 0, k \ge 1$ we have $HH^n_{(k)}(\cA) = 0$, then 
	\[ HC^{(n - (d-2)(\ell-1))}_{(\ell,d)}(\cA) = 0. \]
\end{proposition}
\begin{proof}
	Suppose that we have a cochain $\phi \in C^{n - (d-2)(\ell-1))}_{(\ell,d)}(\cA)$; this is a cyclically invariant/anti-invariant element of $C^n_{(\ell)}(\cA)$, which by assumption is exact, i.e. $\phi = [\mu, \psi]_\nec$ for some $\psi$ which might not be cyclically invariant. We then take the symmetrization/antisymmetrization $\frac{1}{\ell}\sum_{\sigma} \pm \sigma(\psi)$ over cyclic permutations $\sigma$, which now lives in higher cyclic cochains and is also a primitive of $\phi$, since the differential also has cyclic symmetry.
\end{proof}

Let $m$ be a pre-CY structure of dimension $d$ extending the $A_\infty$-structure $\mu$ on $\cA$. The map $[m,-]_\nec$ mixes the spaces above for different $\ell$; it now defines a differential on the space
\[ C^*_{[d]}(\cA) := \bigoplus_{\ell \ge 1} C^*_{(\ell,d)}(\cA) \]
\begin{definition}
	The tangent cohomology of the pre-CY category $(\cA,m)$ is the graded vector space
	\[ H^*_{[d]}(\cA) :=  H^*(C^*_{[d]}(\cA), [m,-]_\nec). \]
\end{definition}
By the general theory of deformations, if $\cM_\mathrm{preCY}$ denotes the (derived) moduli stack of pre-CY structures on $\cA$, the tangent complex of $(\cA,m)$ models the tangent space $T_m\cM_\mathrm{preCY}$.

\subsubsection{The higher cyclic to tangent cohomology spectral sequence}
Remembering only the $\mu = m_{(1)}$ component of a pre-CY structure gives a map
\[ \cM_\mathrm{preCY} \to \cM_{A_\infty} \]
which on tangent spaces at any given point $m$ is $C^*_{[d]}(\cA, m) \to C^*(\cA,\mu)$.

Consider the decreasing filtration
\[ F^k_{[d]}(\cA) := \prod_{n \geq k} C^*_{(n,d)}(\cA) \]
We note that the differential $d = [m,-]_\nec$ preserves this filtration: more precisely, the bracket with the component $m_{(\ell)}$ increases the number of outgoing arrows by $\ell-1$, giving a map $F^k_{[d]}(\cA) \to F^{k+\ell-1}_{[d]}(\cA)$ of cohomological degree one. So we have a spectral sequence associated to this filtered cochain complex, which abuts to the tangent cohomology.

The associated graded of the filtration on cochain is
\[ \Gr^k C^*_{(n,d)}(\cA) :=  F^k_{[d]}(\cA)/ F^{k+1}_{[d]}(\cA) =  C^*_{(k,d)}(\cA)\]
and the differential induced on $\Gr^k C^*_{(n,d)}(\cA)$ agrees with the bracket $[\mu,-]_\nec$ with the $A_\infty$-structure $\mu = m_{(1)}$. The standard theory of filtered complexes then gives:
\begin{proposition}
	Given any pre-CY structure on $\cA$, there is a spectral sequence $E_r^{p,q}$, starting from the higher cyclic cohomology
	\[ E_1^{p,q} = HC^{p+q}_{(p,d)}(\cA) \]
	and converging to the tangent cohomology
	\[ E_\infty^{p,q} = \Gr^p H^{p+q}_{[d]}(\cA). \]
\end{proposition}

\subsubsection{Extending an A-infinity structure to a pre-CY structure}
We now give an interpretation of the relation between the higher cyclic and tangent cohomologies, in terms of extending a solution to the $A_\infty$ Maurer-Cartan equation to a solution of the pre-CY Maurer-Cartan equation.

\begin{proposition}
	Given an $A_\infty$ category $(\cA,\mu)$, for each $k \ge 3$, the higher cyclic cohomology in degree two $HC^2_{(k,d)}(\cA)$ is the group of obstructions to extending a solution of the Maurer-Cartan equation from
	\[ C^*_{[d]}(\cA)/F^{k}_{[d]}(\cA) \text{\ to\ }  C^*_{[d]}(\cA)/F^{k+1}_{[d]}(\cA). \]
\end{proposition}
\begin{proof}
	Let us simply write $C^*$ etc. and leave $\cA$ implicit, for conciseness. We first describe the induction step when $k=3$. Suppose that we have a solution of the Maurer-Cartan equation modulo $F^3_{[d]}$; that is, elements $m_{(1)}, m_{(2)}$ such that
	\[ (m_{(1)}+m_{(2)}) \circnec (m_{(1)}+m_{(2)}) \equiv 0 \pmod{F^3_{[d]}} \]
	which is equivalent to requiring $[m_{(1)},m_{(2)}]_\nec = 0$. We then have
	\[ (m_{(1)}+m_{(2)}) \circnec (m_{(1)}+m_{(2)}) = m_{(2)} \circnec m_{(2)} \]
	We see that this is $[m_{(1)},-]_\nec$-closed, since
	\[ [m_{(1)},m_{(2)} \circnec m_{(2)}]_\nec = m_{(1)} \circnec m_{(2)} \circnec m_{(2)} - m_{(2)} \circnec m_{(2)} \circnec m_{(1)} = 0 \]
	in $C^2_{(3,d)}$. Therefore, if $HC^2_{(3,d)} = 0$ this must also be exact; if we take $m_{(3)}$ such that
	\[ [m_{(1)},m_{(3)}]_\nec = - m_{(2)} \circnec m_{(2)}\]
	we then have
	\[ (m_{(1)}+m_{(2)}+m_{(3)}) \circnec (m_{(1)}+m_{(2)}+m_{(3)}) \equiv 0 \pmod{F^4_{[d]}(\cA)}, \]
	that is, an extension of our solution to $C^*_{[d]}(\cA)/F^4_{[d]}$.

	In general, if we know that
	\[ (m_{(1)}+\dots+m_{(k-1)}) \circnec (m_{(1)}+\dots+m_{(k-1)}) \equiv 0 \pmod{F^k_{[d]}} \]
	then we have
	\[ [m_{(1)}, (m_{(1)}+\dots+m_{(k-1)}) \circnec (m_{(1)}+\dots+m_{(k-1)})]_\nec \equiv 0 \pmod{F^k_{[d]}} \]
	We write all the terms that appear in $F^k_{[d]}/F^{k+1}_{[d]}$, giving
	\begin{align*}
		[m_{(1)}, & (m_{(1)}+\dots+m_{(k-1)}) \circnec (m_{(1)}+\dots+m_{(k-1)})]_\nec  \\
		&\equiv \sum_{i=1}^{k-1} [m_{(1)}, [m_{(i)},m_{(k-i)}]_\nec]_\nec \pmod{F^{k+1}_{[d]}}
	\end{align*}
	which after applying graded Leibniz and cancellations gives
	\[ [m_{(1)}, (m_{(1)}+\dots+m_{(k-1)}) \circnec (m_{(1)}+\dots+m_{(k-1)})]_\nec \equiv 0 \pmod{F^{k+1}_{[d]}} \]
	Thus if $HC^2_{(k,d)} = 0$ we can find a primitive $m_{(k)}$ of $(m_{(1)}+\dots+m_{(k-1)})$ modulo $F^{k+1}_{[d]}$ implying
	\[ (m_{(1)}+\dots+m_{(k)}) \circnec (m_{(1)}+\dots+m_{(k)}) \equiv 0 \pmod{F^{k+1}_{[d]}}. \]
\end{proof}

We combine the proposition above with \cref{prop:symmetrization} to give a sufficient condition in terms of higher Hochschild cohomology.
\begin{corollary}\label{cor:extend}
	If $HH^{d\ell-d-2\ell+4}_{(\ell)}(\cA) = 0$ for every $\ell \ge 3$, then any cocycle $m_{(2)} \in C^2_{(2,d)}(\cA)$ can be extended to a pre-CY structure on $\cA$.
\end{corollary}

\subsection{Calculating higher Hochschild cohomology}
It becomes important therefore to compute $HH_{(\ell)}(\cA)$. Recall from \cref{sec:diagonalBimodule} that, in the case where the category $\cA$ is compact and/or smooth, one can express the (ordinary) Hochschild invariants in terms of certain dual bimodules, which are related to Serre functors. Here we extend that description to higher Hochschild cohomology groups.

\subsubsection{Compact A-infinity categories}
Let $\cA$ be a compact $A_\infty$-category. Recall the two canonical objects in $\cA\mh\Mod\mh\cA$ given by the diagonal bimodule $\cA_\Delta$ and its linear dual $\cA^\vee$. The linear dual has the property that for any perfect $\cA$-bimodule $\cM$, there is a quasi-isomorphism of complexes
\[ (\cM \otimes_{\cA-\cA} \cA_\Delta)^\vee \cong \Hom_{\cA-\cA}(\cM, \cA^\vee) \]
picking $\cM = \cA^\vee$, the preimage of the identity gives an element $\ev_\cA \in (\cA^\vee \otimes_{\cA-\cA} \cA_\Delta)^\vee$. In our graphical notation, we picture $\ev_\cA$ as a vertex
\[\begin{tikzpicture}[baseline={([yshift=-.5ex]current bounding box.center)}]
\node (left) at (0,0) {$\cA^\vee$};
\node (right) at (3,0) {$\cA_\Delta$};
\node [vertex] (mid) at (1.5,0) {$\ev_A$};
\draw [->-=0.9,very thick] (left) to (mid);
\draw [->-=0.9, very thick] (right) to (mid);
\end{tikzpicture}\]
taking two bimodule arrows and any number of $\cA[1]$ arrows along the top and bottom (not pictured).

Let $\phi \in C^*_{{k}}(\cA)$ be a $k$th higher Hochschild cochain. We define an element
\[ \tilde\phi \in \Hom_{\cA-\cA}(\cA_\Delta \otimes_\cA \overbrace{\cA^\vee \otimes_\cA \dots \otimes_\cA \cA^\vee}^{k-1}, \cA_\Delta) \]
by evaluating the following diagram:
\[\begin{tikzpicture}[baseline={([yshift=-.5ex]current bounding box.center)}]
\node (a) at (0,5) {$\cA_\Delta$};
\node (avee1) at (1,5) {$\cA^\vee$};
\node (avee2) at (2,5) {$\cA^\vee$};
\node (avee3) at (4,5) {$\cA^\vee$};
\node [vertex] (eva1) at (1,3.5) {$\ev_\cA$};
\node [vertex] (eva2) at (2,3.5) {$\ev_\cA$};
\node at (3,4) {$\dots$};
\node [vertex] (eva3) at (4,3.5) {$\ev_\cA$};
\node [vertex] (phi) at (2.2,2) {$\phi$};
\node [vertex] (mu) at (0,1.5) {$\mu_{\cA_\Delta}$};
\node (bot) at (0,0) {$\cA_\Delta$};

\draw [->-=0.6,very thick] (a) to (mu);
\draw [->-=0.9,very thick] (avee1) to (eva1);
\draw [->-=0.9,very thick] (avee2) to (eva2);
\draw [->-=0.9,very thick] (avee3) to (eva3);
\draw [->-=0.9,very thick] (phi) to (eva1);
\draw [->-=0.9,very thick] (phi) to (eva2);
\draw [->-=0.9,very thick] (phi) to (eva3);
\draw [-w-=0.6,very thick] (phi) to (mu);
\draw [->-=0.6,very thick] (mu) to (bot);
\end{tikzpicture}\]
that is, we use the evaluation element to convert the last $k-1$ outgoing elements of $\phi$ into incoming elements of $\cA^\vee$.

If $\phi$ is closed under the differential on $C^*_{{k}}(\cA)$, that is, $[\mu,\phi]_\nec = 0$, then by closedness of $\ev_\cA$ and the structure equation for $\mu_{\cA_\Delta}$ we see that the map above satisfies the structure equations to be a morphism of bimodules. We now precompose that morphism with some quasi-isomorphism
\[ (\cA^\vee)^{\otimes_\cA (k-1)} \cong \cA_\Delta \otimes_\cA (\cA^\vee)^{\otimes_\cA (k-1)}\]
to get a map which we also denote by $\tilde\phi \in \Hom_{A-A}((\cA^\vee)^{\otimes_\cA (k-1)}, \cA_\Delta)$.

\begin{proposition}\label{prop:calculatingCompact}
	When $\cA$ is homologically unital and compact, the map $\phi \mapsto \tilde\phi$ gives a quasi-isomorphism
	\[ C^*_{(k)}(\cA) \overset{\sim}{\longrightarrow} \Hom_{A-A}((\cA^\vee)^{\otimes_\cA (k-1)}, \cA_\Delta) \]
	for any $k \ge 1$.
\end{proposition}

\begin{proof}
Note that for $k=1$ we get an isomorphism $CC^*(\cA) \cong \Hom_{A-A}(\cA_\Delta,\cA_\Delta)$ which in this formalism of $A_\infty$-bimodules is proven in \cite[Sec.2.6]{ganatra2013symplectic}. Let us prove the case where $\cA$ is an $A_\infty$-algebra $A$ so we can omit the sums over tuples of objects; we will follow the same strategy used \emph{op.cit.}. Namely, we will use length filtrations to reduce the calculation to bar complexes for associative algebras.

Consider the cone of the morphism $\Psi: CC^*_{(k)}(A) \to \Hom_{A-A}(A_\Delta \otimes_A (A^\vee)^{\otimes_A (k-1)}, A_\Delta)$:
\begin{align*}
	\Cone(\Psi) &=  C^*_{(k)}(A) \oplus \Hom_{A-A}(A_\Delta \otimes_A (A^\vee)^{\otimes_A (k-1)}, A_\Delta)[1] \\
	&= \bigoplus_{\substack{\{n_i \ge 0\}, 1 \le i \le k \\ \{r_j \ge 0\}, 0 \le j \le k}} \Hom_\kk(A[1]^{\otimes n_1} \otimes \dots \otimes A[1]^{\otimes n_k}, A^{\otimes k}) \\
	&\oplus \Hom_\kk(A[1]^{\otimes r_0} \otimes A \otimes A[1]^{\otimes r_1} \otimes A^\vee \otimes \dots \otimes A[1]^{\otimes r_k}, A)[1]
\end{align*}
with differential given by the triangular matrix
\[ d_{\Cone} = \begin{pmatrix}
[\mu, -] & 0 \\
\Psi & [\mu,-]_\nec
\end{pmatrix} \]
where we schematically denote $\mu$ for the appropriate combination of structure maps for $A$, $A_\Delta$ and $A^\vee$.

Consider the following decreasing filtrations:
\begin{align*}
F^p(C^*_{(k)}(A)) &= \bigoplus_{\sum n_i \ge p} \Hom_\kk(A[1]^{\otimes n_1} \otimes \dots \otimes A[1]^{\otimes n_k}, A^{\otimes k}) \\
F^p(\Hom_{A-A}(A_\Delta &\otimes_A (A^\vee)^{\otimes_A (k-1)}, A_\Delta)) \\ &= \bigoplus_{\sum r_i \ge p-1} \Hom_\kk(A[1]^{\otimes r_0} \otimes A \otimes A[1]^{\otimes r_1} \otimes A^\vee \otimes \dots \otimes A^\vee \otimes A[1]^{\otimes r_k}, A)
\end{align*}
The differentials on each complex separately preserve this filtration; the only maps preserving $p$ degree are the components containing only $\mu^1_A$ and $\mu^1_{A^\vee}$, and all higher $\mu^{n \ge 2}$ increase $p$-degree. Moreover $\Psi$ sends
\[ F^p(C^*_{(k)}(A)) \to F^{p+1}(\Hom_{A-A}(A_\Delta \otimes_A (A^\vee)^{\otimes_A (k-1)}, A_\Delta)) \]
because the total number of inputs does not decrease; the $p$-degree one part of $\Psi$ is just given by the component $\mu^2_{A_\Delta}$.

Therefore we get a filtration on $\Cone(\Psi)$ which is compatible with the differential, so we have an associated spectral sequence computing its cohomology. The components of $d_{\Cone}$ preserving the length are only the ones containing the differentials $\mu^1$ so the first page of this spectral sequence is given by the same complex above, but on the unital associative algebra $H = H^*(A,\mu^1)$ instead.

By compactness, $H$ is finite dimensional and $\ev_H$ only has a single component given by the perfect pairing between $H$ and $H^\vee$.  The only nonzero terms in the first page differential are the terms containing $\mu^2_H, \mu^2_{H^\vee}$ and $\mu^2_{H_\Delta}$.

We can then put another filtration on this complex, now by counting the number of $H[1]$ inputs minus $r_0$, that is, by the number $\sum_{i=1}^{k} (n_i + r_i)$. The differential also preserves this filtration so we again take the associated spectral sequence; using the pairing to shift $k-1$ of the outgoing factors of $CC^*_{(k)}(H)$ to we then see that the total complex of the first page is the sum of total bar complexes for $H^\vee \otimes_H \dots \otimes_H H^\vee \otimes H^{r_k}$ as a left $H$-module (total meaning including the last term $\to H^\vee \otimes_H \dots \otimes_H H^\vee \otimes H^{r_k}$ coming from the $r_0 = 0$ component), which is acyclic for unital algebras, so the second spectral sequence, and therefore the first also, converge to zero.
\end{proof}

\subsubsection{Smooth A-infinity categories}
We now prove an analogous result to \cref{prop:calculatingCompact}, but for (homologically) smooth, instead of compact, $A_\infty$-categories. Recall that $\cA$ is smooth when its diagonal bimodule is perfect, and has as bimodule dual the ``inverse dualizing bimodule'' $\cA^!$ which represents Hochschild homology, i.e. for any perfect bimodule $\cM$ there is a quasi-isomorphism
\[ \Hom_{\cA-\cA}(\cA, \cM) \cong \cA^! \otimes_{\cA-\cA} \cM  \]
which is given by composing with a canonical coevaluation element $\ev^!_\cA \in \cA_\Delta \otimes_{\cA-\cA} \cA^!$. We picture $\ev^!_\cA$ as a vertex
\[\begin{tikzpicture}[baseline={([yshift=-.5ex]current bounding box.center)}]
\node (left) at (0,0) {$\cA^\vee$};
\node (right) at (3,0) {$\cA_\Delta$};
\node [vertex] (mid) at (1.5,0) {$\ev_A$};
\draw [->-=0.9,very thick] (left) to (mid);
\draw [->-=0.9, very thick] (right) to (mid);
\end{tikzpicture}\]
with two outgoing bimodule arrows and any number of outgoing $\cA[1]$ arrows along the top and bottom.

Let $\phi \in C^*_{(k)}$ be a $k$th higher Hochschild cochain. We define an element
\[ \tilde\phi \in \Hom_{\cA-\cA}(\cA_\Delta, \cA_\Delta \otimes_\cA \overbrace{\cA^! \otimes_\cA \dots \otimes_\cA \cA^!}^{k-1}) \]
by the following diagram:
\[\begin{tikzpicture}[baseline={([yshift=-.5ex]current bounding box.center)}]
\node (a) at (-1,5) {$\cA_\Delta$};
\node (avee0) at (-0.2,0) {$\cA^!$};
\node (avee1) at (1,0) {$\cA^!$};
\node (avee2) at (2.3,0) {$\cA^!$};
\node (avee3) at (4.5,0) {$\cA^!$};
\node [vertex] (eva0) at (-0.3,1.3) {$\ev^!_\cA$};
\node [vertex] (eva1) at (1,1.3) {$\ev^!_\cA$};
\node [vertex] (eva2) at (2.3,1.3) {$\ev^!_\cA$};
\node at (3.5,1.3) {$\dots$};
\node [vertex] (eva3) at (4.5,1.3) {$\ev^!_\cA$};
\node [vertex] (phi) at (2.2,3.5) {$\phi$};
\node [vertex] (mu) at (-1,3.5) {$\mu_{\cA_\Delta}$};
\node (bot) at (-1,0) {$\cA_\Delta$};
\draw [->-=0.6,very thick] (a) to (mu);
\draw [->-=0.9,very thick] (eva0) to (avee0);
\draw [->-=0.9,very thick] (eva1) to (avee1);
\draw [->-=0.9,very thick] (eva2) to (avee2);
\draw [->-=0.9,very thick] (eva3) to (avee3);
\draw [->-=0.9,very thick,bend left=30] (eva0) to (phi);
\draw [->-=0.9,very thick,bend left=15] (eva1) to (phi);
\draw [->-=0.9,very thick] (eva2) to (phi);
\draw [->-=0.9,very thick,bend right=20] (eva3) to (phi);
\draw [->-,thick,bend right=45] (phi) to (0.5,0.2);
\draw [->-,thick,bend right=15] (phi) to (1.6,0.2);
\draw [->-,thick,bend left=20] (phi) to (4,0.2);
\draw [->-,thick,bend left=60] (phi) to (5,0.2);
\draw [-w-=0.6,very thick] (phi) to (mu);
\draw [->-=0.6,very thick] (mu) to (bot);
\end{tikzpicture}\]

\begin{proposition}\label{prop:calculatingSmooth}
	When $\cA$ is homologically unital and smooth, the map $\phi \mapsto \tilde\phi$ gives a quasi-isomorphism
	\[ C^*_{(k)}(\cA) \overset{\sim}{\longrightarrow} \Hom_{A-A}(\cA_\Delta, (\cA^!)^{\otimes_\cA (k-1)}) \]
	for any $k \ge 1$.
\end{proposition}

The proof of this proposition will be similar to the proof of \cref{prop:calculatingCompact}, but we must first make an auxiliary definition. Let $\cA$ be an $A_\infty$-category and $\cM_1,\dots,\cM_k$ any tuple of $\cA$-bimodules.
\begin{definition}
	The $\cA$-bimodule $\cW(\cM_1,\dots,\cM_n)$ as a graded vector space is given by
	\[ \hspace{-1cm} \prod_{\{X^i_{j_i}\}} \Hom_\kk \left( \bigotimes_{i=1}^{k-1} (\cA(X^i_0, X^i_1)[1] \otimes \dots \otimes \cA(X^i_{n_i - 1}, X^i_{n_i})[1]), \ \bigotimes_{i=1}^{k-1} \cM(X^i_0,X^{i-1}_{n_{i-1}}) \right)  \]
	This graded vector space gets a differential from the structure maps of $\cA$ and of the bimodules $\cM_i$, and a $\cA$-bimodule structure from the maps $\mu^{r|1|s}_{\cM_1}$ and  $\mu^{r|1|s}_{\cM_k}$ of the first and last bimodule.
\end{definition}
We can picture an element of $\cW(\cM_1,\dots,\cM_n)$ as a vertex with outgoing $\cM_i$ arrows, and $n-1$ groups of incoming $\cA[1]$ arrows in between them. In particular, $\cW(\cM) = \cM$ for a single bimodule.

The following fact is a consequence of the univeral property of the inverse dualizing bimodule $\cA^!$ and of the fact that if $\cA$ is smooth, $\cA^!$ is also perfect and there is a quasi-isomorphism $\cA^{!!} \cong \cA$.
\begin{lemma}
	Let $\cA$ be smooth, $\cM_1,\cM_2$ any two $\cA$-bimodules. Then the map
	\[
	\begin{tikzpicture}[baseline={([yshift=-.5ex]current bounding box.center)}]
	\node [vertex] (top) at (1,2) {$\alpha$};
	\node (left) at (0,0) {$\cM_2$};
	\node (right) at (2.2,0) {$\cM_1$};
	\draw [->-=0.9,very thick] (top) to (left);
	\draw [->-=0.9, very thick] (top) to (right);
	\end{tikzpicture} \quad \mapsto \quad
	\begin{tikzpicture}[baseline={([yshift=-.5ex]current bounding box.center)}]
	\node [vertex] (top) at (1,2) {$\alpha$};
	\node (left) at (0,0) {$\cM_2$};
	\node (right) at (2.2,0) {$\cM_1$};
	\node [vertex] (mid) at (1,1) {$\ev^!$};
	\node (bot) at (1,0) {$\cA^!$};
	\draw [->-=0.9,very thick] (top) to (left);
	\draw [->-=0.9, very thick] (top) to (right);
	\draw [->-=0.9,very thick] (mid) to (top);
	\draw [->-=0.9,very thick] (mid) to (bot);
	\end{tikzpicture}
	\]
	gives a quasi-isomorphism $\cW(\cM_1,\cM_2) \xrightarrow{\sim} \cM_1 \otimes_\cA \cA^! \otimes_\cA \cM_2[-1]$.
\end{lemma}

\begin{proof}(of \cref{prop:calculatingSmooth})
	Applying the lemma above $(k-1)$ times we can prove instead the quasi-isomorphism
	\[ C^*_{(k)}(A) \xrightarrow{\sim} \cW(A_\Delta, \overbrace{A_\Delta[1],\dots,A_\Delta[1]}^{k-1}) \]
	We argue this in an entirely analogous way as in the proof of \cref{prop:calculatingCompact}, by using the filtration on the cone
	\[ \Cone(\Psi) = C^*_{(k)}(A) \oplus \cW(A_\Delta, \overbrace{A_\Delta[1],\dots,A_\Delta[1]}^{k-1})[1] \]
	given by the filtration induced by the number of $A[1]$ arrows on $C^*_{(k)}(A)$ and the number of $A[1]$ arrows plus one on $ \cW(A_\Delta, \overbrace{A_\Delta[1],\dots,A_\Delta[1]}^{k-1})$.

	We again get a spectral sequence whose first page is given by the same cone but for the unital associative algebra $H = H^*A$. We now use the second filtration, by the total length `on the right', getting a second spectral sequence whose first page has as total complex the total bar complex for $\cW(\overbrace{H_\Delta, \dots, H_\Delta}^{k-1})$ as a left $H$-module, which is acyclic for unital $H$.
\end{proof}

\subsection{Relation to smooth CY structures}\label{sec:relationToSmooth}
Let $(\cA,\mu)$ be a homologically smooth and unital $A_\infty$-category. Let $m$ be a pre-CY structure on $\cA$ compatible with $\mu$, that is, $m_{(1)} = \mu$. The next component is
\[ m_{(2)} \in C^2_{(2,d)}(\cA) \subset C^{2+ (d-2)(2-1)}_{(2)}(\cA) = C^d_{(2)}(\cA) \]
Using the quasi-isomorphism $C^*_{(2)}(\cA) \simeq \Hom_{\cA-\cA}(\cA_\Delta,\cA^!)$ we get a morphism $\Phi \in \Hom_{\cA-\cA}(\cA_\Delta,\cA^![d])$.

Recall that for $\cA$ smooth there is another quasi-isomorphism $\Hom_{A-A}(\cA^!,\cA_\Delta) \simeq C_*(\cA)$ between the Hochschild \emph{chain} complex and the inverse morphism space of bimodules. By definition, if $\Phi$ is a quasi-isomorphism of $\cA$-bimodules, any quasi-inverse $\Phi^{-1}$ defines a weak smooth CY structure of dimension $d$ on $\cA$. As mentioned in \cref{sec:CYstructures}, an algebra $A$ with such a structure is also known as a `Ginzburg CY algebra'. In \cite{KTV2} we prove the following result.
\begin{theorem}\label{thm:smoothCYtopreCY}
	Let $\tilde\omega \in CC^-_*(\cA)$ be a (strong) smooth CY structure on $\cA$, whose image $\omega \in C_*(\cA)$ induces a quasi-isomorphism $\cA^![d] \simeq \cA_\Delta$. Then there is a pre-CY structure $m$ on $\cA$ whose component $m_{(2)}$ induces an inverse quasi-isomorphism; conversely, given any such pre-CY structure one can produce a (strong) smooth CY structure on $\cA$.
\end{theorem}
Note that this result requires the existence of the lift $\tilde\omega$ in negative cyclic homology in order to produce the pre-CY structure. On the other hand, it guarantees the existence of this lift in the following case
\begin{corollary}
	If $\omega \in CC_d(\cA)$ is a weak smooth CY structure of dimension $d$ which has an inverse $m_{(2)} \in CC^*_{(2)}(\cA)$ such that $\mu + m_{(2)}$ extends to a pre-CY structure, then $\omega$ has a lift $\tilde\omega \in CC^-_d(\cA)$ giving a strong smooth CY structure of dimension $d$.
\end{corollary}
For instance, if $[m_{(2)},m_{(2)}]_\nec = 0$ then $\mu + m_{(2)}$ is already a pre-CY structure (with $m_{(n\ge 3)} = 0$) so this result applies.

\section{Examples}\label{sec:examples}
We now present some examples where one naturally finds pre-Calabi-Yau structures: topology of finite-dimensional manifolds with boundary and the algebraic geometry of varieties with anticanonical section.

\subsection{Finite-dimensional manifolds with boundary}\label{sec:finDimManifolds}
Recall from \cref{sec:SymplecticStructures} that, given an homologically unital and compact $A_\infty$-algebra $A$, there are three equivalent ways of describing compact CY structures on $A$; here we will use the third description, namely, as classes
\[ [\omega] \in H^*(\Omega^0_\cyc(X)/\kk, Lie_Q) \]
in the complex of nc 0-forms with no constant term on the corresponding formal pointed dg manifold $X$ with homological vector field $Q$. We have the following application, which already appears in \cite{kontsevich1993formal}
\begin{proposition}
	Let $M$ be a compact, closed, oriented manifold of dimension $d$. The fundamental class of $Y$ gives a compact CY structure of dimension $d$ on the dg algebra of de Rham forms $B = \Omega^*(M)$ (with coefficients in $\kk$).
\end{proposition}
\begin{proof}
	We regard $B = \Omega^*(M)$ as an $A_\infty$-algebra with $\mu^1 = d_\mathrm{dR}, \mu^2 = \wedge, \mu^{\ge 3} = 0$. Integration against the fundamental class $[M]$ gives a map
	\[ \omega = \int_{[Y]}: B \to \kk[-d] \]
	simply by assigning zero to forms of degree $<d$. We extend this to a nc 0-form $\omega \in \Omega^0_\cyc(B)/\kk$. Stokes' theorem implies that this form is closed under $Lie_Q$, and Poincar\'e duality of $M$ implies that the pairing
	\[ \omega \circ \wedge: B \otimes B \to \kk[d] \]
	is nondegenerate.
\end{proof}

In terms of minimal models, \cite[Thm.10.2.2]{kontsevich2009notes} then implies that
\begin{corollary}
	There is a cyclic $A_\infty$-structure on the graded vector space $H_\mathrm{dR}^*(M) = H^*(B,\mu^1)$.
\end{corollary}
In other words, there is a minimal $A_\infty$ structure on $H_\mathrm{dR}^*(M)$, quasi-isomorphic to the dg-algebra $\Omega(Y)$ and cyclic with respect to the Poincar\'e pairing.

Paul Seidel, in private communication, made the conjecture that if a compact oriented manifold $M$ has non-empty boundary then its cohomology $H^*(M)$ should have a pre-CY structure. Indeed, this follows from the results of \cref{sec:ncLagrangian}:
\begin{theorem}
Let $M$ be a compact oriented manifold of dimension $d$ with compact boundary $\partial M$ then the cohomology $H^*(M)$ of $M$ has the structure of a pre-CY algebra of dimension $d$.
\end{theorem}

\begin{proof}
Let us apply \cref{thm:Minimal} for the dg-algebras $A=\Omega(M)$ and $B=\Omega(\partial M)$, with $f = A \to B$ given by restriction of forms, i.e., pullback under the inclusion $i:\del M \to M$.

Integration against the fundamental classes $[M]$ in degree $d$ and $[\del M]$ in degree $d-1$ give maps
\[ \int_{[M]}: A \to \kk[-d], \qquad \int_{[\del M]}: B \to \kk[1-d] \]
which we extend to nc 0-forms $\omega_A,\omega_B$.

As before, $\omega_B$ is closed and gives a nondegenerate pairing. We now check the other conditions of \cref{thm:Minimal}. For any forms $\alpha_i$ on $M$ we calculate
\[ f^* \omega_B(\alpha_1,\alpha_2) = \omega_B(i^*\alpha_1,i^*\alpha_2) = \int_{[\del M]} i^*\alpha_1 \wedge i^*\alpha_2 \]
and by Stokes' theorem on $M$,
\[ (Lie_{Q_A} \omega_A)(\alpha_1,\alpha_2) = \int_{[M]} (d \alpha_1 \wedge \alpha_2 + (-1)^{\deg(\alpha_2)} \alpha_1 \wedge d\alpha_2) = \int_{[\del M]} i^*\alpha_1 \wedge i^*\alpha_2 \]
Moreover, the component of length 3 $(Lie_{Q_A} \omega_A)(\alpha_1,\alpha_2,\alpha_3)$ vanishes by associativity of $\wedge$, and the higher components vanish because $\mu^{\ge 3} = 0$. Thus we have $f^*\omega_B = Lie_{Q_A} \omega_A$.

The calculation above also shows that $f^1(H^*(A))=i^*(H^*(M))$ is Lagrangian, since $\int_{[\del M]} i^*\alpha_1 \wedge i^*\alpha_2 = 0$ if both $\alpha_1,\alpha_2$ are closed on $M$, and it has maximal dimension by Poincar\'e-Lefschetz duality.

It remains to check that $\omega_A = \int_{[M]}$ defines a non-degenerate pairing on the kernel $K = \ker(H^*(M) \to H^*(\del M))$. To see this, consider the long exact sequence
\[
\dots \to H^n(M) \overset{i^*}{\rightarrow} H^n(\partial M) \to H^{n+1}(M,\partial M)\to \dots
\]
Poincar\'e-Lefschetz duality and the compatibility $f^*\omega_B = Lie_{Q_A} \omega_A$ implies then that the map
\begin{align*}
\hspace{-1cm} K &= \ker(H^*(M) \to H^*(\del M)) \cong \coker(H^{*-1}(\del M) \to H^*(M,\del M)) \\
& \cong (\ker((H^*(M,\del M))^\vee \to (H^{*-1}(\del M))^\vee))^\vee \to \ker(H^{d-*}(M) \to H^{d-*}(\del M))^\vee = K^\vee[-d]
\end{align*}
is nondegenerate.
\end{proof}

\subsubsection{Poincar\'e pairs}
The result above can be applied to a slight generalization of oriented manifolds with boundary, given by the formalism of Poincar\'e pairs, explained in \cite{brav2019relative}. A Poincar\'e pair of dimension $d$ is a continuous map of topological spaces of finite type $f:X \to Y$, together with a class $[Y,X] \in H_d(Y,X)$, satisfying a certain nondegeneracy condition. An instance of such an object is an oriented manifold $M = Y$ with boundary $\del M = X$

Given any topological space of finite type $X$ and any field $\kk$, there is a linearization $\cL(X)$; this is a $\kk$-linear dg category such that there is a (noncanonical) equivalence $\cL(X) \simeq C_*(\Omega_\mathrm{pt} X)$ to the dg algebra of chains on the based loop space, and such that there is an equivalence between $\cL(X)$-modules and ($\infty$-)local systems on $X$ valued in $\kk$-chain complexes.
\begin{theorem} \cite[Thm.5.7]{brav2019relative}
	A Poincar\'e pair of dimension $d$ determines a relative smooth Calabi-Yau structure of dimension $d$ on the functor $\cL(X) \to \cL(Y)$, and therefore a relative compact Calabi-Yau structure on the functor $\mathrm{Loc^{fd}}(Y) \to \mathrm{Loc^{fd}}(X)$.
\end{theorem}

Translating the definition of relative CY structure into the language of $A_\infty$-structures and nc forms, we see that it corresponds exactly to the structure in the assumptions of \cref{thm:Minimal}. Therefore, we conclude the following.
\begin{corollary}
	If the dg categories $\mathrm{Loc^{fd}}(Y)$ and $\mathrm{Loc^{fd}}(X)$ have minimal models (as $A_\infty$-categories with $\mu^1=0$) then the minimal model for $\mathrm{Loc^{fd}}(Y)$ has a pre-CY structure of dimension $d$.
\end{corollary}
\begin{remark}
	Strictly speaking, we only proved \cref{thm:Minimal} in the setting of $A_\infty$-algebras, for the sole reason that the existence of minimal models has only been proven in the algebra case. But if we add the assumption of their existence the rest of the proof is the same.
\end{remark}

\begin{remark}
The results of \cite{KTV2} imply that it is also possible to show that the dg category $\cL(Y)$ itself carries a pre-CY structure, which is moreover nondegenerate in the sense of \cref{sec:relationToSmooth}.
\end{remark}

\subsection{Varieties with section of the anticanonical bundle}
Let $X$ be a quasi-compact separated scheme over $\kk$. We denote by $\cA = D_\mathrm{perf}(X)$ the derived category of perfect complexes on $X$. A result of Bondal and van den Bergh \cite{bondal2003generators} is that $\cA$ is generated under taking cones and direct sums by a single object $E$.

Setting $A_X = \End(E)$, seen as a dg algebra of endomorphisms, there is a triangulated equivalence $\cA \cong [\Perf(A_X)]$ to the derived category of perfect $A_X$-modules. This choice of generator exhibits the dg category $\Perf(A_X)$ as an enhancement of the triangulated category $\cA$. We also have a description of the category of $A_X$-bimodules; there is a quasi-isomorphism
\[ D_\mathrm{perf}(X \otimes X) \cong [\Perf(A_X \otimes A_X^{op})] \]
described more precisely in e.g. \cite{toen2007homotopy}. The algebra $A_X$ is homologically smooth when $X$ is smooth and compact when $X$ is compact.

We can use this enhancement to define Hochschild co/homology of $X$ as the corresponding invariants of $A_X$; the resulting complexes can be shown to be invariant up to quasi-isomorphism under derived equivalence, and agree with the geometric definitions
\[ HH^*(X) = \Ext^*_{X\times X}(\Delta_* \cO_X, \Delta_* \cO_X), \quad HH_*(X) = \mathbb{H}^*_{X\times X}(X, \Delta_* \cO_X \otimes_{X\times X} \Delta_* \cO_X)  \]
where by $\Delta_*,\otimes$ etc. we mean the derived version of those functors.

When $X$ is smooth, we have the Hochschild-Kostant-Rosenberg isomorphisms \cite{hochschild1962differential,alex2009hochschild}
\[ HH^*(X) \cong \bigoplus_{p+q=*} H^p(X,\wedge^q T_X), \quad HH_*(X) \cong \bigoplus_{p-q=*} H^p(X,\Omega^q_X) \]

We now extend this calculation to higher Hochschild homology.
\begin{proposition}
	If $X$ is smooth of dimension $d$, under the equivalence $D_\mathrm{perf}(X \otimes X) \cong [\Perf(A_X \otimes_\kk A_X^{op})]$, the inverse dualizing bimodule $A^!$ corresponds to $\Delta_*(\omega^{-1}_X)[-d]$.
\end{proposition}
\begin{proof}
	Roughly this follows from the adjunction (of derived functors) $\Delta_* \vdash \Delta^! = \Delta^*(-)\otimes_{\cO_X} \omega^{-1}_X[-d]$ and the identification of the diagonal bimodule as $\cO_\Delta = \Delta_* \cO_X$. To be more precise, we must calculate $\Hom_{A_X \otimes A_X}((A_X)_\Delta,A_X^e)$ as a $A_X^e$-module itself; for that we must consider four actions of the algebra $A_X$ which translates to sheaves on $X_1 \times X_2 \times X_3 \times X_4$. We number the copies of $X$ for clarity, and denote by $\pi_i, \pi_{ij}, \Delta_{ij}$ the appropriate projections and diagonal embeddings.

	The 4-module $A^e$ has two bimodule structures, outer and inner, and corresponds to the sheaf
	\[ \cO_{\Delta_{12}} \boxtimes \cO_{\Delta_{34}} \in \Perf(X_1 \times X_2 \times X_3 \times X_4) \]
	with outer and inner bimodule structure given by the identifications
	\[ (\pi_{14})_* (\cO_{\Delta_{12}} \boxtimes \cO_{\Delta_{34}}) \cong \cO_{X_1} \boxtimes \cO_{X_4}, \quad (\pi_{23})_* (\cO_{\Delta_{12}} \boxtimes \cO_{\Delta_{34}}) \cong \cO_{X_2} \boxtimes \cO_{X_2} \]
	The bimodule $A_X^! = \Hom_{A_X \otimes A_X}((A_X)_\Delta,A_X^e)$ is then given by
	\[ (\pi_{23})_* \sheafHom_{X_1 \times X_2 \times X_3 \times X_4}(\pi_{14}^*\cO_{\Delta_{14}}, \cO_{\Delta_{12}} \boxtimes \cO_{\Delta_{34}}) \]
	We now consider the isomorphism $\eta: X_2 \times X_3 \xrightarrow{\simeq} X_1 \times X_4$; from the fact that the support of the sheaf above is contained in the product of the diagonals we have calculate that it is isomorphic to
	\begin{align*}
	\eta_* (\pi_{14})_* &\sheafHom_{X_1 \times X_2 \times X_3 \times X_4}(\pi_{14}^*\cO_{\Delta_{14}}, \cO_{\Delta_{12}} \boxtimes \cO_{\Delta_{34}}) \cong \sheafHom_{X_2 \times X_3}(\cO_{\Delta_{23}}, \cO_{X_2 \times X_3}) \\
		&\cong (\Delta_{23})_*\Delta_{23}^!(\cO_{X_2 \times X_3}) \cong (\Delta_{23})_*(\omega^{-1}_X)[-d]
	\end{align*}
\end{proof}

Together with \cref{prop:calculatingSmooth}, this implies that:
\begin{corollary}
For any $k \ge 1$ there is a quasi-isomorphism of complexes
\[ C^*_{(k)}(A_X) \cong \Hom_{X\times X}(\Delta_*\cO_X, \Delta_*(\omega_X^{1-k})[d(1-k)]). \]
\end{corollary}

Consider now the pair of a smooth variety $X$ of dimension $d$ and a section $s$ of its anticanonical bundle $\omega_X^{-1}$. Its image under the pushforward $\Delta_*$ is an element of
\[ \Ext^d_{X\times X}(\Delta_* \cO_X, \Delta_* \omega_X^{-1}) \cong HH^d_{(2)}(A_X) \]
by the proposition above. We define $m_{(2)} \in C^2_{(2,d)}(A_X)$ to be (a cocycle representative) of the symmetrization of this element in $HH^d_{(2)}(A_X)$. We also denote $m_{(1)}$ to be the $A_\infty$ structure on $A_X$ (which is just a dg algebra structure in this case).
\begin{theorem}\label{thm:varietyWithSection}
	For any smooth $X$ and anticanonical section as above, the element $m_{(1)} + m_{(2)}$ can be extended to a pre-CY structure of dimension $d$ on $A_X$.
\end{theorem}
\begin{proof}
	We note first that since $\Delta$ is a closed immersion, $\Delta_*$ is an exact functor, and therefore for any coherent sheaf $\cF$ (in the abelian category, in degree zero in $\Perf(X)$), $\Delta_*\cF$ is also in the abelian category of coherent sheaves on $X\times X$.

	We now calculate, using the proposition above:
	\[ HH^{dk-d-2k+4}_{(k)}(A_Y) = \Ext^{4-2k}(\Delta_* \cO_X, \Delta_*(\omega_X^{1-k})) \]
	But the Ext groups vanish in negative degree since both objects are in the abelian category of coherent sheaves. So $HH^{dk-d-2k+4}_{(k)}(A_X)=0$ for all $k \ge 3$ and we can apply \cref{cor:extend}.
\end{proof}

\subsubsection{Calabi-Yau spaces}
A special case of the result above applies to varieties with non-vanishing section of their anticanonical sheaf $\omega_X$. An example of such a space is a Calabi-Yau variety of any dimension, open or closed; here we take the broad definition that a Calabi-Yau is just a smooth variety with trivial canonical bundle.

\begin{proposition}
For any Calabi-Yau variety $Y$ of dimension $d$, the dg category $\mathrm{Coh}(X)$ has a pre-CY structure of dimension $d$; moreover its component $m_{(2)}$ is nondegenerate in the sense of \cref{sec:relationToSmooth}.
\end{proposition}
\begin{proof}
Since $\omega_Y \simeq \cO_Y$, we pick a section trivializing its inverse $\omega_Y^{-1}$, which by \cref{thm:varietyWithSection} gives a pre-CY structure. Nondegeneracy follows from the fact that this section is nonvanishing.
\end{proof}

In fact the result above can be extended a little further, to any Gorenstein scheme with trivial canonical bundle. Moreover, by \cref{thm:smoothCYtopreCY}, existence of this pre-CY structure implies that there is a smooth Calabi-Yau structure of dimension $d$ on $\mathrm{Coh}(Y)$ for any such space $Y$; this has also been shown by \cite{brav2019relative}.

\section{PROPs of marked ribbon quivers}\label{sec:PROPs}
We now arrive at the proof of our main result, \cref{thm:MainPROP}. To recall, it says that the data of a pre-Calabi-Yau structure of dimension $d$ on an $A_\infty$-algebra/category $\cA$ determines an action of a certain colored dg \textsc{prop} $Q^d$ on the morphism spaces $\cA(X,Y)$ for all objects $X,Y$ of $\cA$ and on its Hochschild chain complex $C_*(\cA)$. This structure will be related to surfaces whose boundary has `open strings' (corresponding to elements of some space $\cA(X,Y)$) and `closed strings' (corresponding to Hochschild chains).

The \textsc{prop}s $Q^d$ will be defined combinatorially in terms of ribbon quivers in \cref{sec:theClosedPROP}, and using the graphical calculus for signs explained in \cref{sec:graphicalhigher} we prove that its action commutes with the relevant differentials. This proves half of \cref{thm:MainPROP}.

We postpone to \cref{sec:Strebel} the other half of this proof, namely, the fact that the complexes $Q^d$ compute chains on the  moduli spaces of open-closed surfaces; this description relies on an extension of Strebel's theorem to meromorphic quadratic differentials with higher-order poles, based on the description of Gupta and Wolf of the space of such objects \cite{gupta2016quadratic,gupta2019meromorphic}. This geometric description also makes manifest the fact that the composition maps of the combinatorial \textsc{prop} describe the maps induced on chains by gluing of surfaces.

\subsection{Marked ribbon quivers}
Let us introduce the combinatorial objects that we will use to define the \textsc{prop}s.

\subsubsection{Ribbon quivers}\label{sec:ribbonQuivers}
For us, a ribbon graph (or fatgraph) is a finite, connected graph whose vertices are equipped with a cyclic order of the incident half-edges. We will allow vertices of any valence $val(v) \in \ZZ_+ = \{1,2,\dots\}$. Every ribbon graph $\Gamma$ gives rise to an oriented topological surface with boundary $\Sigma_\Gamma$ by assigning a disc to each vertex and a rectangle to each edge, and then gluing according to incidence and ribbon structure.

\begin{definition}
	A \emph{acyclic ribbon quiver} $\vec\Gamma$ is a ribbon graph $\Gamma$ together with an orientation of each edge of $\Gamma$, such that
	\begin{enumerate}
		\item the underlying quiver of $\vec\Gamma$ has no oriented cycles, and
		\item any vertex of valence two is either a source or a sink; it cannot have one arrow in and another one out.
	\end{enumerate}
\end{definition}
Given any $\vec\Gamma$ as above, we denote by $\mathrm{Source}(\vec\Gamma)$ and $\mathrm{Sink}(\vec\Gamma)$ the corresponding subsets of the set $V(\Gamma)$ of vertices. Any vertex that is not a source or a sink we will call a \emph{flow} vertex, denoting their subset $\mathrm{Flow}(\vec\Gamma)$. We also denote by $\mathrm{Source}^1(\vec\Gamma)$ and $\mathrm{Sink}^1(\vec\Gamma)$ the subsets of those that have valence one.

Note that $\mathrm{Source}^1(\vec\Gamma) \sqcup \mathrm{Sink}^1(\vec\Gamma)$ is all the vertices of valence one, and to each element of this set there is a well-defined boundary circle, i.e., component of $\del \Sigma_\Gamma$ which it sits on.

\begin{definition}\label{def:markings}
	A \emph{marking} on an acyclic ribbon quiver $\vec\Gamma$ is the data of five \emph{ordered} subsets of $V(\Gamma)$, labeled
	\[ V_\times, V_\mathrm{open-in}, V_\mathrm{open-out}, V_\circ, V_\mathbf{1} \]
	all pairwise disjoint, with the following properties:
	\begin{enumerate}
		\item $V_\times \subset \mathrm{Source}^1(\vec\Gamma)$, such that if $v \in V_\times$, then no other vertices in $\mathrm{Source}^1(\vec\Gamma)$ sit on the same boundary component of $v$.
		\item $V_\mathrm{open,out} \subseteq \mathrm{Sink}^1(\vec\Gamma)$ such that if $v \in V_\mathrm{open,out}$, the boundary component it sits on doesn't have any vertex in $V_\times$
		\item $V_\circ \subset \mathrm{Sink}(\vec\Gamma) \setminus V_\mathrm{open-out}$
		\item $V_\mathbf{1} \subset \mathrm{Source}^1(\vec\Gamma)$ and every vertex in $V_\mathbf{1}$ is directly connected to a vertex in $V_\circ$, such that for each vertex in $V_\circ$ there is at most one vertex in $V_\mathbf{1}$ connected to it.
		\item $V_\mathrm{open-in} = \mathrm{Source}^1(\vec\Gamma) \setminus (V_\times \sqcup V_1)$.
	\end{enumerate}
	together with a choice of distinguished outgoing arrow for each vertices not in $\mathrm{Sink}(\vec\Gamma)$, and a distinguished incoming arrow for each vertex in $V_\circ$.
\end{definition}

We will see later, in \cref{sec:Strebel}, that marked ribbon quivers label cells in some moduli space of open-closed surfaces; that is, topological surfaces whose boundary has subsets marked as incoming/outgoing `open strings' (intervals) and `closed strings' (framed circles). Let us mention the topological interpretation of these types of vertices. Each vertex in $V_\times$ corresponds to an incoming closed string, and each vertex in $V_\mathrm{open,in/out}$ corresponds to an incoming/outgoing open string. The outgoing closed strings can be either labeled by a vertex in $V_\circ$ alone, or by a vertex in $V_\circ$ with a $V_\mathbf{1}$ attached to it; this latter arrangement should be interpreted as giving the cell (of one dimension more) corresponding to rotating the framing of the corresponding circle output.

\begin{figure}[h!]
	\centering
	\includegraphics[width=0.8\textwidth]{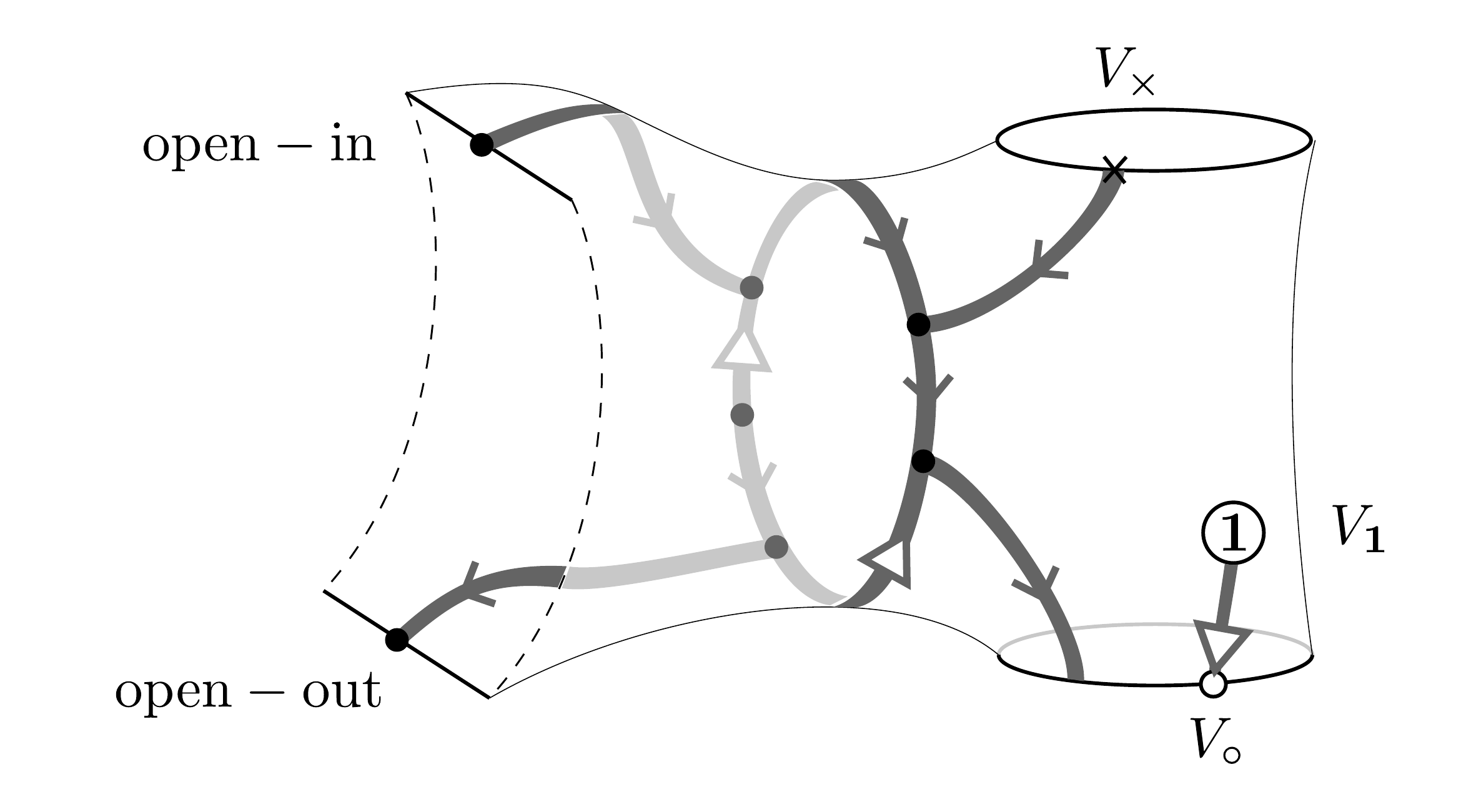}
	\caption{A ribbon quiver embedded in its corresponding open-closed surface $\Sigma$. The $V_\times$ vertices attach to closed inputs, the $V_\circ$ vertices map attach to closed outputs, and the open-in/open-out vertices to open inputs/outputs. We denote the distinguished arrows by white triangles when there are multiple possible choices.}
	\label{fig:ribbongraph}
\end{figure}


	The following marked ribbon quiver corresponds to the open-closed surface in \cref{fig:ribbongraph}:
	\[\begin{tikzpicture}[auto,baseline={([yshift=-.5ex]current bounding box.center)}]
	\node [inner sep=0pt] (x) at (2,0) {$\times$};
	\node [bullet] (top) at (0,1) {};
	\node [bullet] (o1) at (0.2,0.2) {};
	\node [bullet] (o2) at (-0.2,-0.2) {};
	\node [bullet] (bot) at (0,-1) {};
	\node [bullet] (ne) at (0.7,0.7) {};
	\node [bullet] (e) at (1,0) {};
	\node [bullet] (sw) at (-0.7,-0.7) {};
	\node [circ] (out) at (0,-2) {};
	\node [vertex] (one) at (0.6,-1.4) {$1$};
	\draw [->-,shorten <=-3.5pt] (x) to (e);
	\draw [->-] (o1) to (ne);
	\draw [->-] (sw) to (o2);
	\draw [->-] (0,1) arc (90:45:1);
	\draw [->-] (0.7,0.7) arc (45:0:1);
	\draw [-w-] (0,1) arc (90:225:1);
	\draw [-w-] (-0.7,-0.7) arc (225:270:1);
	\draw [->-] (1,0) arc (0:-90:1);
	\draw [-w-=1] (one) to (out);
	\draw [->-] (bot) to (out);
	\end{tikzpicture}\]
	Recall that we are free to attach a $V_\mathbf{1}$ vertex at the $\circ$-sink, as we did in this quiver.

\subsubsection{Genus and degree}
Let $(\vec\Gamma, V_\times, V_\mathrm{open-in}, V_\mathrm{open-out}, V_\circ, V_\mathbf{1})$ be a marked acyclic ribbon graph, which for simplicity we will just call $\vec\Gamma$.

\begin{definition}\label{def:degree}
	The genus $g(\vec\Gamma)$ is the genus of the closed surface $\overline{\Sigma}_\Gamma$. The homological $d$-degree of $\vec\Gamma$ depends on the choice of an integer $d$, and is given by the formula:
	\[ \hspace{-1cm} \deg_d(\vec\Gamma) = \sum_{v \in \mathrm{Source}^{\ge 2}} ((2-d)out(v)+d-4) + \sum_{v \in \mathrm{Flow}} ((2-d) out(v) + d + in(v) -4) + \sum_{v \in V_\circ} (in(v)-1) \]
	where $in(v)$ and $out(v)$ are the number of incoming and outgoing arrows of $v$.
\end{definition}

\begin{remark}
	Note that when $d=0$, the graphs with degree zero are exactly the ones with only trivalent flow vertices, bivalent unmarked sources, and valence one $\circ$-sinks. One gets the other graphs by starting from such a graph and contracting edges; each contracted edge contributes $+1$ to the homological degree. Adding a $V_\mathbf{1}$-leaf to a $V_\circ$-vertex also contributes $+1$ to the degree.
\end{remark}

\subsubsection{Orientations on ribbon quivers}\label{sec:orientations}
In order to define the differential on the \textsc{prop} of ribbon quivers, and moreover to assign its action with correct minus signs, it is necessary to introduce the notion of orientations.

Again let us fix an integer $d$, and a graph $\Gamma$. Suppose now that we assign a degree $|v| \in \ZZ$ to each vertex $v \in V(\Gamma)$, independently of $\Gamma$ itself.\footnote{This degree will be related to the degree of operation we insert at $v$, but for defining the orientations let us just describe it as an arbitrary integer.}

Consider the set $V(\Gamma) \sqcup E(\Gamma)$ of all its vertices and edges. Let us denote by $\mathrm{Ord}(\Gamma)$ the set of orderings of $V(\Gamma) \sqcup E(\Gamma)$; an element of $\mathrm{Ord}(\Gamma)$ is for instance a sequence
\[ (e_1 e_2 v_1 e_3 \dots v_n) \]
We define an action of the symmetric group $S_{|V(\Gamma)| + |E(\Gamma)|}$ on $\{\pm 1\} \times \mathrm{Ord}(\Gamma)$ by the following rule:
\begin{itemize}
	\item Vertices have weight $d + |v|$, edges have weight $d-1$
	\item When we commute any two elements $x,y$ in the sequences, we multiply by a factor $(-1)^{\mathrm{weight}(x)\mathrm{weight}(y)}$.
\end{itemize}

\begin{definition}\label{def:dOrientation}
	A $d$-orientation on a graph $\Gamma$ with vertex degrees $\{|v|\}$ is an element of the two-element set $(\ZZ/2 \times \mathrm{Ord}(\Gamma))/S_{|V(\Gamma)| + |E(\Gamma)|}$ where we take the quotient by the action of dimension $d$ above.
\end{definition}

Note that by definition, the notion of orientations only depends on the degrees $|v|$ up to parity. The case that will be most important to us is when all the degrees are even.
\begin{lemma}
	If all the degrees $|v|$ are even, then:
	\begin{itemize}
		\item if $d$ is even, a $d$-orientation on $\Gamma$ is the same as an orientation (in the classical sense) on the vector space $\mathrm{Span}_\RR(E(\Gamma))$, and
		\item if $d$ is odd, it is an orientation on $\mathrm{Span}_\RR(V(\Gamma))$.
	\end{itemize}
\end{lemma}

Let us now fix a ribbon quiver structure $\vec\Gamma$. This defines a partial order on the set of vertices, with $v > w$ if there is a path $v \to w$, and a cyclic order on the half-edges incident at each vertex. We have the following notion of compatibility:
\begin{definition}
	An ordering in $\mathrm{Ord}(\Gamma)$ is in \emph{normal form} with respect to the ribbon quiver structure $\vec\Gamma$ if it is of the following form:
	\[ (e_{11}, \dots, e_{1 k_1}, v_1, \dots, v_{n-1}, e_{n1}, \dots, e_{n k_n}, v_n), \]
	where $v_1, v_2, \dots, v_n$ is non-decreasing in the partial order, $(e_{i1},\dots,e_{i k_i})$ are the edges \emph{going out} of $v_i$ in some order compatible with the clockwise cyclic order.
\end{definition}
The data of an ordering in normal form can be given by a \emph{linear extension} of the partial order of vertices, together with marking one outgoing edge for every vertex $v_i$, indicating the edge $e_{i1}$ in the notation above.

\begin{example}
	Consider the following marked ribbon quiver of genus zero with two $\times$ sources and one $\circ$ sink:
	\[\begin{tikzpicture}[auto,baseline={([yshift=-.5ex]current bounding box.center)}]
	\node [inner sep=0pt,label=right:{$v_1$}] (v1) at (0,2) {$\times$};
	\node [bullet,label=below:{$v_3$}] (v3) at (0,1) {};
	\node [inner sep=0pt,label=left:{$v_2$}] (v2) at (0,0) {$\times$};
	\node [bullet,label=right:{$v_4$}] (v4) at (1,0) {};
	\node [bullet,label=above:{$v_5$}] (v5) at (0,-1) {};
	\node [circ,label=right:{$v_6$}] (v6) at (0,-2) {};
	\draw [->-,shorten <=-3.5pt] (v1) to node {$a$} (v3);
	\draw [->-,shorten <=-3.5pt] (v2) to node {$d$} (v4);
	\draw [-w-] (0,1) arc (90:0:1) node[midway]{$c$};
	\draw [->-] (0,1) arc (90:270:1) node[midway,swap]{$b$};
	\draw [->-] (1,0) arc (0:-90:1) node[midway]{$e$};
	\draw [->-] (v5) to node {$f$} (v6);
	\end{tikzpicture}\]
	We can pick for instance the linear extension $(v_1 > v_2 > v_3 > v_4 > v_5 > v_6)$ with the marked edges indicated by the white arrows (when there is a choice). The corresponding ordering in normal form is then
	\[ (v_6\ f\ v_5\ e\ v_4\ c\ b\ v_3\ d\ v_2\ a\ v_1). \]
\end{example}

\subsubsection{Action of oriented marked ribbon quivers}\label{sec:action}
Recall that in \cref{sec:graphicalhigher}, for an $A_\infty$-category $\cA$ we defined an action of directed trees of higher Hochschild cochains on $\cA$; more precisely, on the collection of morphism spaces $\cA(X,Y)$. We now show that this action naturally generalizes to an action of oriented marked ribbon quivers: the open in/outputs will still get assigned to $\cA(X,Y)$, but the closed in/outputs get assigned to Hochschild chains $C_*(\cA)$.

This action is described graphically in the following way. For each $\times$-source, we input a Hochschild chain $a_0 \otimes a_1 \otimes \dots \otimes a_p$ by sending $a_0$ along the edge and drawing $p$-arrows corresponding to $a_1 \dots a_p$ \emph{counter-clockwise} from it:
\[\begin{tikzpicture}[baseline={([yshift=-.5ex]current bounding box.center)}]
	\node [inner sep=0pt] (v) at (0,0) {$\times$};
	\draw [->-=0.3] (v) to node [midway,fill=white] {$a_0$} (1.6,0);
	\draw [->-=0.3] (v) to node [midway,fill=white] {$a_1$} (0.8,1.2);
	\draw [->-=0.3] (v) to node [midway,fill=white] {$a_2$} (-0.2,1.2);
	\node at (-1.2,0) {$\vdots$};
	\draw [->-=0.3] (v) to node [midway,fill=white] {$a_p$} (0.8,-1.2);
	\draw [->-=0.3] (v) to node [midway,fill=white] {$a_{p-1}$} (-0.6,-1.2);
\end{tikzpicture}\]

Now take $(\vec\Gamma, (\dots))$ a marked ribbon quiver with an ordering in normal form, such that all the sinks (i.e., vertices in $V_\circ$) are first, then all the $V_\mathbf{1}$ vertices, then all the flow vertices, then all the sources. The ordering then looks like
\[ (o_1 o_2 \dots o_n \dots e 1 \dots e \dots v_N \dots e\dots v_1 e \dots s_m \dots e \dots s_1 ) \]
where $o$ are the $V_\circ$ sinks, $v$ are the flow vertices and $s$ are the sources, and $e$ just generically denotes the edges in normal form.

Now, into each flow vertex with $out(v) = k$ we can insert a \emph{higher cyclic $k$-cochain}, i.e. an element $\phi \in C^*_{(k,d)}(A)$. We evaluate this on $m$ Hochschild chains and output $n$ Hochschild chains, in the following way:
\begin{enumerate}
	\item For each $\times$ source $s_i$ corresponding to each Hochschild chain $a^i_0 \otimes a^i_1 \otimes \dots \otimes a^i_{p_i}$ we draw the $a$ arrows coming out as above, then choose a way to connect them to the vertices embedded in the surface $\Sigma_\Gamma$, \emph{without crossing} each other.
	\item We write the ordering in normal form and the components of the Hochschild chains next to each other:
	\[ (o_1 o_2 \dots o_n e \dots v_N \dots e\dots v_1 e \dots s_m \dots e \dots s_1 )(a^1_0 a^1_1 \dots a^m_{p_m}) \]
	where now we regard all the $a$'s as elements of $A[1]$.
	\item We now proceed as we did in \cref{sec:graphicalhigher}; we bring the inputs of the last vertex to the beginning of the $a$ string.
	\item We then evaluate the cochain $\phi_i$ and write in the place the output, in the given ordering of the output edges. For the sources $s$, we interpret them as just producing the corresponding Hochschild chain.
	\item We repeat steps (3) and (4) until there are only outputs vertices left. We now reorder the elements of $A[1]$ to correspond to the output vertices in the order $o_1 \dots o_n$, beginning from the incoming marked direction of $o_i$; recall that if there is a $V_\mathbf{1}$ vertex attached we must start from it. We then sum over all the possibilities in step (1), with the Koszul sign coming from the transpositions.
\end{enumerate}
Note that all the Koszul signs are for $a$ as seen as elements of $A[1]$. For clarity, we now do an example.

\begin{example}
	We consider the graph below with the normal form ordering
	\[\begin{tikzpicture}[auto,baseline={([yshift=-.5ex]current bounding box.center)}]
	\node [inner sep=0pt,label=right:{$s_1$}] (s1) at (0,2) {$\times$};
	\node [bullet,label=below:{$v_1$}] (v1) at (0,1) {};
	\node [inner sep=0pt,label=left:{$s_2$}] (s2) at (0,0) {$\times$};
	\node [bullet,label=right:{$v_2$}] (v2) at (1,0) {};
	\node [bullet,label=above:{$v_3$}] (v3) at (0,-1) {};
	\node [circ,label=right:{$o$}] (o) at (0,-2) {};
	\draw [->-,shorten <=-3.5pt] (s1) to node {$e_1$} (v1);
	\draw [->-,shorten <=-3.5pt] (s2) to node {$e_4$} (v2);
	\draw [-w-] (0,1) arc (90:0:1) node[midway]{$e_3$};
	\draw [->-] (0,1) arc (90:270:1) node[midway,swap]{$e_2$};
	\draw [->-] (1,0) arc (0:-90:1) node[midway]{$e_5$};
	\draw [->-] (v3) to node {$e_6$} (o);
	\end{tikzpicture}, \qquad (o\ e_6\ v_3\ e_5\ v_2\ e_3\ e_2\ v_1\ e_4\ s_2\ e_1\ s_1) \]
	and input a pair of Hochschild chains $a_0 \otimes a_1 \otimes a_2$ and $b_0 \otimes b_1$ through the sources $s_1,s_2$, and put higher cyclic cochains $\phi \in C^*_{(2,d)}(A)$, $\psi, \lambda \in C^*_{(1,d)}(A) = C^*(A)$. Let us connect the arrows in the following way, for instance:
	\[\begin{tikzpicture}[baseline={([yshift=-.5ex]current bounding box.center)}]
	\draw [-w-] (0,1) arc (90:0:1);
	\draw [->-] (0,1) arc (90:270:1);
	\draw [->-] (1,0) arc (0:-90:1);
	\node [inner sep=0pt] (s1) at (0,2) {$\times$};
	\node [vertex,fill=white] (v1) at (0,1) {$\phi$};
	\node [inner sep=0pt] (s2) at (-0.5,-0.5) {$\times$};
	\node [vertex,fill=white] (v2) at (1,0) {$\psi$};
	\node [vertex,fill=white] (v3) at (0,-1) {$\lambda$};
	\node [circ] (o) at (0,-2) {};
	\draw [->-,shorten <=-3.5pt] (s1) to node [midway,fill=white]{$a_0$} (v1);
	\draw [->-=0.8,bend left=45] (s1) to node [midway,fill=white]{$a_1$} (v2);
	\draw [->-,bend right=90] (s1) to node [midway,fill=white]{$a_2$} (o);
	\draw [->-=1,shorten <=-3.5pt] (s2) to node [midway,fill=white] {$b_0$} (v2);
	\draw [->-=1] (s2) to node [midway,fill=white] {$b_1$} (v1);
	\draw [->-] (v3) to  (o);
	\end{tikzpicture}\]
	and then evaluate $X = (o\ e_6\ v_3\ e_5\ v_2\ e_3\ e_2\ v_1\ e_4\ s_2\ e_1\ s_1)(a_0\ a_1\ a_2\ b_0\ b_1)$.

	The source vertices just output the chains themselves, permuting the $a$s and $b$s back and forth, so we have no sign:
	\[ X =  (o\ e_6\ v_3\ e_5\ v_2\ e_3\ e_2\ v_1)(a_0\ a_1\ a_2\ b_0\ b_1) \]
	We then permute to get
	\[ X = (-1)^{\bar b_1(\bar a_0 + \bar a_1 + \bar a_2 + \bar b_0)}(o\ e_6\ v_3\ e_5\ v_2\ e_3\ e_2\ v_1)(b_1\ a_0\ a_1\ a_2\ b_0) \]
	and evaluate $\phi$, suppose for instance that $\phi(b_1;a_0) = c_1 \otimes c_2$, so that
	\[ X = (-1)^{\bar b_1(\bar a_0 + \bar a_1 + \bar a_2 + \bar b_0)}(o\ e_6\ v_3\ e_5\ v_2\ e_3\ e_2)(c_1\ c_2\ a_1\ a_2\ b_0) \]
	and then continuing this process until we have $(-1)^\# (o)(x\ a_2)$. We then permute $x$ and $a_2$ to read the output, since the edge $e_6
	$ is the marked edge going into the output $o$. The value of the diagram is then the sum over all possible ways of connecting the `extra arrows', in this case, the arrows carrying the elements $a_1,a_2,b_1$.
\end{example}

It is straightforward to check that once we choose higher cyclic cochains $\phi_i$ for each vertex, the action of a graph $\vec\Gamma$ of degree $\deg(\vec\Gamma)$ defines a map of graded vector spaces
\[ \Phi(\Gamma): (C_*(A))^{\otimes m} \to (C_*(A))^{\otimes n} \]
of degree $\sum_i (|\phi_i| - 2) - \deg(\vec\Gamma)$. We have the following result, which motivates our definition of orientation:
\begin{proposition}
	The dimension $d$ action of the symmetric group on edges and vertices intertwines the assignment $(\vec\Gamma,(\dots)) \to \Phi(\vec\Gamma)$, with the sign representation on the target. Therefore, it descends to an action of the set of graphs with $d$-orientation.
\end{proposition}
\begin{proof}
	This statement just says that a permutation between two normal forms acts by the same sign on orderings and on the resulting operator $\Phi(\vec\Gamma)$. To see this, note that any two normal forms are related by two types of moves:
	\begin{itemize}
		\item Switching the order of two vertices (corresponding to cochains $\phi,\psi$) that are incomparable in the partial order, together with their outgoing vertices. This introduces a sign
		\[ ((k-1)(d-1) + |\phi| + d)((\ell-1)(d-1)+ |\psi|+d) \]
		where $\phi \in C^*_{(k,d)}(A), \psi \in C^*_{(\ell,d)}(A)$.
		\item Performing a cyclic permutation of the output edges of any given vertex $\phi$, which introduces a sign $(k-1)(d-1)$.
	\end{itemize}
	For moves of type (1), we explicitly verify that changing the order of evaluation in $\Phi(\vec\Gamma)$ introduces the same sign
	\[ \bar \phi \bar \psi \equiv ((k-1)(d-1) + |\phi| + d)((\ell-1)(d-1)+ |\psi|+d) \pmod{2} \]
	(recall that $\bar\phi$ is the degree as a map to factors of $A[1]$) and for moves of type (2), we get a Koszul sign, which differs from $(k-1)(d-1)$ exactly by the sign used in the definition of higher cyclic cochains $C^*_{(k,d)}(A)$.
\end{proof}

\subsubsection{Signs for the necklace product and pre-CY morphisms}\label{sec:signsNecklace}
By the proposition above, we can define the action of a graph with any ordering of its edges and vertices, even if they are not in normal form. This now allows us to explain the signs appearing in our formulas in a more natural way. We return to the necklace product of \cref{def:NecklaceProduct}, and express the signs using orientations.
\begin{lemma}
	The necklace product of $\phi \in C^*_{(k,d)}(A), \psi \in C^*_{(\ell,d)}(A)$ is given by the symmetrization of the result of the oriented ribbon quiver
	\[\phi \circnec \psi = \left( \begin{tikzpicture}[baseline={([yshift=-.5ex]current bounding box.center)}]
	\draw (0,0) circle (1.5);
	\node [vertex] (psi) at (0,0.6) {$\psi$};
	\node [vertex] (phi) at (0,-0.5) {$\phi$};
	\draw [-w-=1] (phi) to (0,-1.8);
	\draw [->-] (phi) to (0,-1.5);
	\node at (0,-2) {$e_1$};
	\draw [->-] (phi) to (-1.4,-1);
	\node at (-1.5,-1.1) {$e_2$};
	\draw [->-] (phi) to (-1.6,0.5);
	\node at (-1.8,0.55) {$e_n$};
	\node at (-0.8,-0.3) {$\vdots$};
	\draw [->-] (psi) to (-1.1,1.1);
	\node at (-1.6,1.2) {$e_{n+1}$};
	\node at (0,1.1) {$\dots$};
	\draw [->-] (psi) to (1.1,1.1);
	\node at (2,1.2) {$e_{n+\ell-1}$};
	\draw [->-] (phi) to (1.4,-1);
	\node at (2.2,-1.1) {$e_{k+\ell-1}$};
	\node at (0.8,-0.3) {$\vdots$};
	\draw [->-] (phi) to (1.6,0.5);
	\node at (2.1,.55) {$e_{n+\ell}$};
	\draw [->-] (psi) to node [auto] {$e$} (phi);
	\end{tikzpicture}, (e_1 e_2 \dots e_{k+\ell-1}\phi e \psi) \right) + (\mathrm{cyc})\]
	where by (cyc) we denote the sum over cyclic permutations of the labels $e_1, \dots, e_{k+\ell-1}$ on the diagram.
\end{lemma}
One proves the lemma above simply by computing the signs and comparing with the previously given \cref{def:NecklaceBracket}. However, this definition in terms of the orientation allows us to easily prove properties about the necklace product and bracket.

\begin{lemma}
	The necklace product lands in higher cyclic cochains for the dimension $d$ action of $\ZZ_{k+\ell-1}$ and the corresponding necklace bracket, seen as a map
	\[ [-,-]_\nec: C^*_{[d]}(A)[1] \otimes C^*_{[d]}(A)[1] \to C^*_{[d]}(A)[1], \]
	satisfies the graded Jacobi relation.
\end{lemma}
\begin{proof}
	For the first statement, note that cyclic shift under $\ZZ_{k +\ell-1}$ does not change the set of graphs we sum over, but does change their orientations by cyclically rotating the edges $e_1\dots e_{k+\ell-1}$, which introduces a sign $(k-1)(d-1)$, which is exactly the sign added to the Koszul sign in the definition of higher cyclic cochain. The second statement just follows directly from the observation that all graphs appear with orientation given by the same expression.
\end{proof}

In \cref{def:compatibility}, when defining the notion of pre-CY morphisms, we omitted the signs in our formulas. We now return and complete those definitions using orientations. Let $f:(A,m) \to (B,n)$ be a pre-CY morphism; using orientations, the compatibility condition between $f,m,n$ is the equation
\[ \hspace{-1cm} \sum_\Gamma \left( \Gamma = \quad \begin{tikzpicture}[baseline={([yshift=-.5ex]current bounding box.center)}]
\draw (0,0) circle (2.1);
\node [vertex] (m) at (0,0) {$m$};
\node [vertex] (f1) at (-1.3,-0.5) {$f$};
\node [vertex] (f3) at (0,+1.2) {$f$};
\node [vertex] (f4) at (1.1,0) {$f$};
\node [vertex] (f6) at (0,-1.3) {$f$};
\draw [->-] (m) to (f1);
\draw [->-] (f1) to (-1.35,-1.6);
\draw [->-] (f1) to (-2,-0.6);
\draw [->-] (f1) to (-2,0.6);
\draw [->-] (m) to (f3);
\draw [->-] (f3) to (-1,1.85);
\draw [->-] (f3) to (1,1.85);
\draw [->-] (m) to (f4);
\draw [->-] (f4) to (1.85,1);
\draw [->-] (f4) to (1.85,-1);
\draw [->-] (m) to (f6);
\draw [->-] (f6) to (0,-2.1);
\end{tikzpicture}\quad , \cO_\Gamma \right) = \sum_{\Gamma'} \left( \Gamma' = \quad
\begin{tikzpicture}[baseline={([yshift=-.5ex]current bounding box.center)}]
\draw (0,0) circle (2.1);
\node [vertex] (n) at (0,0) {$n$};
\node [vertex] (f1) at (-1,-0.3) {$f$};
\node [vertex] (f3) at (0,+1.2) {$f$};
\node [vertex] (f4) at (1,0.5) {$f$};
\node [vertex] (f6) at (0,-1.3) {$f$};
\draw [->-] (n) to (-2, 0.6);
\draw [->-] (f1) to (n);
\draw [->-] (n) to (-1.4,-1.6);
\draw [->-] (n) to (1.85,-1);
\draw [->-] (f1) to (-2,-0.6);
\draw [->-] (f3) to (n);
\draw [->-] (f3) to (-1,1.85);
\draw [->-] (f3) to (1,1.85);
\draw [->-] (f4) to (n);
\draw [->-] (f4) to (1.85,1);
\draw [->-] (f6) to (n);
\draw [->-] (f6) to (0,-2.1);
\end{tikzpicture} \quad, \cO_{\Gamma'} \right)
\]

Note that each $f$ vertex with $k$ outgoing arrows, seen as a map $A[1]^{\otimes \dots} \to B[1]^{\otimes\dots}$, has degree $(d-3)(k-1) \equiv (d-1)k + (d-1) \pmod 2$; that is, in the ordering of all vertices and edges together, each $f$ vertex and each edge has the same weight $(d-1)$. Therefore the orientation is given by an ordering where switching the order of any two elements (vertices and/or edges) carries a sign $(-1)^{d-1}$.

We define the orientation on some $\Gamma$ such as the one drawn above by: first enumerating all the outgoing edges in cyclic order, then starting from the vertex connected to the first outgoing edge and enumerating all vertices/edges (except for $m$) going around a contour enclosing the tree (omitting repetitions), then ending with $m$. For example, for the graph $\Gamma$ on the left-hand side of the equation above we pick
\[\hspace{-0.7cm} \begin{tikzpicture}[baseline={([yshift=-.5ex]current bounding box.center)}]
\draw (0,0) circle (2.1);
\node [vertex] (m) at (0,0) {$m$};
\node [vertex] (f2) at (-1.3,-0.5) {$f_2$};
\node [vertex] (f3) at (0,+1.2) {$f_3$};
\node [vertex] (f4) at (1.1,0) {$f_4$};
\node [vertex] (f1) at (0,-1.3) {$f_1$};
\node (e1) at (0,-2.3) {$e_1$};
\node (e2) at (-1.4,-1.8) {$e_2$};
\node (e3) at (-2.2,-0.7) {$e_3$};
\node (e4) at (-2.2,0.7) {$e_4$};
\node (e5) at (-1.1,2) {$e_5$};
\node (e6) at (1,2) {$e_6$};
\node (e7) at (2.1,1.1) {$e_7$};
\node (e8) at (2.1,-1.1) {$e_8$};
\draw [->-] (m) to node [auto,swap] {$e_{10}$} (f2);
\draw [->-] (f2) to (-1.35,-1.6);
\draw [->-] (f2) to (-2,-0.6);
\draw [->-] (f2) to (-2,0.6);
\draw [->-] (m) to  node [auto] {$e_{11}$} (f3);
\draw [->-] (f3) to (-1,1.85);
\draw [->-] (f3) to (1,1.85);
\draw [->-] (m) to node [auto] {$e_{12}$} (f4);
\draw [->-] (f4) to (1.85,1);
\draw [->-] (f4) to (1.85,-1);
\draw [->-] (m) to node [auto] {$e_9$} (f1);
\draw [->-] (f1) to (0,-2.1);
\end{tikzpicture} \quad , \quad (e_1\ e_2\ e_3\ e_4\ e_5\ e_6\ e_7\ e_8\ f_1\ e_9\ e_{10}\ e_{11}\ f_3\ e_{12}\ f_4\ m) \]
We do the same for the diagrams on the right-hand side, ending with $n$ instead of $m$. This defines the signs for the compatibility condition.

We now do a similar procedure to define the signs in \cref{def:composition} of composition $h = g \circ f$ of pre-CY morphisms. Recall that each term in the expression for $h$ is given by some height two planar tree quiver in the disc, with arrows outgoing from $g$ vertices. We orient the diagram of each term in $h$ by again first enumerating all the outgoing edges in cyclic clockwise order, then all the vertices and remaining edges by drawing a clockwise contour around the tree. For example, one such diagram with its orientation will be given by
\[\hspace{-0.6cm} \begin{tikzpicture}[baseline={([yshift=-.5ex]current bounding box.center)}]
\draw (0,0) circle (2);
\node [vertex] (f1) at (-0.5,-0.2) {$f_1$};
\node [vertex] (f2) at (0.5,0.5) {$f_2$};
\node [vertex] (f3) at (1,-1) {$f_3$};
\node [vertex] (g1) at (0,-1) {$g_1$};
\node [vertex] (g2) at (-1.5,0) {$g_2$};
\node [vertex] (g3) at (0,1.1) {$g_3$};
\node [vertex] (g4) at (1.3,0) {$g_4$};
\node (e1) at (0,-2.3) {$e_1$};
\node (e2) at (-2,-1.2) {$e_2$};
\node (e3) at (-2,1.2) {$e_3$};
\node (e4) at (-1.1,2) {$e_4$};
\node (e5) at (1,2) {$e_5$};
\node (e6) at (2.2,0) {$e_6$};
\draw [->-] (f1) to node [auto,swap] {$e_7$} (g1);
\draw [->-] (f2) to node [auto] {$e_{10}$} (g1);
\draw [->-] (f1) to node [auto,swap,xshift=-6pt] {$e_8$} (g2);
\draw [->-] (f1) to node [auto] {$e_9$} (g3);
\draw [->-] (f3) to node [auto,swap,yshift=3pt] {$e_{12}$} (g4);
\draw [->-] (f2) to node [auto] {$e_{11}$} (g4);
\draw [->-] (g1) to (0,-2);
\draw [->-] (g2) to (-1.75,-1);
\draw [->-] (g2) to (-1.75,1);
\draw [->-] (g3) to (1,1.75);
\draw [->-] (g3) to (-1,1.75);
\draw [->-] (g4) to (2,0);
\end{tikzpicture} \quad , \quad
(e_1\ e_2\ e_3\ e_4\ e_5\ e_6\ g_1\ e_7\ f_1\ e_8\ g_2\ e_9\ g_3\ e_{10}\ f_2\ e_{11}\ g_4\ e_{12}\ f_3)
\]

We now are ready to prove the result we claimed in \cref{sec:categoryPreCYalgebras} that the definitions of compatibility between $m,n$ and $f$ above, and of the composition law, define a category of pre-CY algebras.
\begin{proof} (of \cref{prop:preCYcategory})
	Let $f: (A,m) \to (B,n), g:(B,n) \to (C,p)$ and $h:(C,p) \to (D,q)$ be two pre-CY morphisms between pre-CY algebras of dimension $d$. We prove first that $g\circ f$ is compatible with the pre-CY structures $m,p$. We write this compatibility equation; on the left-hand side we see a sum over planar tree quivers of height 3, with a single vertex labeled $m$ at height one. We use the compatibility relation between $f,m,n$ to rewrite this into the sum on the other side; for that, we need to check that the signs agree. For every diagram with its orientation in the form
	\[(e_1\ e_2 \dots f_i\ \dots g_i \dots m) \]
	the other side has a diagram with some orientation
	\[(e_1\ e_2 \dots f_i\ \dots g_i \dots n) \]
	We calculate that both of those orientation agree with an \emph{overall} cyclic ordering of the edges and $f,g$ vertices (followed by $m$ or $n$ at the end); this happens since to go between these orderings we always permute pairs of edge and vertex, both of which have equal weight $d-1$.
	
	A similar argument shows that $h \circ (g \circ f) = (h \circ g) \circ f$; the expression for both of these can be written as a sum over all height 3 planar tree quivers, with the overall cyclic orientation of edges and $f,g,h$ vertices.
\end{proof}

\subsection{The PROP of marked ribbon quivers}\label{sec:theClosedPROP}
We are now ready to define the \textsc{prop}s acting on Hochschild chains of pre-CY algebras and categories. For simplicity of presentation, let us first describe the `closed string' part of the \textsc{prop}; for each dimension $d$ this is a (single-colored) \textsc{prop} $Q^d$ which will act on Hochschild chains of pre-CY categories of dimension $d$. Later, in \cref{sec:openClosed}, we will explain how the general open-closed case works.

\subsubsection{Generators}
Let us fix any pair $m,n$ of \emph{positive} integers. Consider the set $RQ(m,n)$ of all marked ribbon quivers with $|V_\times| = m, |V_\circ| = n, V_\mathrm{open-in} = V_\mathrm{open-out} = \varnothing$. We will be inserting the pre-CY structure maps $m_{(k)}$ into every flow vertices of our quivers; as those have degree $|m_{(k)}| = 2$ in $C^*_{[d]}(\cA)$, we fix the degree of all flow vertices to be $|v| = 2$.

\begin{definition}
	As a graded vector space, the space $Q^d_{m,n}$ is defined as
	\[ Q^d(m,n) = \mathrm{Span}_\kk \left( \{ (\vec\Gamma, \cO)\ |\ \vec\Gamma \in RQ_{m,n}, \quad \cO\  d\mathrm{-orientation\ on\ } \vec\Gamma \} \right) / \sim, \]
	where $d$-orientation was defined in \cref{def:dOrientation}, and $\sim$ denotes the equivalence relation that sends $(\vec\Gamma,\cO) \mapsto -(\vec\Gamma,\cO^{op})$, with $\cO^{op}$ the opposite orientation to $\cO$. The generating vector $(\vec\Gamma,\cO)$ is placed in (cohomological) degree $-\deg(\vec\Gamma)$.
\end{definition}

Because each ribbon graph has a genus $g \ge 0$, we get decompositions $RQ_{m,n} = \bigsqcup_g RQ_g(m,n)$ and $Q^d(m,n) = \prod_g Q^d_g(m,n)$.

Let us present some examples in low genus, with their respective degrees:
\begin{example}
	Let us present some examples in genus zero, together with their (homological) degrees:
	\begin{alignat*}{4}
	& m=1, n=1 && && \\
		& (1)\ \ \begin{tikzpicture}[baseline={([yshift=-.5ex]current bounding box.center)}]
		\node [inner sep=0pt] (top) at (0,1) {$\times$};
		\node [circ] (bot) at (0,0) {};
		\draw [->-] (top) to (bot);
		\end{tikzpicture}\ \deg=0 \qquad \qquad
	&& (2)\ \ \begin{tikzpicture}[baseline={([yshift=-.5ex]current bounding box.center)}]
	\node [inner sep=0pt] (top) at (-0.5,1) {$\times$};
	\node [vertex] (side) at (0.5,0.5) {$1$};
	\node [circ] (bot) at (0,0) {};
	\draw [->-] (top) to (bot);
	\draw [->] (side) to (bot);
	\end{tikzpicture}\qquad \deg = 1 \qquad
	&& (3)\ \ \begin{tikzpicture}[baseline={([yshift=-.5ex]current bounding box.center)}]
	\node [inner sep=0pt] (x) at (-1.2,0.98) {$\times$};
	\node [bullet] (left) at (-0.6,0) {};
	\node [bullet] (down) at (0,-0.6) {};
	\node [bullet] (up) at (0,0.6) {};
	\draw [->-] (0,0.6) arc (90:180:0.6);
	\draw [->-] (-0.6,0) arc (180:270:0.6);
	\draw [-w-] (0,0.6) arc (90:-90:0.6);
	\node [circ] (bot) at (0,-1.2) {};
	\draw [->-] (down) to (bot);
	\draw [->-] (-1.2,1) to (left);
	\draw [dashed] (0,0) circle (0.4);
	\end{tikzpicture}\ \deg = -d \\
	& m=1, n=2 && && \\
	& (4)\ \ \begin{tikzpicture}[baseline={([yshift=-.5ex]current bounding box.center)}]
	\node [inner sep=0pt] (top) at (0,1) {$\times$};
	\node [bullet] (midleft) at (-0.5,0.5) {};
	\node [bullet] (midright) at (0.5,0) {};
	\node [circ] (left) at (-1,-0.5) {};
	\node [circ] (right) at (1,-0.5) {};
	\draw [->-, shorten <=-3.5pt] (top) to (midright);
	\draw [-w-] (midleft) to (midright);
	\draw [->-] (midleft) to (left);
	\draw [->-] (midright) to (right);
	\end{tikzpicture}\ \deg=-d \qquad
	&& (5)\ \ \begin{tikzpicture}[baseline={([yshift=-.5ex]current bounding box.center)}]
	\node [inner sep=0pt] (top) at (0,1) {$\times$};
	\node [bullet] (mid) at (0,0) {};
	\node [circ] (left) at (-1,-0.5) {};
	\node [circ] (right) at (1,-0.5) {};
	\draw [->-, shorten <=-3.5pt] (top) to (mid);
	\draw [->-] (mid) to (left);
	\draw [-w-] (mid) to (right);
	\end{tikzpicture}\ \deg=-d+1 && \\
	\end{alignat*}

	We give the ribbon structure from the embedding into the page, and indicate the orientation by ordering all vertices (by height) and choosing the first outgoing (white arrow) when there is ambiguity. Note that graph (3) has a boundary component (dashed circle) without any incoming or outgoing leg; this will be a \emph{free boundary} of the surface.
\end{example}

\subsubsection{The differential}
We now describe a differential on each of the spaces $Q^d_{m,n,g}$. Recall from \cref{def:degree} that the vertices contributing nonzero degree to a given marked ribbon quiver $\vec\Gamma$ are: (unmarked) sources with valence $>2$, flow vertices with $in >2$ and/or $out >1$, and $\circ$-marked sinks with valence $>1$.

For each such vertex $v$, draw it on the plane and draw a dashed curve separating the incident half-edges into two non-empty subsets.
\begin{definition}\label{def:separations}
	A \emph{separation} of the vertex $v$ is a ribbon quiver with two vertices $a,b$ obtained by splitting the vertex into two, and connecting them by an arrow $e:a \to b$, with the following conditions:
	\begin{enumerate}
		\item None of the resulting vertices have $in = out = 1$.
		\item If $v$ is not a sink, then $b$ is not a sink.
	\end{enumerate}
	For the distinguished edge at a $\circ$-sink, when we separate such a vertex, we end up with a flow vertex and a $\circ$-sink; if the previously distinguished edge now lands in the $\circ$-sink, we mark it. If not, we mark the \emph{new edge}.
	
	And as for the $\circ$-sinks with a $V_\mathbf{1}$-vertex attached, if such a vertex has valence 2 we just declare it to have no separations; otherwise, if the resulting graph has a $V_\mathbf{1}$-vertex attached to a flow vertex of valence 3, we delete the $V_\mathbf{1}$-vertex and its incident edge; if it has a $V_\mathbf{1}$-vertex attached to a flow vertex of valence $> 3$, we exclude this separation.
\end{definition}
For example, along the dashed curve below, there are two possible separations
\[
\begin{tikzpicture}[baseline={([yshift=-.5ex]current bounding box.center)}]
	\node [bullet] (v) at (0,0) {};
	\draw [->-=1] (v) to (1,0);
	\draw [->-=1] (v) to (0,1);
	\draw [->-=1] (v) to (-1,0);
	\draw [->-] (0.9,0.3) to (v);
	\draw [->-] (0.3,0.9) to (v);
	\draw [->-] (-0.7,0.7) to (v);
	\draw [->-] (0.1,-0.9) to (v);
	\draw [->-] (0.4,-0.6) to (v);
	\draw [dashed] (1,1) to (-1,-1);
\end{tikzpicture} \qquad \rightsquigarrow \qquad
\begin{tikzpicture}[baseline={([yshift=-.5ex]current bounding box.center)}]
	\node [bullet] (v1) at (0.2,-0.2) {};
	\node [bullet] (v2) at (-0.2,0.2) {};
	\draw [->-=1] (v1) to (1,0);
	\draw [->-] (0.1,-0.9) to (v1);
	\draw [->-] (0.5,-0.7) to (v1);

	\draw [->-=1] (v2) to (0,1);
	\draw [->-=1] (v2) to (-1,0);
	\draw [->-] (0.9,0.3) to (v1);
	\draw [->-] (0.3,0.9) to (v2);
	\draw [->-] (-0.7,0.7) to (v2);
	\draw [->-=0.7] (v1) to (v2);
\end{tikzpicture} \quad \mathrm{and} \quad
\begin{tikzpicture}[baseline={([yshift=-.5ex]current bounding box.center)}]
\node [bullet] (v1) at (0.2,-0.2) {};
\node [bullet] (v2) at (-0.2,0.2) {};
\draw [->-=1] (v1) to (1,0);
\draw [->-] (0.1,-0.9) to (v1);
\draw [->-] (0.5,-0.7) to (v1);

\draw [->-=1] (v2) to (0,1);
\draw [->-=1] (v2) to (-1,0);
\draw [->-] (0.9,0.3) to (v1);
\draw [->-] (0.3,0.9) to (v2);
\draw [->-] (-0.7,0.7) to (v2);
\draw [->-] (v2) to (v1);
\end{tikzpicture}
\]
In contrast, if we separate only incoming arrows to one side, by condition (2) we only have one separation
\[\begin{tikzpicture}[baseline={([yshift=-.5ex]current bounding box.center)}]
\node [bullet] (v) at (0,0) {};
\draw [->-=1] (v) to (1,0);
\draw [->-=1] (v) to (0,1);
\draw [->-=1] (v) to (-1,0);
\draw [->-] (0.9,0.3) to (v);
\draw [->-] (0.3,0.9) to (v);
\draw [->-] (-0.7,0.7) to (v);
\draw [->-] (0.1,-0.9) to (v);
\draw [->-] (0.4,-0.6) to (v);
\draw [dashed] (-0.8,-0.55) arc (160:20:0.8);
\end{tikzpicture} \qquad \rightsquigarrow \qquad
\begin{tikzpicture}[baseline={([yshift=-.5ex]current bounding box.center)}]
\node [bullet] (v1) at (0,-0.3) {};
\node [bullet] (v2) at (-0.2,0.2) {};
\draw [->-=1] (v2) to (1,0);
\draw [->-] (0.1,-0.9) to (v1);
\draw [->-] (0.5,-0.7) to (v1);
\draw [->-=1] (v2) to (0,1);
\draw [->-=1] (v2) to (-1,0);
\draw [->-] (0.9,0.3) to (v2);
\draw [->-] (0.3,0.9) to (v2);
\draw [->-] (-0.7,0.7) to (v2);
\draw [->-=0.7] (v1) to (v2);
\end{tikzpicture}
\]
Finally, to exemplify the special rules at $\circ$-sinks, we decreed that the vertex $\begin{tikzpicture}[baseline={([yshift=-.5ex]current bounding box.center)}]
\node [circ] (a) at (0,0) {};
\node [vertex] (one) at (1,0) {$1$};
\draw [->-] (-1,0) to (a);
\draw [->-] (one) to (a);
\end{tikzpicture}$ has no separations. With more incoming edges we do have separations at vertices with $\bm{1}$ attached, for example,
\[ \begin{tikzpicture}[baseline={([yshift=-.5ex]current bounding box.center)}]
\node [circ] (a) at (0,0) {};
\node [vertex] (one) at (1,0) {$1$};
\draw [->-] (-0.7,0.3) to (a);
\draw [->-] (-0.7,-0.3) to (a);
\draw [->-] (one) to (a);
\end{tikzpicture} \qquad \rightsquigarrow \qquad
 \begin{tikzpicture}[baseline={([yshift=-.5ex]current bounding box.center)}]
\node [bullet] (a) at (-0.4,0) {};
\node [circ] (b) at (0.2,0) {};
\node [vertex] (one) at (1,0) {$1$};
\draw [->-] (-0.8,0.3) to (a);
\draw [->-] (-0.8,-0.3) to (a);
\draw [->-] (a) to (b);
\draw [->-] (one) to (b);
\end{tikzpicture}, \quad
\begin{tikzpicture}[baseline={([yshift=-.5ex]current bounding box.center)}]
\node [circ] (b) at (0.2,0) {};
\draw [-w-=1] (-0.8,0.3) to (b);
\draw [->-] (-0.8,-0.3) to (b);
\end{tikzpicture}, \quad \mathrm{and} \quad
\begin{tikzpicture}[baseline={([yshift=-.5ex]current bounding box.center)}]
\node [circ] (b) at (0.2,0) {};
\draw [->-] (-0.8,0.3) to (b);
\draw [-w-=1] (-0.8,-0.3) to (b);
\end{tikzpicture}
\]

We now define the differential on $Q^d_{m,n,g}$ by defining it on the basis elements given by a marked ribbon quiver $\vec\Gamma$ together with some orientation. We first put the orientation into some normal form
\[ (o_1 o_2 \dots o_n \dots 1_i \dots v_N \dots v_1 x_m \dots x_1 ) \]
with all the $\circ$ sinks $o_i$ before all the $V_1$ vertices $1_i$, before all the flows and unmarked sources $v_i$, before all the $\times$ sources $x_i$.

\begin{definition}
	The differential $\del$ of the element above is given by
	\begin{align*}
	\hspace{-0.8cm} \del((\vec\Gamma,&((o_1 o_2 \dots o_n \dots 1_i \dots v_N \dots \mathbf{v} \dots v_1 \dots x_m \dots x_1))) = \\
	&\sum_{\substack{v \in V(\Gamma) \\ (e:a\to b)\mathrm{separation\ of\ }v}} \left( \vec\Gamma_{(e:a\to b)}, (o_1 o_2 \dots o_n \mathbf{e\ a} \dots 1_i \dots v_N \dots \mathbf{b} \dots v_1 \dots x_m \dots x_1)\right)
	\end{align*}
	In other words, we sum over all separations of all vertices that can be separated (that is, all vertices with nonzero degree), with the orientation such that the extra vertex $a$ and extra edge $e$ created are always immediately after the output vertices, with the other vertex $b$ replacing $v$.
	
	As for the special case of separations involving $V_\mathbf{1}$-vertices, if we deleted a $V_\mathbf{1}$-vertex attached as
	\[ \begin{tikzpicture}[baseline={([yshift=-.5ex]current bounding box.center)}]
	\node [bullet] (a) at (0,0.6) {};
	\node [circ] (b) at (0,0) {};
	\node [vertex] (one) at (-0.8,0.6) {$1$};
	\draw [->-] (0,1.2) to (a);
	\draw [->-] (one) to (a);
	\draw [->-] (a) to (b);
	\end{tikzpicture}  \quad \rightsquigarrow \quad
	\begin{tikzpicture}[baseline={([yshift=-.5ex]current bounding box.center)}]
	\node [circ] (b) at (0,0) {};
	\draw [->-] (0,1.2) to (b);
	\end{tikzpicture}\]
	we just produce the orientation given by deleting the vertex $\bm{1}$ and its incident edge; if we instead deleted
	\[ \begin{tikzpicture}[baseline={([yshift=-.5ex]current bounding box.center)}]
	\node [bullet] (a) at (0,0.6) {};
	\node [circ] (b) at (0,0) {};
	\node [vertex] (one) at (0.8,0.6) {$1$};
	\draw [->-] (0,1.2) to (a);
	\draw [->-] (one) to (a);
	\draw [->-] (a) to (b);
	\end{tikzpicture}  \quad \rightsquigarrow \quad
	\begin{tikzpicture}[baseline={([yshift=-.5ex]current bounding box.center)}]
	\node [circ] (b) at (0,0) {};
	\draw [->-] (0,1.2) to (b);
	\end{tikzpicture}\]
	we produce \emph{minus} that orientation.
\end{definition}

\begin{remark}
	The rules for the separations involving $V_\mathbf{1}$-vertices might seem arbitrary, but in our \textsc{prop} action we will input strict units at these vertices; the conventions in fact they follow from the relations satisfied by those strict units. For example, the reason for excluding separations of vertices with valence 2 and one $\bm{1}$-vertex attached is that when evaluating the ribbon quiver with the $A_\infty$-maps, the differential above would give a term $\mu^2(1,x) + (-1)^{\bar x} \mu^2(x,1)$ = 0. 
\end{remark}

\begin{lemma}
	The differential $\del$ squares to zero.
\end{lemma}
\begin{proof}
	Any two consecutive separations can be performed in either order; resulting in the same graph but with orientations given by
	\[ (o_1 o_2 \dots o_n \mathbf{e_2\ a_2\ e_1\ a_1} \dots 1_i  \dots v_N \dots \mathbf{b_1} \dots \mathbf{b_2} v_1 \dots x_m \dots x_1) \]
	if we do the separation $(e_1:a_1\to b_1)$ and then $(e_1:a_1\to b_1)$, and
	\[ (o_1 o_2 \dots o_n \mathbf{e_1\ a_1\ e_2\ a_2} \dots 1_i \dots v_N \dots \mathbf{b_1} \dots \mathbf{b_2} v_1 \dots x_m \dots x_1) \]
	if we do them in the opposite order. The sign difference between these is from switching a pair of an edge and vertex with another pair, giving $(-1)^{(d+d+1)(d+d+1)} = -1$.
\end{proof}

\subsubsection{Isomorphisms between even and odd dimensions}\label{sec:isoEvenOdd}
We defined each of the complexes $\{Q^d(m,n)\}$ separately for each choice of integer $d$; changing $d$ shifts the degree assigned to each marked ribbon quivers in $Q^d(m,n)$, and also changes the signs in the differential.

Upon fixing the number of inputs $m$, the number of outputs $n$ and the genus $g$, for any two integers $d_1,d_2$ the complexes $Q^{d_1}_g(m,n)$ and $Q^{d_2}_g(m,n)$ are spanned by the same marked ribbon quivers, with an overall shift depending on those integers. Moreover, if $d_1,d_2$ have the same parity, the signs in the differentials are all the same, so we have that:
\begin{proposition}
	If $d_1 \equiv d_2 \pmod{2}$, the complexes $Q^{d_1}_g(m,n)$ and $Q^{d_2}_g(m,n)$ are isomorphic up to shift.
\end{proposition}

It turns out that at partial form of the result above also holds between dimensions $d_1,d_2$ of different parity, but only once we restrict to ribbon quivers corresponding to surfaces \emph{without free boundary circles}. We will be more precise about the relations between marked ribbon quivers and open-closed surfaces in \cref{sec:Strebel}, but for now we will say that a free boundary circle of $\Gamma$ is a boundary component of the surface $\Sigma_\Gamma$ that has no vertex of valence one neighboring it.

\begin{definition}
	For any $d,g,m,n$, the subcomplex of marked ribbon quivers \emph{without free boundaries}
	\[ Q^d_{g,F=0}(m,n) \subset Q^d_g(m,n) \]
	is the subcomplex spanned by the marked quivers with no free boundary circles.
\end{definition}
Note that the differential preserves the number of free boundary circles, so it preserves the subcomplex $Q^d_{g,F=0}(m,n)$. We will now argue that the complexes $Q^d_{g,F=0}(m,n)$ is independent of the dimension $d$, up to shift; this is seen by a computation involving the signs in the differential, which we now explain.

Recall that when $d$ is even, a $d$-orientation on a marked ribbon quiver $\vec\Gamma$ is an ordering of the \emph{edges} of $\vec\Gamma$ modulo the alternating group, and when $d$ is odd, it is an ordering of the \emph{vertices} of $\vec\Gamma$ modulo the alternating group. The sets of orientations for either the edges or the vertices are both abstractly isomorphic to $\ZZ_2$ so in order to define a bijection between orientations, it is enough to determine a single sign between a fixed ordering of edges and a fixed ordering of vertices.

We now choose a pair of those fixed orderings. Recall that the surface $\Sigma_\Gamma$ is divided by $\Gamma$ into regions; as long as there are no free boundary circles, each such region corresponds to a $\times$-vertex and is homeomorphic to a disc with a cut going from the $\times$-vertex to the boundary. The boundary of each such disc is made up of a sequence of edges and vertices, and possibly some trees attached to some of these vertices.

We now fix an ordering of the $\times$-vertices and an ordering of the $\circ$-vertices (corresponding to fixing the ordering on the inputs/outputs of the \textsc{prop} operation); let that ordering be
\[ (x_1, \dots, x_m), \qquad (o_1,\dots,o_n) \]
respectively.

On each region corresponding to the source $x_i$, we start from $x_i$ along the edge incident to it, and go around the perimeter of the disc. We record the data of the edges along the boundary of this perimeter in the following way: we fill a matrix with two rows, one column at a time, such that
\begin{itemize}
	\item when we encounter an new vertex followed by an outgoing edge with the same orientation of our boundary walk, we write that vertex and edge as a new column, and
	\item when we encounter a vertex followed by an outgoing edge with the same orientation of our boundary walk, but \emph{we already encountered that vertex before}, we write a new column leaving the top entry empty.
\end{itemize}
We perform the operation above for each $\times$-vertex $x_i$ in sequence, keeping the columns as we move to the next disc.

Because the surface $\Sigma_\Gamma$ is orientable, if there are no free boundary circles we guarantee that every edge appears exactly once in this matrix; and every vertex that is not a sink ($\circ$-vertex) also appears exactly once.

\begin{remark}
	Note that if, on the other hand, there are a nonzero number of free boundary circles, it is possible that the procedure above will miss some edges; namely, every edge that is part of the clockwise-oriented cycle around each free boundary circle.
\end{remark}

Once we are done, this matrix contains each edge and vertex exactly once, except for the sink $\circ$-vertices. We now add those back in the beginning, according to their order, and define $\Sgn(\mathrm{sinks,\ top\ row})$ to be the \emph{fixed vertex orientation}. As for the edges, we define $\Sgn(\mathrm{bottom\ row})$ to be the \emph{fixed edge orientation}. Note that these depend only of the marked ribbon quiver $\vec\Gamma$ and on the ordering of the $\times$- and $\circ$-vertices.

We now define a sign between these orientations, in the following way:
\begin{itemize}
	\item For each column that has a vertex and an edge, we record the number of nonempty entries of the matrix \emph{before} that column,
	\item For each column that has just an edge, we record the number of nonempty entries of the matrix \emph{after} that column, and 
	\item For each column that has just a vertex, we record zero. 
\end{itemize}
We then sum over all these numbers to get an integer $S$.

\begin{definition}\label{def:fixedOrientations}
	The isomorphism between the sets of vertex orientations and edge orientations is given by relating the fixed orientations we specified by
	\[ \Sgn(\mathrm{top\ row}) = (-1)^S \Sgn(\mathrm{bottom\ row}). \]
\end{definition}

Doing this for every marked ribbon quiver $\vec\Gamma$ gives isomorphisms of graded vector spaces
\[ Q^{\mathrm{even}}_{g,F=0}(m,n) \cong Q^{\mathrm{odd}}_{g,F=0}(m,n)[s], \]
for some shift $s$.
\begin{proposition}
	The isomorphisms above intertwine the differential $\del$, and thus give isomorphisms of complexes between all the $Q^d_{g,F=0}(m,n)$ for varying $d$, up to shift.
\end{proposition}
\begin{proof}
	The proof consists in checking the sign introduced by the differential $\del$ with respect to our fixed orientations. Recall the definition of $\del$: for each vertex $v$ that is either a flow vertex or a sink we consider all separations $e:a \to b$ at $v$, and sum over all of those. For the orientation, starting  an orientation given by inserting the new higher vertex $a$ and the new edge $e$ between all the sinks and the other vertices:
	\[(o_1 o_2 \dots o_n \boldsymbol{e}\ \boldsymbol{a} \dots 1_i \dots v_N \dots \boldsymbol{b} \dots v_1 \dots x_m \dots x_1) \]
	
	Let $\Gamma$ be our starting graph and $\Gamma' = \Gamma_{e:a\to b}$ be one of its separations at $v$. Starting now with the fixed orientation (we got by identifying the fixed edge and vertex orientations) for $\Gamma$, we can permute it to some normal form (getting two signs for vertices and edges), apply the separation 
	and then permute the vertices/edges other than $a$ and $e$ again to the fixed orientations (getting those same two signs for vertices and edges).
	
	We then compare the obtained orientations with the fixed orientations of $\Gamma'$, and calculate that the relative sign between the new edge orientation and the new vertex orientation is always $+1$. This involves checking each possible configuration of edge directions around the new edge $e$, together with each possible ordering of the regions around it. 
	
	Essentially there are three distinct cases to be checked, depending on whether in the matrix the new edge and vertex $e$ and $a$ appear together (in the same column) or separately, and if not whether the vertex $b$ appears right after that column or not. We now do one of those cases explicitly for the sake of clarity.
	
	Suppose that the region of the surface $\Sigma_\Gamma$ around $v$, and the separation $e:a\to b$ of $v$, look like the following drawing:
	\[\begin{tikzpicture}[baseline={([yshift=-.5ex]current bounding box.center)}]
	\node [bullet, label=below:{$\bm{v}$}] (v) at (0,0) {};
	\node [bullet, label=below:{$v_1$}] (v1) at (-1.5,-0.5) {};
	\node [bullet, label=below:{$v_2$}] (v2) at (1.5,-0.5) {};
	\node [bullet, label=above:{$v_3$}] (v3) at (0,+1) {};
	\node [bullet, label=below:{$v_4$}] (v4) at (-1.5,0.5) {};
	\node [bullet, label=below:{$v_5$}] (v5) at (1.5,0.5) {};
	\draw [->-] (v1) to node [auto,swap] {$e_1$} (v);
	\draw [->-] (v) to node [auto,swap] {$e_2$} (v2);
	\draw [->-] (v3) to node [auto] {$e_3$} (v);
	\draw [->-] (v) to node [auto] {$e_4$} (v4);
	\draw [->-] (v5) to node [auto] {$e_5$} (v);
	\node [vertex] at (0,-1) {I};
	\node [vertex] at (-0.7,1) {II};
	\node [vertex] at (0.7,1) {III};
	\end{tikzpicture} \qquad \rightsquigarrow \qquad
	\begin{tikzpicture}[baseline={([yshift=-.5ex]current bounding box.center)}]
	\node [bullet, label=below:{$\bm{b}$}] (b) at (0,0) {};
	\node [bullet, label=below:{$\bm{a}$}] (a) at (1,0) {};
	\node [bullet, label=below:{$v_1$}] (v1) at (-1.5,-0.5) {};
	\node [bullet, label=below:{$v_2$}] (v2) at (2.5,-0.5) {};
	\node [bullet, label=above:{$v_3$}] (v3) at (0,+1) {};
	\node [bullet, label=below:{$v_4$}] (v4) at (-1.5,0.5) {};
	\node [bullet, label=below:{$v_5$}] (v5) at (2.5,0.5) {};
	\draw [->-] (v1) to node [auto,swap] {$e_1$} (b);
	\draw [->-] (a) to node [auto,swap] {$e_2$} (v2);
	\draw [->-] (v3) to node [auto] {$e_3$} (b);
	\draw [->-] (b) to node [auto] {$e_4$} (v4);
	\draw [->-] (v5) to node [auto] {$e_5$} (a);
	\draw [->-] (a) to node [auto] {$\bm{e}$} (b);
	\node [vertex] at (0,-1) {I};
	\node [vertex] at (-0.7,1) {II};
	\node [vertex] at (1,1) {III};
	\end{tikzpicture}
	\]
	On the left we have a piece of $\Gamma$ and on the right a piece of $\Gamma'$,
	and the roman numerals indicate in which order those three regions around $v$ appear in our fixed orientations.
	
	We now make the matrix of fixed orientations for $\Gamma$:
	\[\begin{bmatrix}
		\dots & v_1 & \bm{v} & \dots & v_3 & \fbox{\phantom{X}} & \dots & v_5 & \dots \\
		\dots & e_1 & e_2 & \dots & e_3 & e_4 & \dots & e_5 & \dots
	\end{bmatrix}\]
	and for $\Gamma'$:
	\[\begin{bmatrix}
	\dots & v_1 & \bm{a} & \dots & v_3 & \bm{b} & \dots & v_5 & \fbox{\phantom{X}} & \dots \\
	\dots & e_1 & e_2 & \dots & e_3 & e_4 & \dots & e_5 & \bm{e} & \dots
	\end{bmatrix}\]
	where we emphasize the vertices and edge being separated for convenience.
	
	Using the sign prescription above gives us two signs $S_\Gamma$ and $S_{\Gamma'}$; we now calculate their difference to be equal to
	\[ \#(\dots v_1) + \#(a \dots v_3) + \#(\dots e_1 \dots e_5) + 1 \]
	modulo 2, where each $\#$ term is the length of the indicated string in the matrix. We note that this is exactly the sign of the permutation that brings $a$ and $e$ to the beginning of their respective rows, and $b$ to where $v$ was in the top row for $\Gamma$.
	
	Therefore we have that both the differential for even and odd $d$, applied to the fixed edge/vertex orientation that we identified in \cref{def:fixedOrientations}, give $\Gamma'$, also in its fixed edge/vertex orientation, with the same sign. We then check other cases (with different ordering of the regions and orientations of $e_1,\dots,e_5$) and get analogous results.
\end{proof}

\subsubsection{Compositions}
Finally, we must discuss compositions; in a \textsc{prop} one can compose along any number of outputs and inputs, without connectedness restrictions. We now describe how to compose marked ribbon quivers with orientations.

To describe the composition, one must consider the topology of the surface $\Sigma_\Gamma$ associated to the ribbon graph. We embed $\Gamma$ into this surface and for each $\times$ source $x_i$, we consider the region of $\Sigma_\Gamma \setminus \Gamma$ adjacent to it. We interpret this region as a disc with a cut from the boundary to the center, going along the edge incident at $x_i$.

We call the boundary of this disc the \emph{boundary cycle} $B_{x_i}$ associated to the source $x_i$. Note that this boundary cycle might include the same edge of $\Gamma$ once or twice (on opposite sides), and might include a vertex of $\Gamma$ multiple times.

Note also that different $\times$-sources $x_i,x_j$ have adjacent regions that are disjoint from one another; also, even though their boundary cycles might overlap, each angle around a vertex of $\Gamma$ is at most associated to one $\times$-source.

Let $\Gamma_1$ be some graph with a $\circ$-vertex $o$ and $\Gamma_2$ with a $\times$-vertex $x$. By definition, a $\times$-vertex has only one outgoing arrow, but a $\circ$-vertex has any number $k \ge 1$ of incoming arrows, one of which is distinguished. A composition of $\Gamma_1$ and $\Gamma_2$ at the pair $(o,x)$ is given by:
\begin{itemize}
	\item Deleting the $\circ$-vertex $o$ from $\Gamma_1$ and the $\times$-vertex $x$ from $\Gamma_2$,
	\item Connecting the marked arrow in $\Gamma_1$ to the arrow in $\Gamma_2$ leaving the removed vertex $\times$
	\item Connecting the $k-1$ other arrows in $\Gamma_1$ to vertices in the boundary cycle of $x$, \emph{respecting the cyclic ordering}.
\end{itemize}
One checks that this operation produces another marked ribbon quiver, and that it is additive on degrees. We show an example in \cref{fig:composition}. Moreover, since all the regions adjacent to the $\times$-vertices are disjoint, we can compose along multiple pairs as above, by performing the connections in each of those regions.

\begin{figure}[h!]
	\centering
	\includegraphics[width=0.9\textwidth]{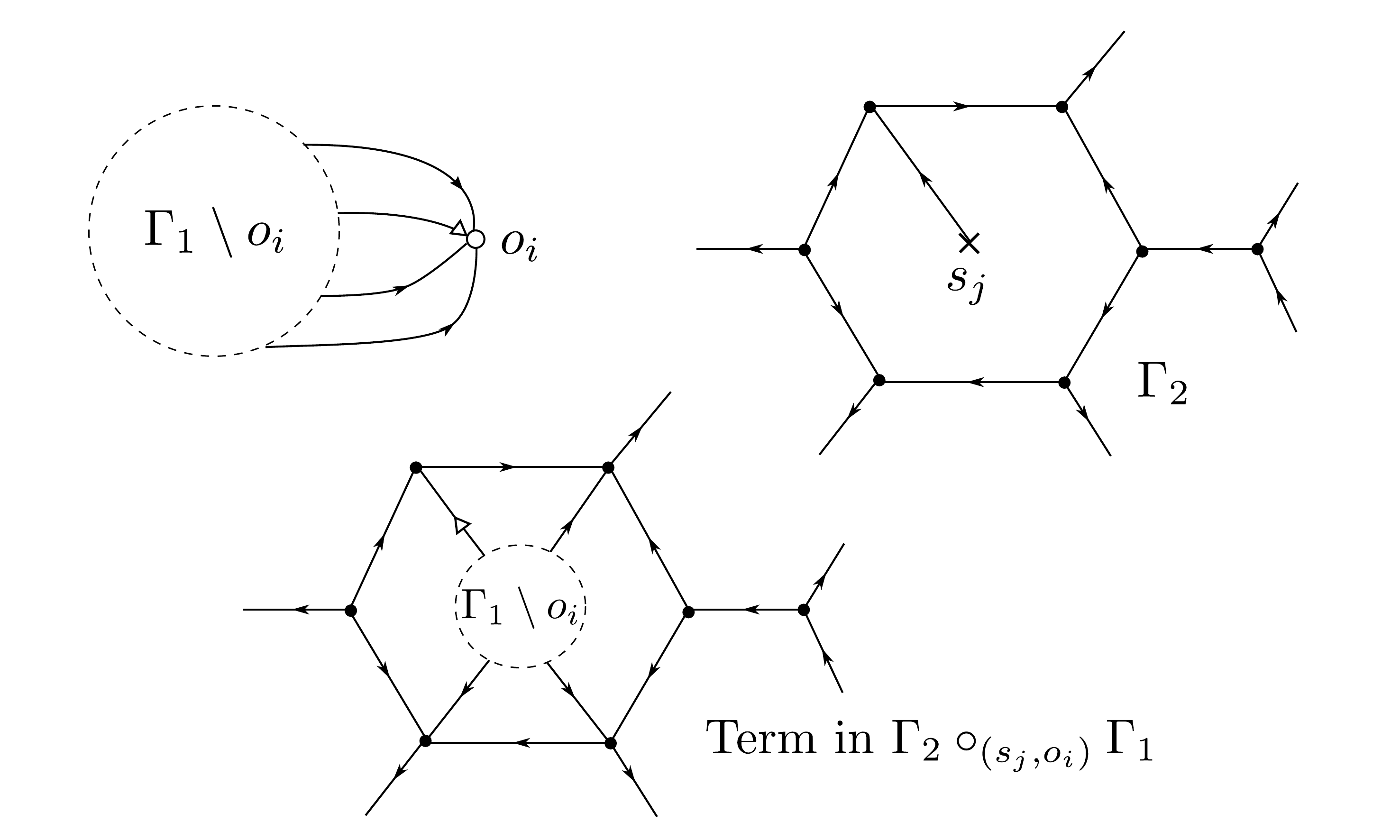}
	\caption{Composition of two marked ribbon quivers with one closed input and one closed output (i.e. $(m,n) = (1,1)$. The composition at the pair $(s_j,o_i)$ of an input and output is a sum over ways of distributing the three arrows incident at $o_i$ around the boundary cycle of $s_j$; the distinguished arrow (white triangle) always gets connected to the unique edge leaving $s_j$.}
	\label{fig:composition}
\end{figure}

We must now describe how to compose orientations. Note that composition at each pair $(o_i,x_i)$ deletes the two vertices and replaces the edge $e_i$ incident at $x_i$ with an edge connecting $\Gamma_1$ to $\Gamma_2$. We define the composition of orientations at a pair $(o_i,x_i)$ as follows:
\[ (\dots e_i x_i) \circ (o_i \dots) = (\dots e_i \dots) \]
and if there are multiple pairs, we equally define
\[ (\dots e_3 x_3 e_2 x_2 e_1 x_1) \circ (o_1 o_2 o_3 \dots) = (\dots e_3 e_2 e_1) \]
Note that the formula above is equivariant with respect to changing the order of compositions; transposing $(e_i x_i e_j x_j) \leftrightarrow (e_j x_j e_i x_i)$ gives a $-1$ sign, $(o_i o_j) \leftrightarrow (o_j o_i)$ gives $(-1)^d$ and $(e_i e_j) \leftrightarrow (e_j e_i)$ gives $(-1)^{d+1}$.

\begin{definition}
	The composition of two elements $(\Gamma_1, \cO_1) \in Q^d(m_1,n_1)$ and $(\Gamma_2, \cO_2) \in Q^d(m_2,n_2)$ at some number of pairs $(o_1,x_1), \dots, (o_p,x_p)$ is given by the linear combination
	\[ (\Gamma_1, \cO_1) \circ_{\{(o_i,x_i)\}} (\Gamma_2, \cO_2) = \sum_{\mathrm{compositions\ } \Gamma} (\Gamma, \cO_2 \circ \cO_1) \]
	where we sum over all compositions, i.e. all ways of distributing the extra incident vertices on $o_i$ around the boundary cycle of $x_i$, with the orientation given by the composition as above.
\end{definition}

The following result follows directly from checking the axioms of a symmetric monoidal category.
\begin{proposition}
	The spaces $\{Q^d(m,n)\}$ for all $m,n \ge 1$ form a dg \textsc{prop}, that is, a symmetric monoidal category enriched over the category of chain complexes $\mathrm{Ch}_\kk$, with object monoid given by the natural numbers.
\end{proposition}
More precisely, each $Q^d_g(m,n)$ is the space of connected operations of genus $g$ in the \textsc{prop}; to get all operations (or equivalently, the morphisms spaces of the symmetric monoidal category) one must allow disconnected ribbon quivers, which means allowing tensor products of any number of spaces $Q^d_g(m,n)$.

\subsection{Action of the PROP on Hochschild chains}
Let $\cA$ be an pre-CY category of dimension $d$, with pre-CY structure $\{m_{(k)}\}$. As before we will call $\mu = m_{(1)}$ its $A_\infty$-structure. Recall that in \cref{sec:action} we describe how a marked ribbon quiver acts on Hochschild chains, once we input higher cyclic cochains into the vertices of the quiver.

Let us be more precise about which Hochschild complex we will use. We assume the category $\cA$ is homologically unital as an $A_\infty$-category; we extend it to an $A_\infty$-equivalent, strictly unital $A_\infty$-category $\cA^+$, with strict units $1_X \in \cA(X,X)$ for each $X$.

The Hochschild chain complex $C_*(\cA)$ has a subcomplex spanned by all terms $a_0 \otimes a_1 \otimes \dots \otimes a_p$ where some $a_{i \neq 0} = 1_X$ for some $X$. The \emph{nonunital Hochschild chain complex} $C^{nu}_*(\cA)$ is then defined as the quotient of $C_*(\cA^+)$ by this subcomplex; as $\cA$ is homologically unital the composition
\[ C_*(\cA) \hookrightarrow C_*(\cA^+) \twoheadrightarrow C^{nu}_*(\cA) \]
is a quasi-isomorphism. So for simplicity of notation we will just denote $C_*(\cA)$ for this non-unital complex, which allows us to work with the strict units in $\cA^+$.

We now input the pre-CY structure map $m_{(k)}$ into every vertex with $k$ outgoing edges, and get a map of graded vector spaces
\[ \Phi: Q^d(m,n) \otimes C_*(\cA)^{\otimes m} \to C_*(\cA)^{\otimes n} \]

We have the Hochschild differential $b$ acting on $C_*(\cA)$ and the differential $\del$ acting on $Q^d(m,n)$ as we defined using separations.
\begin{theorem}\label{thm:actionClosedPROP}
	The map $\Phi$ commutes with the differential and defines a map
	\[ H^*(Q^d(m,n)) \otimes HH_*(\cA)^{\otimes m} \to HH_*(\cA)^{\otimes n}. \]
\end{theorem}
\begin{proof}
	It will be easier to shift the Hochschild complex, so we will instead describe an action
	\[ \Phi: Q^d(m,n)[n-m] \otimes (C_*(\cA)[1])^{\otimes m} \to (C_*(\cA)[1])^{\otimes n} \]
	In other words, given Hochschild chains $a^i, 1 \le i \le m$, we want to prove the following identity:
	\[ \Phi(\del \vec\Gamma)(a^1,\dots, a^m) + (-1)^{\deg(\Gamma) + m - n} \Phi(\vec\Gamma)\circ b (a^1,\dots,a^m) = b \circ\Phi(\vec\Gamma)(a^1,\dots, a^m) \]
	The action of a graph on a Hochschild chain is as described in \cref{sec:action}; in order to prove the identity above we have to understand how to express the Hochschild differential $b$, applied before and after $\phi(\vec\Gamma)$, in terms of modifications to $\vec\Gamma$.

	We describe the ribbon quivers corresponding to $\Phi(\vec\Gamma)\circ b$: for the component of the Hochschild differential on the $i$th chain in $(C_*(\cA)[1])^{\otimes m}$, we sum over insertions of a vertex
	$\begin{tikzpicture}[baseline={([yshift=-.5ex]current bounding box.center)}]
	\node [vertex] (a) at (0,0) {$\mu$};
	\draw [->-] (a) to (1,0);
	\end{tikzpicture}$
	attached to all angles around the boundary cycle of the $i$th $\times$-source of $\Gamma$, added to an insertion of a vertex $\begin{tikzpicture}[baseline={([yshift=-.5ex]current bounding box.center)}]
	\node [vertex] (a) at (0,0) {$\mu$};
	\draw [->-] (a) to (1,0);
	\draw [->-] (-1,0) to (a);
	\end{tikzpicture}$ along the edge incident at that $\times$-source.

	We orient all those graphs in the following way: if $\Gamma$ had an orientation in normal form
	\[ (o_1 o_2 \dots o_n \dots v_N \dots v_1\  e_m \dots x_m \dots e_1 \dots x_1 ) \]
	where $e_1,\dots, e_m$ are the edges incident at the sources $x_1,\dots, x_m$, for each modified graph we insert the new vertex $\mu$ and its outgoing edge $e$ as
	\[ (o_1 o_2 \dots o_n \dots v_N \dots v_1\ \bm{e\ \mu}\  e_m \dots x_m \dots e_1 \dots x_1 ) \]
	between the $\times$-sources and their edges, and all the other edges and vertices.

	Now we sum over all these new ribbon quivers with the orientation above. We see that the effect of each new vertex $\mu$ is to exactly precede the application of $\Phi(\Gamma)$ by an operation
	\begin{align*}
	\hspace{-2cm} a^1_0 &\otimes a^1_1 \otimes \dots \otimes a^1_{p_1} \otimes \dots \otimes a^m_0 \otimes a^m_1 \otimes \dots \otimes a^m_{p_m} \mapsto \\
	& a^1_0 \otimes a^1_1 \otimes \dots \otimes a^1_{p_1} \otimes \dots \otimes a^i_0 \otimes a^i_1 \otimes \dots \otimes \bm{\mu}(a^i_j,\dots) \otimes \dots \otimes  a^i_{p_i} \otimes \dots \otimes a^m_0 \otimes a^m_1 \otimes \dots \otimes a^m_{p_m}
	\end{align*}
	and the sum over such operations is the (shifted) Hochschild differential $b$ on $(C_*(\cA)[1])^{\otimes m}$.

	Now we describe the ribbon quivers corresponding to $b\circ\Phi(\vec\Gamma)$: for the $i$th $\circ$-vertex $o_i$, we sum over all insertions of 	$\begin{tikzpicture}[baseline={([yshift=-.5ex]current bounding box.center)}]
	\node [vertex] (a) at (0,0) {$\mu$};
	\draw [->-] (a) to (1,0);
	\end{tikzpicture}$ around the angles of $o_i$ and also over all insertions of $\begin{tikzpicture}[baseline={([yshift=-.5ex]current bounding box.center)}]
	\node [vertex] (a) at (0,0) {$\mu$};
	\draw [->-] (a) to (1,0);
	\draw [->-] (-1,0) to (a);
	\end{tikzpicture}$ along the edges incident at $o_i$.

	We orient all those graphs in a similar way, by placing the new vertex $\mu$ and its outgoing edge as follows
	\[ (o_1 o_2 \dots o_n\ \bm{e\ \mu}\ \dots v_N \dots v_1\ e_m \dots x_m \dots e_1 \dots x_1 ) \]
	that is, right between the $\circ$-outputs and all the other elements. We sum over all these new ribbon quivers with this orientation, and observe that the effect is \emph{almost} to follow the application of $\Phi(\Gamma)$ by $b$ on $(C_*(\cA)[1])^{\otimes n}$: what we are missing are terms where some subset of size $\ge 2$ of the $k$ incoming edges to $o_i$ themselves get input into $\mu$; these are exactly the separations of a sink in \cref{def:separations}, which appear in $\del\Gamma$.

	We now turn to the flow vertices. The Maurer-Cartan equation satisfied by the elements $\{m_{(k)}\}$ is
	\[ \sum_{i + j = k+1} m_{(i)} \circnec m_{(j)} = 0 \]
	for every $k \ge 1$. Consider now some vertex $v \in \vec\Gamma$, with $k$ outgoing edges and $\ell$ incoming edges. Two of the types of terms in the equation above are terms with $j=1$ or $i=1$; by calculating their orientation they enter in the equation above respectively as
	\[
	\sum_{1 \le p \le k}\begin{tikzpicture}[baseline={([yshift=-.5ex]current bounding box.center)}]
	\draw (0,0) circle (1.5);
	\node [vertex] (psi) at (0,0.6) {$\mu$};
	\node [vertex] (phi) at (0,-0.5) {$m_{(k)}$};
	\draw [-w-] (phi) to (0,-1.6);
	\node at (0,-2) {$e_1$};
	\draw [->-] (phi) to (-1.4,-1);
	\node at (-1.5,-1.1) {$e_2$};
	\draw [->-] (phi) to (-1.6,0.5);
	\node at (-1.8,0.55) {$e_p$};
	\node at (-0.8,-0.3) {$\vdots$};
	\draw [->-] (phi) to (1.4,-1);
	\node at (1.7,-1.1) {$e_k$};
	\node at (0.8,-0.3) {$\vdots$};
	\draw [->-] (phi) to (1.6,0.5);
	\node at (2,.55) {$e_{p+1}$};
	\draw [->-] (psi) to node [auto] {$e$} (phi);
	\end{tikzpicture} +
	\sum_{1 \le p \le k}(-1)^{(k-1)(d-1)} \begin{tikzpicture}[baseline={([yshift=-.5ex]current bounding box.center)}]
	\draw (0,0) circle (1.5);
	\node [vertex] (psi) at (0,0.6) {$\mu$};
	\node [vertex] (phi) at (0,-0.5) {$m_{(k)}$};
	\draw [-w-] (phi) to (0,-1.6);
	\node at (0,-2) {$e_1$};
	\draw [->-] (phi) to (-1.4,-1);
	\node at (-1.5,-1.1) {$e_2$};
	\draw [->-] (phi) to (-1.6,0.5);
	\node at (-1.9,0.55) {$e_{p-1}$};
	\node at (-0.8,-0.3) {$\vdots$};
	\draw [->-] (phi) to (1.4,-1);
	\node at (1.7,-1.1) {$e_k$};
	\node at (0.8,-0.3) {$\vdots$};
	\draw [->-] (phi) to (1.6,0.5);
	\node at (2,.55) {$e_{p+1}$};
	\draw [->-] (phi) to node [auto] {$e$} (psi);
	\draw [->-] (psi) to (0,1.6);
	\node at (0,1.8) {$e_p$};
	\end{tikzpicture}
	\]
	in terms of the ribbon quiver, the first type of term corresponds to the sum of two types of modifications to the ribbon quiver: adding
	$\begin{tikzpicture}[baseline={([yshift=-.5ex]current bounding box.center)}]
	\node [vertex] (a) at (0,0) {$\mu$};
	\draw [->-] (a) to (1,0);
	\end{tikzpicture}$ in angles around $v$, and also to separations of $v$ where one side of the dashed line in \cref{def:separations} only has incoming arrows. Note that the sign for all terms in the second sum is the same, and that
	\[ (k-1)(d-1) = dk -d -k +1 \equiv \bar v + 1 \]
	where again $\bar v = \bar m_{(k)} = dk - d + 3k -4$ is the degree of the vertex as a map from copies of $\cA[1]$ to copies of $\cA[1]$.

	All the other separations of $v$ appearing in $\del\Gamma$ correspond to the other terms in the necklace equation above, with $i,j \ge 2$. The last calculation we need is to analyze the effect on the orientation of the operation that does not change anything about the ribbon quiver, but commutes the order between some new vertex $\begin{tikzpicture}[baseline={([yshift=-.5ex]current bounding box.center)}]
	\node [vertex] (a) at (0,0) {$\mu$};
	\draw [->-] (a) to (1,0);
	\end{tikzpicture}$ or $\begin{tikzpicture}[baseline={([yshift=-.5ex]current bounding box.center)}]
	\node [vertex] (a) at (0,0) {$\mu$};
	\draw [->-] (a) to (1,0);
	\draw [->-] (-1,0) to (a);
	\end{tikzpicture}$ past some vertex $v$ that is not attached to it. We calculate that this has the effect of introducing a sign
	\[ \bar \mu \bar v = \bar v \]

	We are now ready to assemble all these calculations. Starting with the sum of ribbon quivers for $\Phi(\vec\Gamma)\circ b$, in sequence we `pass' the $\mu$ vertices past the flow vertices $v_i$, either by the necklace relation, if they are connected, or by switching the orientation as shown above, if they are not. The sign gained is always $\bar v$. If $v$ has $k$ outgoing edges and $\ell$ incoming edges this is
	\[ \bar v = \bar m_{(k)} \equiv dk - d - k \equiv \deg(v) + out(v) + in(v) \pmod{2} \]
	Therefore repeating this procedure for all the flow vertices we get a global sign
	\[ \sum_{v \in \mathrm{Flow}} (\deg(v) + out(v) + in(v)) \]
	but every edge appears twice in the sum above, with the exception of the edges connected to the $\times$-sources and $\circ$-sinks. As each $\circ$-sink has a degree of $(in(o)-1)$, this global sign is $\deg(\Gamma) + m - n$ and we have the desired identity.
\end{proof}

\subsection{The open-closed PROPs}\label{sec:openClosed}
We now describe a modification of the \textsc{prop}s $Q^d$, which will act not only on Hochschild chains of a pre-CY category $\cA$, which are the associated to closed strings, but also on the morphism spaces $\cA(X,Y)$ for objects $X,Y$ of $\cA$, which are associated to open strings. We already described the graphs appearing in this \textsc{prop} in \cref{sec:ribbonQuivers}; the open-in and open-out vertices are inputs and outputs of open strings.

\subsubsection{Colors and boundary type}
A colored \textsc{prop}s is just a particular type of symmetric monoidal category. Given a set $\cS$ of colors, a $\cS$-\textsc{prop} is a strict symmetric monoidal category whose monoid of objects is isomorphic to the free monoid generated by $\cS$.

That is, if $Q$ is a $\cS$-\textsc{prop}, for any two sequences $\vec c = (c_1, \dots, c_n)$ and $\vec c\ ' = (c_1' \dots, c_m')$, there is a set of morphisms $Q(\vec c,\vec c\ ')$. Just from the axioms of symmetric monoidal categories these spaces then come with appropriate actions by the symmetric groups $S_n,S_m$, compatible with permutations of $\vec c,\vec c\ '$.

Let $\cA$ be some $A_\infty$-category, and denote $\mathrm{Ob}(\cA)$ its set of objects. We now fix a set of colors associated to $\cA$ to be the set
\[ \cS_\cA = (\mathrm{Ob}(\cA) \times \mathrm{Ob}(\cA)) \sqcup \{*\} \]
In other words, there is one color for each ordered pair $(X,Y)$ of objects of $\cA$ and one extra color $*$. If $\cA$ is an algebra, that is, has a single object $X$, the set of colors is the two-element set $\{(X,X),*\}$ (for open and closed boundaries, respectively).

The closed \textsc{prop} $Q^d$ we described previously can be seen as colored \textsc{prop} where we only use the color $*$; in that case, the sequences $\vec c, \vec c\ '$ are described solely by the two positive integers $m,n$ which determine the space $Q^d(m,n)$, by specifying how many $\times$- and $\circ$-vertices were required.

For the open-closed \textsc{prop} we need to describe which ribbon quivers (also with open in/outputs) are compatible with a pair $\vec c, \vec c\ '$.
\begin{definition}
	The \emph{boundary type} of $\vec\Gamma$ is the tuple
	\[ (|V_\times|, |V_\circ|, (i,o,i,i,o,\dots,o), \dots, (o,i,o,\dots,i), F) \]
	of three integers, where $F$ is the number of boundary components without any marked vertices, and some finite number of cyclically ordered sequences on the symbols $i,o$, each corresponding to a boundary component with open-in and open-out vertices.
\end{definition}
Note that by definition, the sets $V_\times,V_\circ$ and all the $i$'s and $o$'s are also linearly ordered, corresponding to some order of open inputs and outputs; this ordering has nothing to do with the cyclic ordering in the boundary type.

\begin{definition}
	A marked acyclic ribbon quiver $\vec\Gamma$ is \emph{compatible} with the sequences of colors $\vec c,\vec c\ '$ if the boundary type of $\vec\Gamma$ satisfies these conditions:
	\begin{enumerate}
		\item $|V_\times| =$ number of $*$'s appearing in $\vec c$,
		\item $|V_\circ| =$ number of $*$'s appearing in $\vec c\ '$,
		\item $|\vec c| - |V_\times|$ = number of $i$'s in the boundary type,
		\item $|\vec c\ '| - |V_\circ|$ = number of $o$'s in the boundary type,
	\end{enumerate}
	together with the following condition on the colors. By (3) and (4) above we have a bijection between the incoming/outgoing open colors $c_k = (X_k,Y_k)$ (pairs of objects of $\cA$) and the sets of $i$ and $o$; we require that along every boundary component:
	\begin{enumerate}
		\item if $i_k$ appears immediately before $i_\ell$ then $Y_k = X_\ell$,
		\item if $i_k$ appears immediately before $o_\ell$ then $Y_k = Y_\ell$,
		\item if $o_k$ appears immediately before $i_\ell$ then $X_k = X_\ell$, and
		\item if $o_k$ appears immediately before $o_\ell$ then $X_k = Y_\ell$.
	\end{enumerate}
\end{definition}
Paraphrasing the conditions above in more informal terms, once we draw the surface associated to the marked acyclic quiver $\vec\Gamma$, we have exactly the right numbers of open/closed colors on each side, and along each boundary component made up of open colors, we can draw them as incoming/outgoing oriented strings such that their sources and targets are compatible.

\begin{figure}[h!]
	\begin{minipage}[c]{0.4\textwidth}
		\centering
		\includegraphics[width=3.5cm]{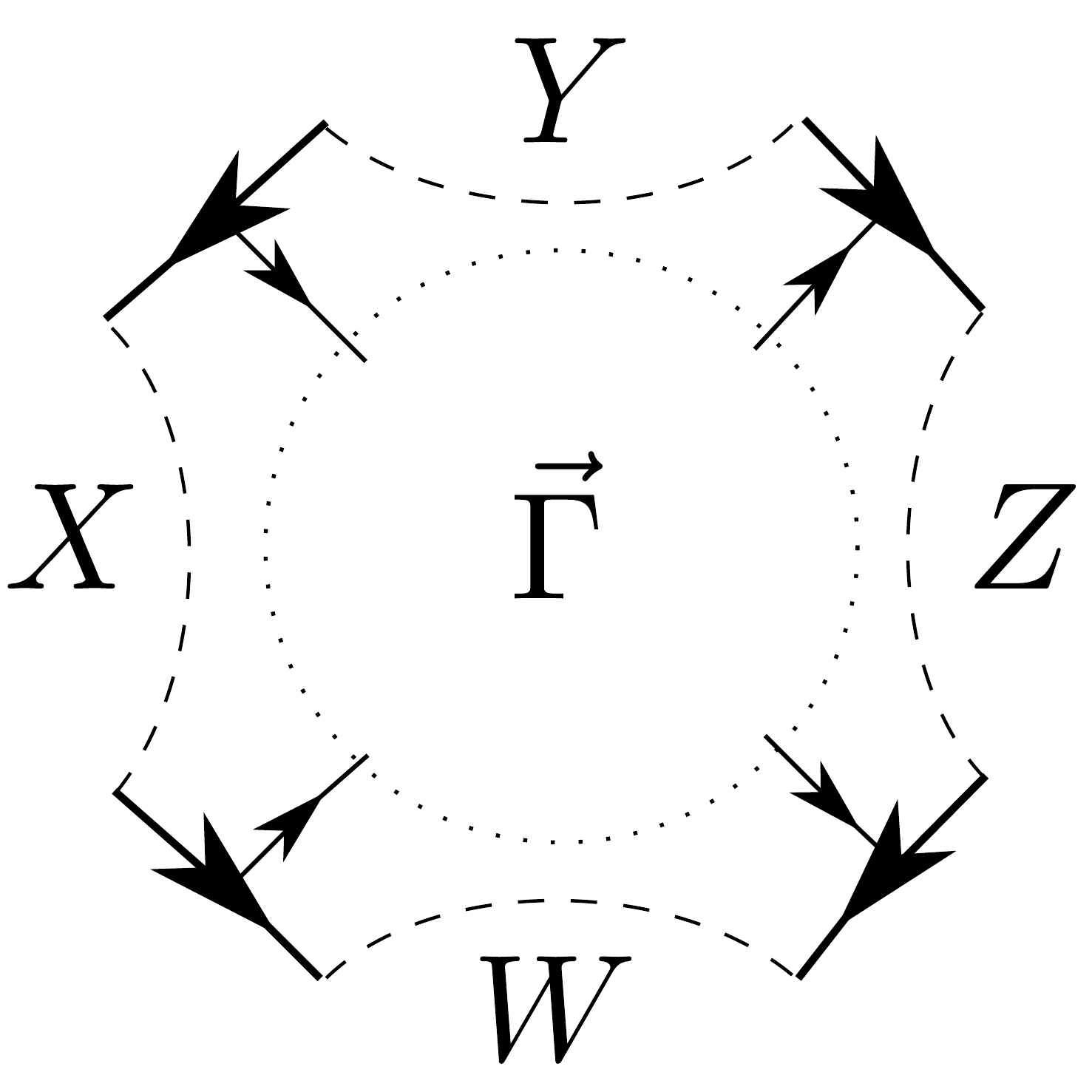}
	\end{minipage}\hfill
	\begin{minipage}[c]{0.6\textwidth}
	\caption{Boundary type along a boundary component of $\Gamma$ with two open inputs (on the left) and two open outputs (on the right). The associated boundary type has a cyclically ordered tuple $(i,i,o,o)$ which is compatible with any sequence of colors of the form $\left((Y,X),(X,W),(Z,W),(Y,Z)\right)$, for any four objects $X,Y,Z,W$ of $\cA$.}
\label{fig:boundaryType}
	\end{minipage}
\end{figure}

\subsubsection{Action of the open-closed PROPs}
We now use the compatibility condition above to define the desired \textsc{prop}s. For any pair of colors $\vec c, \vec c\ '$, consider the set $RQ(\vec c,\vec c\ ')$ of all marked ribbon quivers compatible with it.
\begin{definition}
	The space $Q^d(\vec c,\vec c\ ')$ is defined analogously to $Q^d(m,n)$ (\cref{sec:theClosedPROP}), but summing over all ribbon quivers in $RQ(\vec c,\vec c\ ')$, with orientations.
\end{definition}

As for the action of $Q^d(\vec c,\vec c\ ')$, we proceed the same way with the closed inputs ($\times$-vertices), and on each open input (valence one vertex in $V_\mathrm{open-in}$) labeled by a color $(X,Y)$, we input the corresponding element of the hom space $\cA(X,Y)$; on each open output labeled by a color $(X,Y)$ (valence one vertex in $V_\mathrm{open-out}$ we again read out the arrow traveling along the incident edge as an element of some hom space $\cA(X,Y)$ on components where the source and target $X,Y$ of that element agree with the desired color, or as zero if they do not.

The same argument as in \cref{thm:actionClosedPROP} proves the following result.
\begin{theorem}
	For any pre-CY category $\cA$ of dimension $d$, and any sequences of colors $\vec c, \vec c\ '$, having respectively $m, n$ instances of the color $*$, and open colors given by pairs of objects $(X_i,Y_i)$ and $(X'_i,Y'_i)$, there is an morphism of complexes
	\[ Q^d(\vec c, \vec c\ ') \otimes (C_*(\cA))^{\otimes m} \otimes \prod_i \cA(X_i,Y_i) \to (C_*(\cA))^{\otimes n} \otimes \prod_j \cA(X'_j,Y'_j)  \]
\end{theorem}

Finally, composition of ribbon quivers $\Gamma_2 \circ \Gamma_1$ along open in/outputs with compatible colors is done by erasing the pair of vertices from $V_\mathrm{open-out} \subset V(\Gamma_1)$ and $V_\mathrm{open-in} \subset V(\Gamma_2)$ and identifying their incident edges; this makes the collection of spaces $Q^d_{\vec c,\vec c\ '}$ into a $\cS_\cA$-colored dg \textsc{prop}.

In order to define composition and the action as above, one has to define the orientation on these ribbon quivers with open in/outputs, just as we have done before. This can be done by pretending that the open inputs are $\times$-vertices and the open outputs are $\circ$-vertices of valence one; and proceeding just as we did for the closed case.

\begin{remark}
	We would like to point out a new feature that occurs when composing open inputs and outputs: the creation of free boundary circles. Let $\Gamma_1, \Gamma_2$ be graphs which contain consecutive open outputs and inputs as follows:
	
	\begin{figure}[h!]
		\centering
		\includegraphics[width=0.6\textwidth]{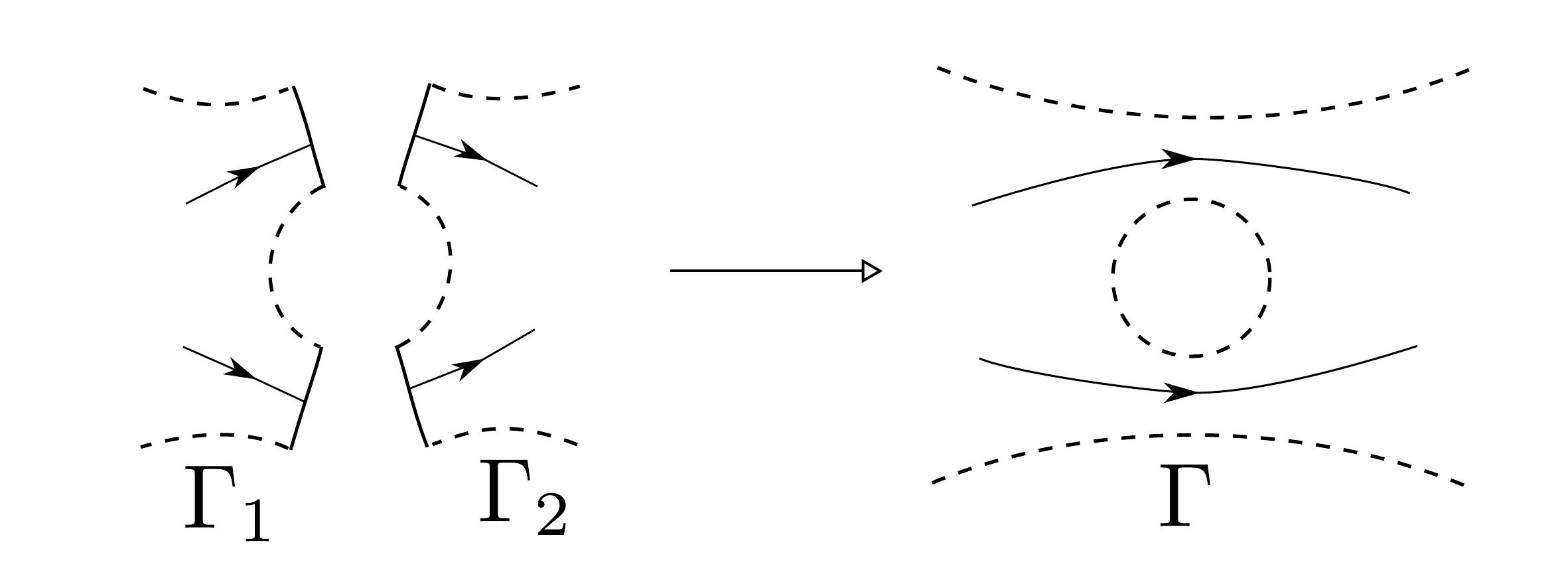}
		\caption{Creation of a new free boundary component (dashed circle from open gluing along two neighboring open intervals)}
	\end{figure}

	The composition $\Gamma_2 \circ \Gamma_1$ then will have a free boundary (i.e. without any marked sources or sinks), indicated by the dashed circle. On the other hand, all the free boundaries that $\Gamma_1$ and $\Gamma_2$ already had will still be in $\Gamma_2 \circ \Gamma_1$; the composition map $\circ$ is \emph{superadditive} on the number $F$ of free boundary circles. Because of this, in the open-closed case one cannot restrict to the case without no free boundary circles, and one does not get the isomorphisms between the even and odd dimensions that we discussed in the closed case in \cref{sec:isoEvenOdd}.
\end{remark}

\subsection{Features of the open-closed PROPs}
We discuss some features of the \textsc{prop}s $Q^d$, as well as some smaller algebraic structures that are part of it.

\subsubsection{Connes' differential and identity maps}
When marking ribbon quivers, we allowed some $\circ$-outputs to have a valence one vertex attached to them, labeled by \textbf{1}; this corresponds to inputting the cochain $1 \in C^0(A)$, for the case of an algebra, or the unit morphism $1_X \in \cA(X,X)$ for the appropriate object $X$, for the case of a category.

Recall that we have been denoting by $C_*(\cA)$ the `nonunital chain complex' $C^{nu}_*(\cA)$ for Hochschild homology, that is the quotient of the usual complex $C_*(\cA^+)$ of the augmented $A_\infty$-category by the subcomplex of chains that have some strict unit $1_X$ in some place with nonzero index.

Therefore, if a certain $\circ$-vertex $o_i$ has a \textbf{1} attached to it, unless the edge connecting $o_i$ and \textbf{1} is the distinguished edge of $o_i$, the resulting output chain is zero in $C_*(\cA)$, so we will always assume that edge is the distinguished edge.

With this convention, the ribbon quiver giving Connes' differential $B$ of cohomological degree $-1$ is given by
\[ \Gamma_B =
\begin{tikzpicture}[baseline={([yshift=-.5ex]current bounding box.center)}]
\node [inner sep=0pt] (top) at (-0.5,1) {$\times$};
\node [vertex] (side) at (0.5,0.5) {$1$};
\node [circ] (bot) at (0,0) {};
\draw [->-] (top) to (bot);
\draw [->] (side) to (bot);
\end{tikzpicture} \]

The ribbon quiver above has (homological) degree $\deg=1$ and genus zero, and is part of the closed \textsc{prop}. The identity map on $C_*(\cA)$ is also in that same space $Q^d(1,1)$, but with degree zero, and is given simply by
\[ \Gamma_{\id} =
\begin{tikzpicture}[baseline={([yshift=-.5ex]current bounding box.center)}]
\node [inner sep=0pt] (left) at (-1,0) {$\times$};
\node [circ] (right) at (1,0) {};
\draw [->-,shorten <=-3.5pt] (left) to (right);
\end{tikzpicture} \]

Another simple operation is given by the disc with $k$ open inputs and one closed output at the origin:
\[\begin{tikzpicture}[baseline={([yshift=-.5ex]current bounding box.center)}]
\draw (0,0) circle (1.5);
\node [circ] (o) at (0,0) {};
\draw [-w-] (1.6,0) to (o);
\draw [->-] (1.1,1.1) to (o);
\draw [->-] (0,1.6) to (o);
\draw [->-] (-1.1,1.1) to (o);
\node at (-1,0) {$\vdots$};
\draw [->-] (1.1,-1.1) to (o);
\draw [->-] (0,-1.6) to (o);
\draw [->-] (-1.1,-1.1) to (o);
\node at (1,0.4) {$X_1$};
\node at (1,-0.4) {$X_2$};
\node at (0.4,-1) {$X_3$};
\node at (-0.4,-1) {$X_4$};
\node at (-0.4,1) {$X_{k-1}$};
\node at (0.4,1) {$X_k$};
\end{tikzpicture}\]
which describes the map of degree $1-k$ that sends a sequence
\[ (a_0,a_1,\dots,a_{k-1}) \in \cA(X_1,X_2) \otimes \cA(X_2,X_3) \otimes \dots \otimes \cA(X_k,X_1) \]
to the Hochschild chain $a_0 \otimes a_1 \otimes \dots \otimes a_{k-1} \in C^*(\cA)$.

\subsubsection{The A-infinity operad}
Consider now the sub-\textsc{prop} of the open-closed \textsc{prop} $Q^d$ made up of genus zero ribbon quivers without $\times$- or $\circ$-vertices (therefore without \textbf{1}-vertices), without free boundary components and with a single open output; every such ribbon quiver is a union of some directed trees, with every vertex having a single output.

\begin{figure}[h!]
	\centering
	\includegraphics[width=0.3\textwidth]{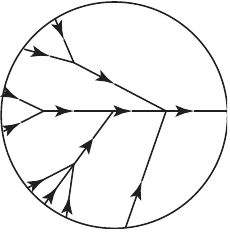}
\end{figure}

Note that the degree and differential on all these ribbon quivers is independent of the integer $d$. Moreover, if we restrict to connected ribbon quivers, only composition along one output is allowed, and we just have an operad. The following proposition follows from just checking that the signs in our prescription are just the Koszul signs appearing in the $A_\infty$-relations.
\begin{proposition}
	This operad of genus zero ribbon graphs with open inputs and one open outputs is equivalent to the $A_\infty$-operad, that is, chains of the operad of \emph{rooted planar trees}.
\end{proposition}

\subsubsection{The multicorolla dioperad}\label{sec:multicorolla}
We now consider a slightly bigger sub-\textsc{prop} of $Q^d$, where we still only have genus zero connected ribbon tree quivers with open in/outputs, without free boundaries, but now we allow multiple outputs.

If we still only allow composition along a single edge, the composed ribbon quiver still satisfies those conditions, and we get a dioperad (in the language of e.g. \cite{gan2003koszul}). This dioperad is generated by `multicorollas' such as:
\[\begin{tikzpicture}
\node [bullet] (phi) at (0,0) {};
\node (a11) at (-.7,-1) {};
\node (a12) at (-1.2,-0) {};
\node (a21) at (-0,1.3) {};
\node (a31) at (1.4,0.2) {};
\node (a32) at (1.3,-0.6) {};
\node (a33) at (0.7,-1) {};
\node (b1) at (0,-1.5) {};
\node (b2) at (-1.5,0.9)  {};
\node (b3) at (1.5,0.9) {};

\draw [->-] (phi) to (b1);
\draw [->-] (phi) to (b2);
\draw [->-] (phi) to (b3);
\draw [->-] (a11) to (phi);
\draw [->-] (a12) to (phi);
\draw [->-] (a21) to (phi);
\draw [->-] (a31) to (phi);
\draw [->-] (a32) to (phi);
\draw [->-] (a33) to (phi);
\end{tikzpicture}\]

Let us denote this multicorolla operad by $MC^d$. Each space of operations of $MC^d$ splits as a sum of complexes
\[ MC^d_{n_1,\dots,n_k} \]
from $n = \sum_i n_i$ inputs to $k$ outputs; this space is spanned by directed trees embedded in the disc with $k$ arrows going out, and $n_i$ arrows coming in between the $i$th and the $(i+1)$th outgoing arrows.

Later in \cref{sec:dimensionsAndOrientations} we will prove that the complexes $Q^d$ model chains on appropriate moduli spaces of metric ribbon quivers. In this case, each of those spaces is given by a quotient
\[ \mathrm{MetRT_{n_1,\dots,n_k}} / \ZZ_k, \]
where $\mathrm{MetRT_{n_1,\dots,n_k}}$ is a space of metric ribbon tree quivers with $k$ outputs, and $n = \sum_i n_i$ inputs (distributed as we described).

Note the quotient by the cyclic action; this is because in identifying the surfaces with $n$ inputs as belonging to one of the spaces $\mathrm{MetRT_{n_1,\dots,n_k}}$ we must be free to apply a cyclic rotation. The spaces $\mathrm{MetRT_{n_1,\dots,n_k}}$ are contractible and retract to the unique cell given by the ribbon tree with a single vertex.

Thus the spaces $\mathrm{MetRT_{n_1,\dots,n_k}} / \ZZ_k$ are rationally contractible; for $\kk$ of characteristic zero we have that
\[ MC^d_{n_1,\dots,n_k} \cong C_*(\mathrm{MetRT_{n_1,\dots,n_k}} / \ZZ_k, \cL^d), \]
where $\cL$ is a certain $\kk$-local system, and thus we have the following characterization:
\begin{proposition}
	$A$ has a pre-CY structure of dimension $d$ if and only if it is a module over the dioperad $MC^d$.
\end{proposition}

Similar spaces to the these already appeared in the work of Poirier-Tradler \cite{poirier2018combinatorics}, where the authors consider trees with a distinguished outgoing edge, meaning that they do not take the quotient of the chain complex by the cyclic action. Some related works also include \cite{poirier2019koszuality,drummondcole2015chainlevel}.

\section{Meromorphic Strebel differentials and the open-closed moduli space}\label{sec:Strebel}
In this section, we turn to the theory of Strebel differentials and explain how the open-closed \textsc{prop} $Q$ that we defined in the previous section relates to certain moduli spaces of surfaces with open/closed/free boundaries.

\subsection{Strebel differentials}
Let us briefly recall some relations between the geometry of quadratic differentials and the description of moduli spaces of Riemann surfaces.

Let us fix a compact and connected Riemann surface $S$. A meromorphic quadratic differential $\varphi$ on $S$ determines a flat metric $|\varphi|$ on the complement of its set of zeros and poles, and a measured foliation given by its \emph{horizontal foliation}.

The classical work of Jenkins and Strebel \cite{jenkins1957existence,strebel1984quadratic} deals with meromorphic quadratic differentials with poles of order at most two. Let us first discuss the holomorphic case. Such a differential $\varphi$ is a \emph{(holomorphic) Strebel differential} if the union of all non-closed leaves of its horizontal foliation has measure zero.

Such a differential determines a finite graph $\Gamma_\varphi$ embedded in $S$, consisting of the union of all the critical leaves, zeros and simple poles of $\varphi$, and decomposes $S \setminus (\Gamma \cup \{\text{double poles}\})$ into some number of maximal ring domains, or finite-height cylinders. Each such cylinder is foliated by the horizontal leaves of $\varphi$, which are simple closed curves of some isotopy class $\gamma_i$, all pairwise distinct and each not nullhomotopic.

\begin{theorem}\cite{strebel1984quadratic}
	Fix $S$ (of genus $\ge 2$), $n$ (isotopy classes) of simple closed curves $\gamma_i$ as above, and $n$ positive real numbers $m_i$; then there is a unique (up to scale) Jenkins-Strebel differential $\varphi$ whose ring domains are cylinders associated to $\gamma_i$ with modulus $m_i$.
\end{theorem}

Hubbard and Masur gave another perspective on the result above. Let $\cM\cF$ denote the space of measured foliations; a holomorphic quadratic differential gives such an object by taking its horizontal foliation.
\begin{theorem}\cite{hubbard1979quadratic}
	Any measured foliation $F \in \cM\cF(S)$ is realized by a unique holomorphic quadratic differential on $S$.
\end{theorem}

This gives a homeomorphism between the space of measured foliations and the space of holomorphic quadratic differentials on $S$, both homeomorphic to $\RR^{6g-6}$. In other words, the map $\cQ \to \cT_g$ presents the space of quadratic differentials as a fiber bundle over Teichm\"uller space, with fiber identified with $\cM\cF(\Sigma_g)$. Strebel's theorem for holomorphic differentials is then recovered by taking a particular measured foliation.

For differentials with double poles, the story is similar but the maximal ring domains surrounding each double pole is a \emph{infinite-height cylinder}. If we now set all the heights of the finite-height cylinders to be zero, we then have the following variant of Strebel's theorem.
\begin{theorem}\cite{strebel1984quadratic}
	For a fixed Riemann surface $C$ with $k$ distinct points $p_1,\dots,p_k$, and a choice of positive real numbers $\ell_1,\dots,\ell_k$, there is a unique Strebel differential on $C$ with double poles at $p_i$ and holomorphic on $C\setminus \{p_i\}$, such that all the maximal ring domains of $\varphi$ are half-infinite cylinders of circumference $\ell_i$ surrounding the points $p_i$.
\end{theorem}

For simplicity we refer to such objects as \emph{Strebel differentials}; in particular, for such a differential the residue of $\sqrt{\varphi}$ at every double pole is real. Each Strebel differential then determines a finite metric ribbon graph $\Gamma$ embedded in $S$, given by the critical leaves of $\varphi$, to which $C\setminus\{p_i\}$ contracts.

Strebel's uniqueness theorem can then be used to give an interpretation of moduli space of punctured curves by such graphs. The set $\cM^{comb}_{g,k}$ of all such metric ribbon graphs with genus $g$ and $k$ boundary cycles can be given a natural topology and orbifold structure.
\begin{theorem}\cite{kontsevich1992intersection}
	The map $\cM_{g,k} \times \RR_+^k \to \cM^{comb}_{g,k}$, given by taking the graph of critical leaves of Strebel differentials, is a homeomorphism of orbifolds.
\end{theorem}

\subsection{Higher-order poles}
Recent work of Gupta and Wolf \cite{gupta2016quadratic,gupta2019meromorphic} has described a generalization of the Hubbard-Masur theorem to meromorphic quadratic differentials with poles of arbitrary order, precisely describing the compatibility between geometric data on the surface (e.g. measured foliations) and the analytic behavior of $\varphi$.

As before, let $S$ be a Riemann surface with $k \ge 1$ points $p_1,\dots,p_k$ and choose $k$ positive integers $n_i \ge 2$. A quadratic differential $\varphi$ with poles of order $n_i$ at $p_i$ induces a measured foliation with pole singularities on $S$; this is a measured foliation on $S\setminus\{p_i\}$ but with some specific local behavior around each $p_i$. The space of such measured foliations is denoted by $\cM\cF(S,\{n_i\})$.

Around each pole, for some arbitrary choice of coordinate $z$, we have the local expressions
\[ \sqrt{\varphi} = \frac{1}{z^{n/2}}\left(p(z) + z^{n/2} g(z)\right)dz \]
for $n$ even, where $p(z)$ is a polynomial of degree $(n-2)/2$, and
\[ \sqrt{\varphi} = \frac{1}{z^{n/2}}\left(p(z) + z^{(n-1)/2} g(z)\right)dz \]
for $n$ odd, where $p(z)$ is a polynomial of degree $(n-3)/2$. In both formulas $g(z)$ is some non-vanishing holomorphic function.

The polynomials $p(z)$ are then the \emph{principal parts} of $\sqrt{\varphi}$; in the even $n$ case there is one real compatibility condition between $p$ and the measured foliation determined by $\varphi$. Taking into account this condition, one calculates that at the point $p_i$, the space of compatible principal parts is a manifold of real dimension $n-1$. When $n \ge 3$ this is homeomorphic to $\RR_+^{n-2} \times S^1$, and when $n = 2$ this is homeomorphic to $\RR_+$.

One can prove then a generalization of the Hubbard-Masur theorem, which we paraphrase from \cite{gupta2016quadratic,gupta2019meromorphic}.
\begin{theorem}\label{thm:guptawolf}
	With $S$ and $\{p_i\}, \{n_i\}$ as above, given a measured foliation with poles $F \in \cM\cF(S,\{n_i\})$ and the data of `compatible principal parts' (as defined in \emph{op.cit.}) at each $p_i$, there is a unique meromorphic quadratic differential $\varphi$ with poles of order $n_i$ at $p_i$ realizing $F$ and with the chosen principal parts, depending continuously on that data.
\end{theorem}
That is, up to the one real compatibility condition at poles of even order, one can pick the measured foliation and the principal parts independently. The continuity statement implies, in particular, that the natural map $\cQ(g, \{n_i\}) \to \cT_{g,k}$ presents the space of such meromorphic quadratic differentials as a fiber bundle with fiber homeomorphic to
\[ \hspace*{-0.7cm} \cM\cF(S,\{n_i\}) \times \{\text{compatible principal parts}\} \cong \RR^{6g-6 + \sum_i n_i} \times \left(\prod_{i, n_i \ge 3} \RR_+^{n_i-2} \times S^1 \right) \times \left(\prod_{i, n_i = 2} \RR_+\right) \]

For quadratic differentials with poles of orders $\le 2$, one recovers Strebel's theorem from this description: one chooses the zero measured foliation in $\cM\cF(S,\{n_i\})$ and gets a homeomorphism $\cQ(\{n_i\})^\mathrm{Str} \cong \cT_{g,k} \times \RR^k_+$. Here $\RR^k_+$ comes from the principal parts, which in this case are the residues of $\phi$ at the double poles. This isomorphism is moreover equivariant with respect to the action of the mapping class group $\Mod(\Sigma_{g,n})$, giving the isomorphism of orbifolds $\cM^\mathrm{Str} \cong \cM_{g,n} \times \RR^k_+$.

We now extend this to the case of higher order poles.
\begin{definition} (Meromorphic Strebel differentials)
	A meromorphic quadratic differential $\varphi$ is \emph{meromorphic Strebel} if it maps to the zero measured foliation.
\end{definition}

One can also characterize such differentials by the following properties:
\begin{enumerate}
	\item Every leaf asymptotic to a pole of $\varphi$ of order $\ge 2$ in one direction either
	\begin{itemize}
		\item Goes to a zero of $\varphi$ in the other direction, or
		\item Goes to the same pole in the other direction, and is homotopic to the constant curve at that pole, through a homotopy of curves that are also leaves of the horizontal foliation.
	\end{itemize}
	\item The closure $\Gamma$ of the union of all the other critical leaves (i.e. leaf going to a zero or simple pole, and not contained in the item above) of the horizontal foliation is measure zero, and
	\item The complement of the graph $\Gamma$ in $S$ is a disjoint union of some number of discs (with no cylinders).
\end{enumerate}

\begin{remark}
	In \cite{gupta2016quadratic} the authors refer to these differentials as `half-plane differentials', but do not include differentials with simple poles. Or rather, they get rid of simple poles by taking double covers ramified at them; for our purposes we cannot do that, and will eventually need to include simple poles. We have decided to call this notion `meromorphic Strebel' to emphasize this difference and the relation to Strebel's theorem.
\end{remark}

\begin{figure}[h!]
	\centering
	\begin{minipage}[c]{0.45\textwidth}
		\hspace{1cm}
		\includegraphics[width=\textwidth]{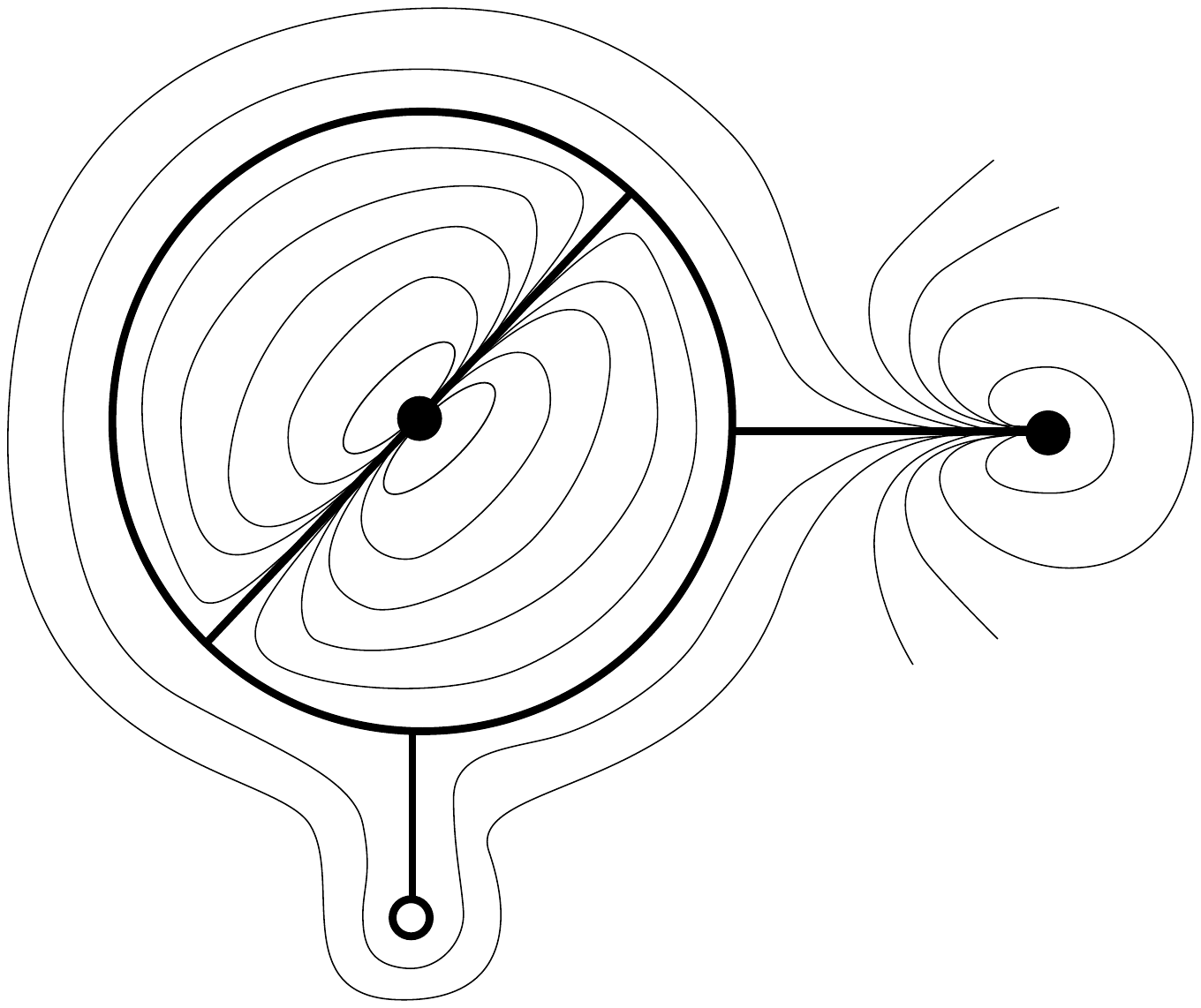}
	\end{minipage}\hfill
	\begin{minipage}[c]{0.5\textwidth}
		\caption{Picture of the horizontal foliation of a meromorphic Strebel differential on $\CC\PP^1$ with one pole of order 4 (inside the circle), one pole of order 3 (to the right) and a simple pole (white circle at the bottom).}
		\label{fig:meromorphic}
	\end{minipage}
\end{figure}

Consider the locus $\cQ(g,\{n_i\})^\mathrm{mStr}$ of Riemann surfaces of genus $g$ with meromorphic Strebel differentials with poles of orders $n_i$,  and no simple poles, together with the natural map $\pi: \cQ(g,\{n_i\})^\mathrm{mStr} \to \cT_{g,k}$. As a consequence of \cref{thm:guptawolf} we have
\begin{corollary}
	Each fiber of $\pi$ is identified with the space of compatible principal parts, homeomorphic to
	\[ \left(\prod_{i, n_i \ge 3} \RR_+^{n_i-2} \times S^1 \right) \times \left(\prod_{i, n_i = 2} \RR_+\right) \].
\end{corollary}

\subsection{A moduli space of meromorphic Strebel differentials}
We now construct a space that will give us a classifying space for certain open-closed cobordisms. Consider a compact topological surface with boundary $(\Sigma,\del\Sigma)$ of genus $g$. Let us choose a partition of its boundary $\del\Sigma$ into the following subsets:
\begin{itemize}
	\item Incoming closed boundaries $C_\mathrm{in}$ given by some disjoint union of circles.
	\item Incoming open boundaries $O_\mathrm{in}$ given by some disjoint union of (open) intervals.
	\item Outgoing closed boundaries $C_\mathrm{out}$ given by some disjoint union of circles.
	\item Outgoing open boundaries $O_\mathrm{out}$ given by some disjoint union of (open) intervals.
	\item Free boundaries, given by the complement of the above subsets, a disjoint union of circles and (closed) intervals.
\end{itemize}
We now pick a linear ordering of each of the set of connected components of the first four subsets above; i.e. we label all the incoming/outgoing open/closed boundaries by some ordered sets.

Consider the `bordered mapping class group' preserving all the open and closed boundaries \emph{pointwise}; in contrast, it can freely rotate and permute the free boundary circles. We will simply denote the corresponding mapping class group by $\Mod^\mathsc{oc}(\Sigma)$.

Let us now take a classifying space for this group, decomposed as
\[ \bigsqcup_{F \text{ free boundary circles}} B\Mod^\mathsc{oc}(\Sigma_F), \]
where we keep open and closed boundaries fixed, but include an arbitrary number $F$ of free boundary circles. This disjoint union is a classifying space for cobordisms between $O_\mathrm{in} \sqcup C_\mathrm{in}$ and $O_\mathrm{out} \sqcup C_\mathrm{out}$, with any number of holes in the interior. We can work with each one of those spaces separately by fixing the number of free boundary circles.

We now define some data from this surface.
\begin{definition}(Pole data)
	To each connected component $i$ of $\del \Sigma$, we assign an integer $n_i$ as follows:
	\begin{itemize}
		\item If $i \in \pi_0(C_\mathrm{in})$ (i.e. incoming closed), we assign $n_i = 3$, and if $i \in \pi_0(C_\mathrm{out})$ (i.e. outgoing closed) we assign $n_i = 2$.
		\item If $i$ contains exactly $N \ge 1$ open intervals (either incoming or outgoing) we assign $n_i = N+2$.
		\item Finally, if $i$ is a free boundary circle, we assign $n_i=2$.
	\end{itemize}
\end{definition}

Consider then the space $\cQ(g,\{n_i\})^\mathrm{mStr}$ of Riemann surfaces $S$ of genus $g$ with meromorphic Strebel differentials of pole orders $\{n_i\}$, and no simple poles. Consider one such differential. By definition, each pole $p$ of order 2 is surrounded by a ring domain that is a half-infinite cylinder, with boundary given by some circle $S^1_p$ which is the union of some critical leaves.
\begin{definition}\label{def:MarkingOutputs}
	An \emph{marking with $\ell$ outputs} on $\varphi \in \cQ(g,\{n_i\})^\mathrm{mStr}$ is a choice of some subset of the double poles of $\varphi$, of size $\ell$, together with a single point in $S^1_p$ for each $p$ in that subset. We denote the space of all such objects by $\cQ(g,\{n_i\})^\mathrm{mStr}_\ell$.
\end{definition}

For each double pole marked as an output, we take the unique vertical geodesic going from that pole to its marked point; this defines a tangent direction in $T_p S$. Moreover, for every pole $p_i$ of order $n_i \ge 3$, we have distinguished $n_i -2$ directions, given by the critical leaves asymptotic to $p_i$.

Therefore such a differential with marked outputs gives a point in the `bordered Teichm\"uller space' $\cT_{g, \vec{k}}$ of Riemann surfaces of genus $g$ with choices of some numbers of distinguished points on the boundaries. Here we use some generic subscript $\vec{k}$ to indicate all the data of the boundary; note that the free boundary circles (corresponding to unmarked double poles) do not have any distinguished points.

\begin{proposition}\label{prop:isoFiberBundles}
	There is an isomorphism of fiber bundles between
	\[ \pi: \cQ(g,\{n_i\})^\mathrm{mStr}_\ell \to \cT_{g,\vec{k}} \]
	and the trivial $\RR_+^k$-bundle over $\cT_{g,\vec{k}}$, which is equivariant for the action of the bordered mapping class group $\Mod^\mathsc{oc}(\Sigma)$.
\end{proposition}
\begin{proof}
	The existence of an isomorphism follows directly from Gupta-Wolf's result; the only non-trivial fact is that one can construct this isomorphism equivariantly, where the trivial $\RR_+^k$-bundle has a trivial action along the fiber. For this we must give a $\RR_+^k$-valued invariant function on $\cQ(g,\{n_i\})^\mathrm{mStr}_\ell$.

	For $\varphi \in \cQ(g,\{n_i\})^\mathrm{mStr}_\ell$, to each pole of order $n_i \ge 3$, resp. $=2$ there is a punctured disc around it given by the union of $n_i-2$ half-planes, resp. half-infinite cylinder, whose boundary is a union of critical leaves between zeros. The lengths of such boundaries give the desired $\Mod^\mathsc{oc}(\Sigma)$-invariant function. One must check that this indeed gives an isomorphism of fiber bundles; this follows from observing that, using the description of the spaces of compatible principal parts, such a function gives on each fiber a continuous embedding $\RR^k \hookrightarrow \{\text{compatible principal parts}\}$, moreover continuous with respect to variations of complex structure.
\end{proof}

\subsection{The perimeter-shrinking map}
Each meromorphic Strebel differential with marked outputs we used above gives a metric ribbon graph, which in turns determines the surface and differential; one could use it to give a cell model for the moduli space above. We will now explain how to relate this model to the marked ribbon quivers we discussed in \cref{sec:PROPs}.

For that, we will need an operation that shrinks the perimeter of the half-infinite cylinders corresponding to outputs. We start with a surface $S$ and a meromorphic Strebel differential $\varphi_0$ on $S$ having a double pole at a point $p$. The horizontal foliation and metric associated to $\varphi_0$ give a half-infinite cylinder surrounding $p$, with an $S^1$ boundary made of a sequence of horizontal geodesics between zeros of $\varphi$; this is a cycle in the associated metric ribbon graph $\Gamma$.

Let us say there are $n$ zeros on that cycle; pick one of these zeros as a starting point, and encode the data of $\Gamma$ near this cycle as a tuple of lengths
\[ (d_1,c_1, d_2,c_2, \dots, d_n,c_n) \in \RR^{2n}_{\ge 0} \]
where $d_i$ is associated to the $i$th zero in the cyclic order as follows: if the order of that zero is $\ge 2$ (i.e. the vertex in $\Gamma$ has valency $\ge 4$) then we assign $d_i=0$. If it is a simple zero, then $d_i$ is the length of the edge in $\Gamma$ incident there and \emph{not contained} in the cycle (i.e. the edge pointing out). We define $c_i$ to be the length of the edge between the $i$th and the $(i+1)$th zero in the cycle.

\begin{definition}
	A \emph{partial perimeter-shrinking family} of meromorphic Strebel differentials is a continuous family over $[0,T)$ of tuples $(S^t, \varphi^t,\gamma^t)$ of Riemann surfaces, meromorphic Strebel differentials, and cycles $c^t$ as above, such that their tuples of lengths around $\gamma^t$ satisfy
	\[ d_i^t = d_i^0 + t/2, \qquad c_i^t = c_i^0 - t \]
	while keeping all other lengths of the ribbon graph constant.
\end{definition}
In other words, as we increase $t$, the sides of the cycle `zip up' by a distance of $t/2$ on each side, gluing more of the cells neighboring the cycle and reducing the circumference of the cycle by $Nt$.

It is obvious from continuity that if none of the $c_i^0$ is equal to $T$, the family extends to $t=T$. But also if, say, $c_1 = T$, we can extend the family with a differential over $t=T$ that has $(n-1)$ zeros on the cycle; as long as at least one of the $c_i$ is bigger than $T$, this still gives a continuous family of quadratic differentials over $[0,T]$.

We can iterate this process, shrinking the perimeter while decreasing $n$ accordingly, until we end up with some differential where all the $c_i$ are equal to some $T$. Then we can complete this family by a meromorphic Strebel differential with \emph{one less double pole}. This gives a continuous family valued in the locus of quadratic differentials with a bounded above number of higher-order poles.

\begin{definition}
	A \emph{perimeter-shrinking family} of meromorphic Strebel differentials is a sequence of completed partial perimeter-shrinking families $(S_i^t,\varphi^t_i,\gamma_i^t), 1 \le i \le M$, each over some interval $[0,T_i]$, together with isomorphisms
	\[ (S_i^{T_i},\varphi_i^{T_i},\gamma_i^{T_i}) \cong (S_{i+1}^{0},\varphi_{i+1}^{0},\gamma_{i+1}^{0})\]
	and such that in the last family, $\varphi^{t=T_M}_M$ has one less double pole than $\varphi^{t=0}_1$.
\end{definition}

\begin{lemma}
	Such a perimeter-shrinking family from $\varphi$ to $\varphi'$ gives a map of metric ribbon graphs $\Gamma \to \Gamma'$ between the corresponding metric ribbon graphs.
\end{lemma}
\begin{proof}
	We can construct this map by hand on each partial family, by defining it on the edges with lengths $c_i$ and $d_i$ by the zipping description. One can easily check that this map extends to a the endpoint of such a family. This uniquely defines the map, since the other edges of $\Gamma$ are kept constant.
\end{proof}

Consider now starting from a surface and a meromorphic Strebel differential with marked outputs $(S,\varphi) \in \cQ(g,\{n_i\})^\mathrm{mStr}_\ell$ as in \cref{def:MarkingOutputs}. Let $p$ be a marked output (i.e., double pole of $\varphi$) defining a cycle $\gamma$ in the corresponding ribbon graph, and $q \in \gamma$ the point giving its marked direction.
\begin{lemma}
	There is a unique perimeter-shrinking family (up to isomorphism) from $(S,\varphi,\gamma)$ such that the end differential $\varphi'$ has one fewer double pole. Moreover, this family gives a distinguished point $p'$ in $\Gamma$, coming from the vanishing double pole, and, coming from the marked direction $q$, either a marked edge of $\Gamma$ incident at $p'$, or a marked angle at $p'$.
\end{lemma}
\begin{proof}
	Starting from $(S,\varphi,\gamma)$, the existence and uniqueness of this perimeter-shrinking family follows from \cref{thm:guptawolf} and \cref{prop:isoFiberBundles}, which identifies the space of such marked differentials with an appropriate space of metric ribbon graphs; we can then just specify the lengths of this graph to satisfy our condition.
	
	A perimeter-shrinking family is made of a sequence of partial perimeter-shrinking families, so let us look at the last such family $\varphi^t_M$, which has $\varphi^{T_M}_M = \varphi'$. At the start of such a family we have a differential $\varphi'' = \varphi^0_M$ with some number $m$ of zeros around the relevant double pole, such that all the lengths $c_i$ between them are equal to $T_M$.

	We take the image $q''$ of the point $q$ in $\Gamma''$. Let us take the distance in $\Gamma''$ between $q''$ and each of the $m$ zeros around that pole. By construction, there are only  two possibilities:
	\begin{enumerate}
		\item There is a unique minimum among those distances, for some zero labeled by $j$, or
		\item There are two consecutive zeros, labeled by $j$ and $j+1$, which are equidistant to $q''$.
	\end{enumerate}
	Now we produce the markings. The point $p'$ is defined to be the image of the circle in the graph $\Gamma'$; this is either a simple pole (if $m=1$), a regular point (if $m=2$) or a zero of order $m-2$ (if $m \ge 3$). As for the direction, in case 1 above we pick the edge incident at $p'$ corresponding to the $j$th zero, and in case 2 we pick the angle between the edges corresponding to $j$ and $j+1$.
\end{proof}

\begin{figure}[h!]
	\centering
	\hspace*{-1cm}
	\includegraphics[width=1.2\textwidth]{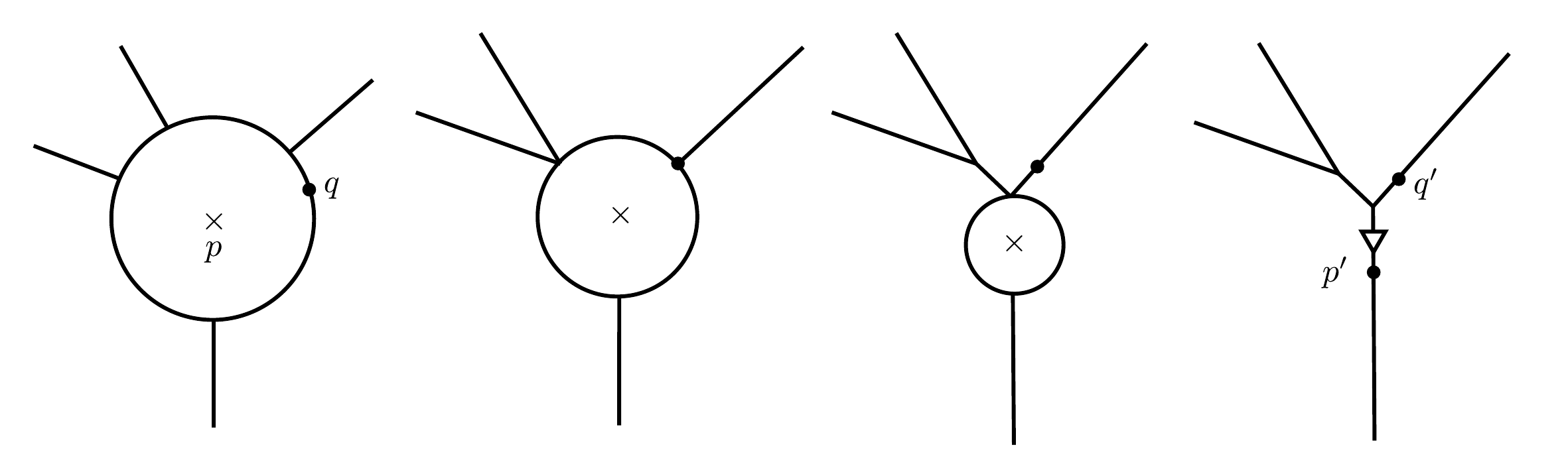}
	\caption{A perimeter-shrinking family for the double pole $p$ with a generic position of its marking $q$. The end result is a differential with one less double pole; the image of $p$ is $p'$ and its distinguished edge is the top one, as it leads to the image of $q$, denoted $q'$.}
\end{figure}

\begin{figure}[h!]
	\centering
	\hspace*{-1cm}
	\includegraphics[width=1.2\textwidth]{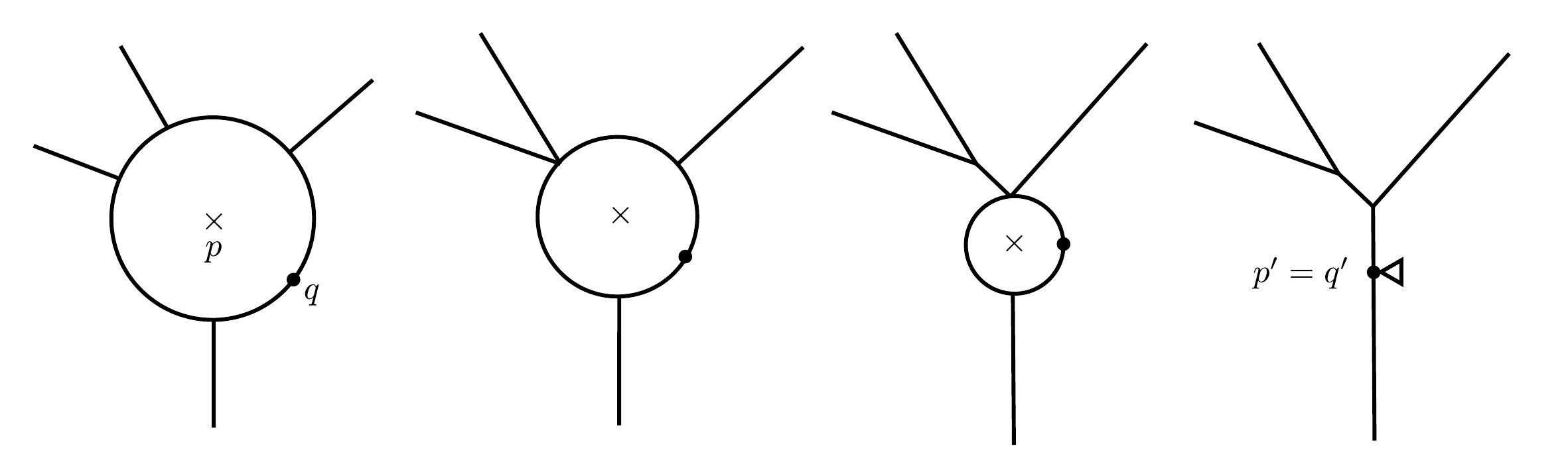}
	\caption{A perimeter-shrinking family for the double pole $p$ with a non-generic position of its marking $q$. If the images $p',q'$ coincide, we then mark the angle inside of which the image of $q$ is before the last partial perimeter shrinking.}
\end{figure}

Consider now such a differential with marked outputs $\varphi \in \cQ(g,\{n_i\})^\mathrm{mStr}_\ell$ as in \cref{def:MarkingOutputs}. Let us pick any linear ordering of the $\ell$ marked outputs. Doing the procedure above for each of the double poles in order results in a sequence $(S^i, \varphi^i, \{\text{markings}\}^i)$ where $S^0 = S$, $\varphi^{i+1}$ has one less double pole than $\varphi^i$ and there are $i$ markings (i.e. points in the ribbon graph with a distinguished incident direction or angle).

After contracting the perimeters of all the half-infinite cylinders corresponding to the $\ell$ marked double poles, we get a metric ribbon graph $\Gamma$, with $\ell$ marked points, each with a marked direction (either incident half-edge or angle between half-edges). All of the internal edges of $\Gamma$ have finite length, but some of its leaves (corresponding to the higher-order poles) have infinite length. This data satisfies the conditions
\begin{enumerate}
	\item The marked points $p_i$ are all distinct.
	\item Every leaf of $\Gamma$ with finite length ends in a marked point (this comes from a simple pole after contracting perimeters of half-infinite cylinders)
\end{enumerate}

Let $\Gamma$ be such a metric ribbon graph with $\ell$ marked points $p_i$, some number $e$ of marked incident half-edges and $(\ell-e)$ marked angles. Let us denote the set of such data as $\mathrm{MetGr}_\ell$. We now make a larger set $\mathrm{MetRG}'_\ell$ by replacing every element in $\mathrm{MetRG}_\ell$ by the set
$\RR^e$.

This replacement accounts for the fact that from the perimeter-shrinking map, for each time we ended up with a marked half-edge, there was a whole angular sector on which our marked point around the double pole could have been, whereas for each marked angle there was a single such point.

We now partition the set $\mathrm{MetRG}'$ by genus and number of infinite-length leaves on each boundary component of the corresponding surface. Let $(m_i)$, $1 \le i \le N$ denote that tuple of numbers, with $m_i \ge 0$. We then have a partition
\[ \mathrm{MetRG}'_\ell = \bigsqcup_{g, (m_i)} \mathrm{MetRG}'_{\ell,g,(m_i)} \]

We then make another tuple $(n_i)$, of length $N + \ell$, by setting
\[ n_i = m_i +2,\ 0 \le i \le N, \text{\ and\ } n_i = 2,\ N+1 \le i \le N+\ell \]
The tuple $(n_i)$ is the orders of poles \emph{before} perimeter-shrinking.

\begin{corollary}\label{cor:perimeterShrinking}
	Perimeter-shrinking gives a bijection of sets
	\[ \cQ(g,(n_i))^\mathrm{mStr}_{\ell} \xrightarrow{\sim} \mathrm{MetRG}'_{\ell,g,(m_i)} \times \RR^{\ell}\]
	Therefore we can endow $\mathrm{MetRG}'_{\ell,g,(m_i)}$ with the action of the corresponding mapping class group $\Mod^\mathsc{oc}_{g,(m_i)}$, acting trivially on the $\RR^\ell$ component, and with the finest topology for which the map above is a homeomorphism, so that the quotient
	\[ \cM_{\ell,g,(m_i)} := \left[ \mathrm{MetRG}'_{\ell,g,(m_i)}/\Mod^\mathsc{oc}_{g,(m_i)} \right] \]
	is an orbifold classifying space for the open-closed mapping class group.
\end{corollary}

\subsection{Cell decomposition by marked ribbon quivers}
We now use the model discussed in the previous Subsection to produce a cell decomposition of the open-closed moduli space, and relate it to the dg \textsc{prop} constructed in \cref{sec:PROPs}.

Let us paraphrase the result of the previous section. We choose a genus $g$ and a boundary type of $\Sigma$: each boundary component is either incoming closed, outgoing closed, a free circle, or a combination of incoming open intervals, outgoing open intervals and free intervals.

The number of outgoing closed boundaries is the integer $\ell \ge 1$ and the rest of the boundaries determines the tuple $\{m_i\}$. Following the perimeter-shrinking map in the previous subsection gives a space $\cM'_{\ell,g,(m_i)}$: each point in this space is given by  $(S, \varphi, \{\vec v_i\}_{1 \le i \le \ell})$ where $S$ is a Riemann surface diffeomorphic to a compactification of $\Sigma$, $\varphi$ is a meromorphic Strebel differential with critical graph $\Gamma$, and $\vec v_i$ are unit length tangent vectors, projecting down to points $p_i \in \Gamma$.

Each such vector $\vec v_i$ corresponds to one outgoing closed boundary, and encodes a distinguished direction; the quadratic differential $\varphi$ has poles of order $\ge 2$ corresponding to the other boundary components of $\Sigma$. Namely, it has poles of:
\begin{itemize}
	\item order 2 for each free boundary circle,
	\item order 3 for each incoming closed circle,
	\item order $n+2$ for each boundary component with $n$ open intervals (incoming and outgoing)
\end{itemize}
Each pole of order 3 has a single edge of $\Gamma$ asymptotic to it, and each pole of order $n+2$ has $n$ such edges. The data of such a differential is encoded by the metric graph $\Gamma$. We now use the marking data $\{\vec v_i\}$ to give $\Gamma$ the structure of a \emph{marked ribbon quiver} as defined in \cref{sec:ribbonQuivers}.

For that, we need to make a small modification to $\Gamma$ to properly deal with the open outputs. Recall that every open in/output corresponds to an infinite-length edge of $\Gamma$, asymptotic to a higher-order pole.

Let $|V_\mathrm{open-out}|$ be the number of open outputs. We now pick a tuple of positive reals $(\lambda_1,\lambda_2,\dots) \in (\RR_{>0})^{|V_\mathrm{open-out}|}$ and regularize $\Gamma$ by setting the length of the corresponding edge to be given by these numbers $\lambda_i$.

We now use the marking to define a height function: consider the subset $P \subset \Gamma$ consisting of all the $\ell$ points $p_i$ corresponding to the markings $\vec v_i$, together with all the end points of the regularized edges. That is, $P$ has one point for each output, open or closed.
\begin{definition}
	The height function $h$ on the metric graph $\Gamma$ is the distance along $\Gamma$ to the set $P$. We define the ribbon quiver $\vec\Gamma$ to be the ribbon graph $\Gamma$ directed with the \emph{negative gradient} of this height function.
\end{definition}
The regularization procedure guarantees that the outgoing open edges always have directions pointing into their corresponding sink.

We then use the rest of the data to produce a marking on the ribbon quiver, in the sense of \cref{sec:ribbonQuivers}, as follows
\begin{itemize}
	\item $V_\times$ are the valence one sources (corresponding to poles of order 3) associated to incoming closed boundaries,
	\item $V_\mathrm{open-in}, V_\mathrm{open-out}$ are the valence one sources, resp. sinks associated to the other higher order poles,
	\item $V_\circ$ is the set of marked points $\{p_i\}$ in $\Gamma$,
	\item for each point in $V_\circ$, if after perimeter-shrinking we ended up with a marked angle, we attach an extra leaf to that angle
	\[\begin{tikzpicture}[baseline={([yshift=-.5ex]current bounding box.center)}]
	\node [vertex] (a) at (0,0) {$1$};
	\draw [->-] (a) to (1,0);
	\end{tikzpicture}\]
\end{itemize}

The procedure above uniquely defines a marked ribbon quiver, and it stratifies the space
\[ \cM'_{\ell,g,(m_i)} \times (\RR_{>0})^{|V_\mathrm{open-out}|}  \]
by the type of marked ribbon quiver obtained; this gives a cell decomposition of this space.

\begin{remark}
	Note that the quiver obtained generically will have more vertices than the ribbon graph; there will be some number of valence two sources, one for each maximum of the height function occurring inside an edge of the graph.
\end{remark}

\subsubsection{Dimensions and orientations}\label{sec:dimensionsAndOrientations}
We now discuss the dimensions and orientations of the cells labeled by marked ribbon quivers.

The dimension of the Teichm\"uller space of a genus $g$ surface with $k$ punctures is given by $6g-6+2k$. We considered instead the moduli space $\cT_{g, \vec{k}}$ where each puncture is decorated by a number of tangent directions depending on the boundary type; in terms of those boundary types, we have that
\[ \dim \cT_{g, \vec{k}} = 6g -6 +2F +3 |V_\times| + 2O + |V_\mathrm{open-in}| + |V_\mathrm{open-out}| + 3|V_\circ| -1 \]
where $F$ is the number of free boundaries (i.e. boundary components of $\Gamma$ without marked sources or sinks) and $O$ is the number of boundary components with open-in and open-out intervals. We get this formula from the dimension of the punctured Teichm\"uller space since every $\times$-source, open-in, open-out, $\circ$-sink contributes one extra dimension (by picking a direction at the corresponding puncture), minus one to account for overall rotation of the directions.

We then had the space $\cQ(g,(n_i))^\mathrm{mStr}_{|V_\circ|}$ of meromorphic Strebel differentials, before shrinking perimeters; this was a $(\RR_{>0})^{k}$ bundle over $\cT_{g, \vec{k}}$ so its dimension is 
\[ \dim \cQ(g,(n_i))^\mathrm{mStr}_{|V_\circ|} = 6g -7 +3F +4 |V_\times| + 3O + |V_\mathrm{open-in}| + |V_\mathrm{open-out}| + 4|V_\circ| \]
and after perimeter-shrinking we have the space $\mathrm{MetRG}'_{\ell,g,(m_i)}$ and its quotient, our desired moduli space $\cM = \cM_{g,(m_i),|V_\times|}$, which by \cref{cor:perimeterShrinking} has dimension
\[ \dim \cM = 6g -7 +3F +4 |V_\times| + 3O + |V_\mathrm{open-in}| + |V_\mathrm{open-out}| + 3|V_\circ| \]
If we include the data of the lengths of the edges connected to the open outputs, we then have
\[ \dim \cM \times (\RR_{>0})^{|V_\mathrm{open-out}|} = 6g -7 +3F +4 |V_\times| + 3O + |V_\mathrm{open-in}| + 2|V_\mathrm{open-out}| + 3|V_\circ| \]

Recall now the definition of the $d$-degree of a marked ribbon quiver, from \cref{def:degree}:
\[ \deg_d(\vec\Gamma) = \sum_{v \in \mathrm{Source}^{\ge 2}} ((2-d)out(v)+d-4) + \sum_{v \in \mathrm{Flow}} ((2-d) out(v) + d + in(v) -4) + \sum_{v \in V_\circ} (in(v)-1). \]
Calculating this for $d=0$ gives
\[ \deg_0(\vec\Gamma) = \sum_{v \in \mathrm{Source}^{\ge 2}} (2 out(v)-4) + \sum_{v \in \mathrm{Flow}} (2 out(v) + in(v) -4) + \sum_{v \in V_\circ} (in(v)-1), \]
so $\vec\Gamma$ has $\deg_0(\vec\Gamma) = 0$ when it has only valence 2 unmarked sources, valence 3 flow vertices and valence 1 $\circ$-vertices. Such marked ribbon quivers are exactly the ones labeling a top-dimensional cell in the decomposition of $\cM \times (\RR_{>0})^{|V_\mathrm{open-out}|}$ that we described; other marked ribbon quivers label cells with (real) codimension given by $\deg_0$.

\begin{lemma}\label{lem:dimension}
	The (real) dimension of a cell in $\cM \times (\RR_{>0})^{|V_\mathrm{open-out}|}$ labeled by a marked ribbon quiver $\vec\Gamma$ is given by
	\[ \dim_\RR C_{\vec\Gamma} = |V(\Gamma)| - |V_\times| - |V_\mathrm{open-in}| - |V_\mathrm{open-out}| - |V_\mathbf{1}| - 1. \]	
\end{lemma}
\begin{proof}
	Recall that each point in this cell can be given by a marked ribbon quiver with a metric; however, we are not free to choose the lengths of all the edges independently, since the lengths all come from differences of values of the height function.
	
	Given any two vertices, the length of any directed path between them is the same, given by the difference of their heights no matter the path. Conversely, fixing all height differences fixes all the edge lengths. Thus the space of possible edge lengths has dimension given by the number of vertices at finite height minus one. This excludes the vertices in $V_\times,V_\mathrm{open-in}$ and $V_\mathrm{open-out}$; the first two groups are at $+\infty$ height, and we set the open-out vertices to be at zero height.
	
	The term $-|V_\mathbf{1}|$ comes from the fact that for every $\circ$-sink with distinguished incident edge we added a factor of $\RR$, and for every $\circ$-sink with distinguished angle we replaced that angle with a $\mathbf{1}$-vertex. Counting all these factors we get the formula above.
\end{proof}

\begin{example}
	It is easy to see that the formulas above give the correct dimension in the usual case of Stasheff associahedra, which correspond to the disc with one open output, that is, $g=F=|V_\times|=|V_\circ|=|V_\mathbf{1}|=0$, with one open boundary, $O=1$, and one output, $|V_\mathrm{open-out}|=1$. The formulas above then calculate
	\[ \dim_\RR \cM \times \RR^{|V_\mathrm{open-out}|} = |V_\mathrm{open-in}|-2.\]
	This is equal to the dimension of the top cells, parametrized by metric binary trees with $|V_\mathrm{open-in}|$ inputs, and also equal to minus the cohomological degree of the corresponding $A_\infty$ operation $A^{\otimes|V_\mathrm{open-in}|} \to A$.
\end{example}

We are now ready to relate this space to the \textsc{prop}s $Q^d$ constructed in \cref{sec:ribbonQuivers}. Note that here it is important that we assume $\QQ \subseteq \kk$.
\begin{theorem}
	The complexes $Q^{d=0}_g$ with their differential $\del$ calculate the homology $H_{-*}(\cM,\kk)$ of the corresponding classifying spaces $\cM$ for the open-closed mapping class groups.
\end{theorem}
\begin{proof}
	We basically follow the argument for the analogous statement about the usual graph complex, see \cite{penner1986moduli,conant2003theorem,kontsevich1993formal}. With the degree conventions we set, it will be easier to work cohomology so we use Poincar\'e duality; since the spaces $\cM$ are not compact we must use an appropriate form of Poincar\'e duality by working in cohomology with compact supports $H^*_c(\cM,\kk)$ instead.
	
	We now construct a compactification $\overline\cM$ of $\cM$, by using the cell decomposition by \emph{ribbon graphs}, and adding cells where the length of any number of edges of finite length goes to zero or $\infty$. Each cell of $\overline\cM$ is labeled by a ribbon graph with edge lengths $\in [0,+\infty]$; the (homological) boundary differential is still given by contracting some number of the finite lengths. We note now that the boundary $\del\overline\cM$ is a subcomplex of the cell complex for $\overline\cM$; one cannot get rid of an infinite-length edge by contracting a finite-length edge, and if in $\Gamma$ a cycle has length zero, it also has length zero in any edge contraction of $\Gamma$. We thus see that the cell complex of ribbon graphs we constructed for $\cM$ computes the compactly supported rational cohomology $H^*_c(\cM,\QQ)$.
	
	To check orientations, it will be more straightforward to deal with vertex orderings instead of edge orderings; the correspondence between those is given in \cref{sec:isoEvenOdd}. The terms of the differential not involving $\mathbf{1}$-vertices are given by dual of the (homological) differential that sums over contractions of edges; such contractions correspond to making the difference in height $h(a) - h(b)$ between the source and target vertices go to zero. When this happens we delete one of the vertices (say $a$); recall that in the definition of $\del$ the induced orientation was given by putting $a$ at the start, so $\del$ is indeed the dual of the cell boundary differential
	
	As for the terms involving $V_\mathbf{1}$, these can be understood by a local calculation as the $\circ$-vertices they are attached to are not allowed to collide; we check that the differential $\del$ agrees with the required orientations.
\end{proof}

\begin{example}
	We take the surface of genus zero with two closed inputs and one closed output; each input and output is a puncture with a distinguished tangent direction. The corresponding Teichm\"uller space $\cT_{0,\vec 3}$ is homeomorphic to $(S^1)^3$. Therefore the space $\cQ(0,\{3,3,2\})^\mathrm{mStr}_{\ell = 1}$ of meromorphic Strebel differentials has (real) dimension 6, and the modified space of metric graphs $\mathrm{MetRG}'$ has dimension 5.
	
	Therefore our moduli space has $\dim_\RR \cM = 5$, and is homeomorphic to $(S^1)^3 \times (\RR_{\ge 0})^2$, retracting to the 3-torus. The following 1-cycles (given by sums of diagrams with some appropriate orientation) map to three 1-cycles spanning $H_1(\cM,\kk) \cong \kk^3$:
	\begin{align*} 
	C_1 &= \begin{tikzpicture}[auto,baseline={([yshift=-.5ex]current bounding box.center)}]
	\node [inner sep=0pt] (v1) at (0,1.5) {$\times$};
	\node [bullet] (v3) at (0,0.75) {};
	\node [inner sep=0pt] (v2) at (0,0) {$\times$};
	\node [bullet] (v4) at (0.75,0) {};
	\node [bullet] (v5) at (0,-0.75) {};
	\node [circ] (v6) at (0,-1.5) {};
	\draw [->-,shorten <=-3.5pt] (v1) to (v3);
	\draw [->-,shorten <=-3.5pt] (v2) to (v4);
	\draw [->-] (0,0.75) arc (90:0:0.75);
	\draw [->-] (0,0.75) arc (90:270:0.75);
	\draw [->-] (0.75,0) arc (0:-90:0.75);
	\draw [->-] (v5) to (v6);
	\end{tikzpicture} \quad + \quad
	\begin{tikzpicture}[auto,baseline={([yshift=-.5ex]current bounding box.center)}]
	\node [inner sep=0pt] (v1) at (1.5,0) {$\times$};
	\node [bullet] (v3) at (0,0.75) {};
	\node [inner sep=0pt] (v2) at (0,0) {$\times$};
	\node [bullet] (v4) at (0.75,0) {};
	\node [bullet] (v5) at (0,-0.75) {};
	\node [circ] (v6) at (0,-1.5) {};
	\draw [->-,shorten <=-3.5pt] (v1) to (v4);
	\draw [->-,shorten <=-3.5pt] (v2) to (v4);
	\draw [->-] (0,0.75) arc (90:0:0.75);
	\draw [->-] (0,0.75) arc (90:270:0.75);
	\draw [->-] (0.75,0) arc (0:-90:0.75);
	\draw [->-] (v5) to (v6);
	\end{tikzpicture} +
	\begin{tikzpicture}[auto,baseline={([yshift=-.5ex]current bounding box.center)}]
	\node [inner sep=0pt] (v1) at (0.75,-1.5) {$\times$};
	\node [bullet] (v3) at (0,0.75) {};
	\node [inner sep=0pt] (v2) at (0,0) {$\times$};
	\node [bullet] (v4) at (0.75,0) {};
	\node [bullet] (v5) at (0,-0.75) {};
	\node [circ] (v6) at (0,-1.5) {};
	\draw [->-,shorten <=-3.5pt] (v1) to (v5);
	\draw [->-,shorten <=-3.5pt] (v2) to (v4);
	\draw [->-] (0,0.75) arc (90:0:0.75);
	\draw [->-] (0,0.75) arc (90:270:0.75);
	\draw [->-] (0.75,0) arc (0:-90:0.75);
	\draw [->-] (v5) to (v6);
	\end{tikzpicture} +
	\begin{tikzpicture}[auto,baseline={([yshift=-.5ex]current bounding box.center)}]
	\node [inner sep=0pt] (v1) at (0.75,-1.5) {$\times$};
	\node [bullet] (v3) at (0,0.75) {};
	\node [inner sep=0pt] (v2) at (0,0) {$\times$};
	\node [bullet] (v4) at (0.75,0) {};
	\node [bullet] (v5) at (0,-0.75) {};
	\node [circ] (v6) at (0,-1.5) {};
	\draw [->-,shorten <=-3.5pt] (v1) to (v6);
	\draw [->-,shorten <=-3.5pt] (v2) to (v4);
	\draw [->-] (0,0.75) arc (90:0:0.75);
	\draw [->-] (0,0.75) arc (90:270:0.75);
	\draw [->-] (0.75,0) arc (0:-90:0.75);
	\draw [-w-] (v5) to (v6);
	\end{tikzpicture} +
	\begin{tikzpicture}[auto,baseline={([yshift=-.5ex]current bounding box.center)}]
	\node [inner sep=0pt] (v1) at (-0.75,-1.5) {$\times$};
	\node [bullet] (v3) at (0,0.75) {};
	\node [inner sep=0pt] (v2) at (0,0) {$\times$};
	\node [bullet] (v4) at (0.75,0) {};
	\node [bullet] (v5) at (0,-0.75) {};
	\node [circ] (v6) at (0,-1.5) {};
	\draw [->-,shorten <=-3.5pt] (v1) to (v5);
	\draw [->-,shorten <=-3.5pt] (v2) to (v4);
	\draw [->-] (0,0.75) arc (90:0:0.75);
	\draw [->-] (0,0.75) arc (90:270:0.75);
	\draw [->-] (0.75,0) arc (0:-90:0.75);
	\draw [->-] (v5) to (v6);
	\end{tikzpicture} \\
	C_2 & =
	\begin{tikzpicture}[auto,baseline={([yshift=-.5ex]current bounding box.center)}]
	\node [inner sep=0pt] (v1) at (-1.5,0) {$\times$};
	\node [bullet] (v3) at (0,0.75) {};
	\node [inner sep=0pt] (v2) at (0,0) {$\times$};
	\node [bullet] (v5) at (0,-0.75) {};
	\node [circ] (v6) at (0,-1.5) {};
	\node [bullet] (v7) at (-0.75,0) {};
	\draw [->-,shorten <=-3.5pt] (v1) to (v7);
	\draw [->-,shorten <=-3.5pt] (v2) to (v7);
	\draw [->-] (0,0.75) arc (90:-90:0.75);
	\draw [->-] (0,0.75) arc (90:180:0.75);
	\draw [->-] (-0.75,0) arc (180:270:0.75);
	\draw [->-] (v5) to (v6);
	\end{tikzpicture} +
	\begin{tikzpicture}[auto,baseline={([yshift=-.5ex]current bounding box.center)}]
	\node [inner sep=0pt] (v1) at (-1.5,0) {$\times$};
	\node [bullet] (v3) at (0,0.75) {};
	\node [inner sep=0pt] (v2) at (0,0) {$\times$};
	\node [bullet] (v5) at (0,-0.75) {};
	\node [circ] (v6) at (0,-1.5) {};
	\node [bullet] (v7) at (-0.75,0) {};
	\draw [->-,shorten <=-3.5pt] (v1) to (v7);
	\draw [->-,shorten <=-3.5pt] (v2) to (v3);
	\draw [->-] (0,0.75) arc (90:-90:0.75);
	\draw [->-] (0,0.75) arc (90:180:0.75);
	\draw [->-] (-0.75,0) arc (180:270:0.75);
	\draw [->-] (v5) to (v6);
	\end{tikzpicture} +
	\begin{tikzpicture}[auto,baseline={([yshift=-.5ex]current bounding box.center)}]
	\node [inner sep=0pt] (v1) at (-1.5,0) {$\times$};
	\node [bullet] (v3) at (0,0.75) {};
	\node [inner sep=0pt] (v2) at (0,0) {$\times$};
	\node [bullet] (v5) at (0,-0.75) {};
	\node [circ] (v6) at (0,-1.5) {};
	\node [bullet] (v7) at (-0.75,0) {};
	\draw [->-,shorten <=-3.5pt] (v1) to (v7);
	\draw [->-,shorten <=-3.5pt] (v2) to (v5);
	\draw [->-] (0,0.75) arc (90:-90:0.75);
	\draw [->-] (0,0.75) arc (90:180:0.75);
	\draw [->-] (-0.75,0) arc (180:270:0.75);
	\draw [->-] (v5) to (v6);
	\end{tikzpicture} \\
	C_3 &= \begin{tikzpicture}[auto,baseline={([yshift=-.5ex]current bounding box.center)}]
	\node [inner sep=0pt] (v1) at (-1.5,0) {$\times$};
	\node [bullet] (v4) at (0.75,0) {};
	\node [bullet] (v3) at (0,0.75) {};
	\node [inner sep=0pt] (v2) at (0,0) {$\times$};
	\node [bullet] (v5) at (0,-0.75) {};
	\node [circ] (v6) at (0,-1.5) {};
	\node [bullet] (v7) at (-0.75,0) {};
	\node [vertex] (v8) at (0.75,-1.5) {$1$};
	\draw [->-,shorten <=-3.5pt] (v1) to (v7);
	\draw [->-,shorten <=-3.5pt] (v2) to (v4);
	\draw [->-] (0,0.75) arc (90:0:0.75);
	\draw [->-] (0.75,0) arc (0:-90:0.75);
	\draw [->-] (0,0.75) arc (90:180:0.75);
	\draw [->-] (-0.75,0) arc (180:270:0.75);
	\draw [->-] (v5) to (v6);
	\draw [->-] (v8) to (v6);
	\end{tikzpicture}
	\end{align*}
	 
\end{example}

\subsubsection{The determinant line bundle}
Finally, let us discuss the meaning of the integer $d$ in terms of the geometry of the moduli spaces $\cM$. For that we calculate the shift in the degrees for some dimension $d$ and dimension zero. Let us take a marked ribbon quiver $\vec\Gamma$ whose unmarked sources are all of valence $2$; let $N$ be the number of those sources. Calculating the degrees, we find that the difference between the degrees for $d$ and zero is:
\[ \deg_d(\vec\Gamma) - \deg_0(\vec\Gamma) = dN \]

\begin{proposition}
	The number $N$ is equal for all such graphs with a fixed genus and boundary type, and is given by
	\[ N = 2g -2 + |V_\times| + |V_\circ| + O + F + |V_\mathrm{open-out}|  = -\chi(\Sigma_\mathsc{OC}) + |V_\mathrm{open-out}| \]
	where $O$ is again the number of boundary circles with open intervals, $F$ is the number of free boundary circles, and $\chi(\Sigma^\mathsc{OC})$ is the Euler characteristic of the corresponding open-closed surface.
\end{proposition}
\begin{proof}
	The first statement follows from the fact that the space $\cM$ is connected for each fixed genus and boundary type, and the fact that any two graphs with all unmarked sources of valence $2$ are related by contracting and expanding edges.
	
	As for the formula, it is enough to construct one such quiver; we can produce a quiver with $2g$ such sources from a polygon with $4g$ edges, one $\times$-vertex and one $\circ$-vertex. One then sees by construction that each other source in $V_\times$, each other sink and each boundary component with open-ins and -outs contributes another valence 2 unmarked source.
\end{proof}

The proposition above implies that
\[ Q^d = H_c^{\dim\cM-*}(\cM, \cL^{\otimes d}) \]
for a rank one sheaf $\cL$ on $\cM$, given by a line bundle shifted in degree by $-\chi(\Sigma_\mathsc{OC}) + |V_\mathrm{open-out}|$. This agrees with the description given by Costello in \cite{costello2007topological}; the action of the open-closed \textsc{prop} in degree $d$ should be twisted by $\cL = \det$ which is a certain determinant line bundle on the open-closed moduli space. The argument in \cref{sec:isoEvenOdd} can thus be seen as a combinatorial proof that, in the case where there are no free boundary circles, the determinant line bundle is trivial up to shift, a claim which already appears in \emph{op.cit.}

\printbibliography

\end{document}